\documentclass{amsart}

\usepackage{amssymb,amscd,a4wide,graphicx,hyperref,amsbsy,color}
\newtheorem{theorem}{Theorem}[section]
\newtheorem{corollary}[theorem]{Corollary}
\newtheorem{conjecture}[theorem]{Conjecture}
\newtheorem{lemma}[theorem]{Lemma}
\newtheorem{proposition}[theorem]{Proposition}
\theoremstyle{definition}
\newtheorem{definition}[theorem]{Definition}
\theoremstyle{remark}
\newtheorem{remark}[theorem]{Remark}

\newcommand{\tr}[1]{\vphantom{#1}^t #1}

\makeatletter
\newcommand{\labitem}[2]{\def\@itemlabel{#1} \item \def\@currentlabel{#1}\label{#2}}
\makeatother

\numberwithin{equation}{section}

\title{Geometry, dynamics, and arithmetic of $S$-adic shifts$^1$}

\author[V.~Berth\'e]{Val\'erie Berth\'e}
\author[W.~Steiner]{Wolfgang Steiner}
\address{IRIF, CNRS UMR 8243, Universit\'e Paris Diderot -- Paris 7, Case 7014, 75205 Paris Cedex 13, FRANCE}
\email{berthe@irif.fr, steiner@irif.fr}
\author[J. M. Thuswaldner]{J\"org M. Thuswaldner}
\address{Chair of Mathematics and Statistics, University of Leoben, A-8700 Leoben, AUSTRIA}
\email{joerg.thuswaldner@unileoben.ac.at}
\thanks{This work was supported by the Agence Nationale de la Recherche and the Austrian Science Fund through the projects ``Fractals and Numeration'' (ANR-12-IS01-0002, FWF I1136), ``Discrete Mathematics'' (FWF W1230) and ``Dyna3S'' (ANR-13-BS02-0003).}
\thanks{$^1$After this paper was published in {\it Ann.\ Inst.\ Fourier (Grenoble)} {\bf 69} (2019) 1347--1409, we observed that Theorem~\ref{t:3} holds in a more general setting. The present manuscript contains this more general version of Theorem~\ref{t:3}. To prove it we needed to change the last part of the proof of Theorem~\ref{t:3} (see p.~\pageref{proof:3.3}).}

\date{\today}



\begin{document}

\begin{abstract}
This paper studies geometric and spectral properties of $S$-adic shifts and their relation to continued fraction algorithms. These  shifts are symbolic dynamical systems obtained by iterating infinitely many substitutions. Pure discrete spectrum for $S$-adic shifts and tiling properties of associated Rauzy fractals are established  under  a  generalized  Pisot  assumption together with  a geometric coincidence condition. These general results extend the scope of the Pisot substitution conjecture to the $S$-adic framework. They are applied to families of $S$-adic shifts generated by Arnoux-Rauzy as well as Brun substitutions. It is shown that almost all of these shifts have pure discrete spectrum. Using $S$-adic words related to Brun's continued fraction algorithm, we exhibit bounded remainder sets and natural codings for almost all translations on the two-dimensional torus.  Due to the lack of self-similarity properties present for substitutive systems we  have to develop new proofs to obtain our results in the $S$-adic setting. 
\end{abstract}

\maketitle

\section{Introduction}

Substitutive dynamical systems for substitutions with dominant Pisot eigenvalue are widely known to yield pure discrete spectrum in the symbolic setting as well as for tiling spaces, cf.\  \cite{Rauzy:82,Fog02,Barge-Kwapisz:06,CANTBST,AkiBBLS}. The aim of this paper is to extend substitutive dynamical systems to the non-stationary (i.e., time inhomogeneous) framework. The iteration of a single transformation is replaced by a sequence of transformations, along a sequence of spaces; see e.g.\ \cite{Arnoux-Fisher:01,Arnoux-Fisher:05,Fisher:09} for sequences of substitutions and Anosov maps as well as for relations to Vershik's adic systems. In this setting, the Pisot condition is replaced by the requirement that the second Lyapunov exponent of the  dynamical system is negative, leading to hyperbolic dynamics with a one-dimensional unstable foliation. This requirement has an arithmetical meaning, as it assures  a.e.\ strong convergence of continued fraction algorithms associated with these dynamical systems; see \cite{SCHWEIGER,Berthe:11,Berthe-Delecroix,AD13}.
  
We consider \emph{$S$-adic} symbolic dynamical systems, where the letter~$S$ refers to ``substitution''.  
These shift spaces are obtained by iterating different substitutions in a prescribed order, generalizing the substitutive case where a single substitution is iterated. 
An $S$-adic expansion of an infinite word~$\omega$ is given by a sequence $(\sigma_n,i_n)_{n \in \mathbb{N}}$, where the $\sigma_n$ are substitutions and the $i_n$ are letters, such that $\omega = \lim_{n\to\infty} \sigma_0 \sigma_1 \cdots \sigma_{n-1}(i_n)$. 
Under mild assumptions (needed in order to exclude degenerate constructions), the orbit closure under the action of the shift~$\Sigma$ on the infinite word~$\omega$  
is a minimal symbolic dynamical system equipped with an $S$-adic substitutive structure, and has zero entropy~\cite{Berthe-Delecroix}.
The $S$-adic shifts are closely related to Vershik's adic systems~\cite{Vershik:81,sadic4}, which have  provided the  terminology ``$S$-adic''. More generally they belong to the family of fusion systems (see  \cite{PriebeFrank-Sadun11,PriebeFrank-Sadun14}), which also includes Bratteli-Vershik systems and multidimensional cut-and-stack transformations, and pertain to arithmetic dynamics~\cite{Sidorov}. The connections with continued fractions are natural in  this framework: they had big influence on the set-up of the $S$-adic formalism, inspired by the Sturmian dynamics which is thoroughly described by regular continued fractions; see e.g.\ \cite{Arnoux-Fisher:01,Berthe-Ferenczi-Zamboni:05}.

In the classical Pisot substitutive setting, the basic object is a single Pisot substitution, i.e., a substitution~$\sigma$ whose incidence matrix~$M_{\sigma}$ has a Pisot number as dominant eigenvalue. When the characteristic polynomial of~$M_{\sigma}$ is furthermore assumed to be irreducible, then  the associated  symbolic dynamical system $(X_{\sigma}, \Sigma)$ is conjectured to have pure discrete spectrum. This is the {\em Pisot substitution conjecture}. For more details and partial results on this conjecture, see \cite{Fog02,ST09,CANTBST,AkiBBLS}. One now classical approach for exhibiting the translation on a compact abelian group to which $(X_{\sigma}, \Sigma)$ is conjectured to be isomorphic relies on the associated {\it Rauzy fractal}. 
\begin{figure}[h]
\includegraphics[trim=0 200 0 200,width=0.35\textwidth,angle=235]{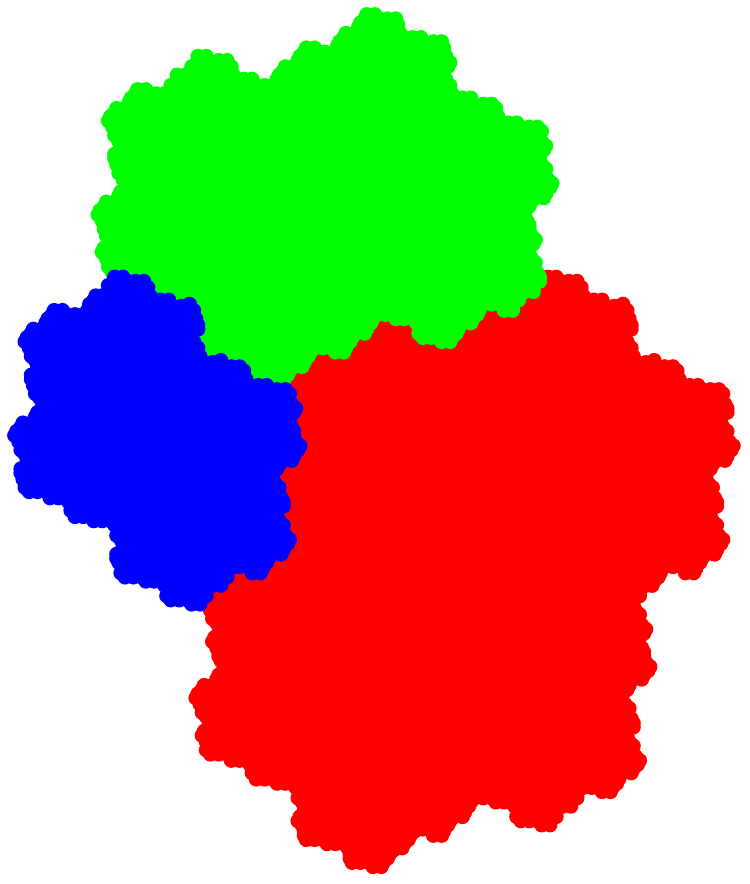}
\caption{The classical Rauzy fractal. \label{fig:1}}
\end{figure}
This set, which has fractal boundary in most cases, can be viewed as the solution of a graph-directed iterated function system, which allows to derive many of its geometric and topological properties, see~\cite{SirventWang02}. Moreover, it forms a fundamental domain for the $\mathbb{Z}$-action provided by the Kronecker group translation (or at least for a Kronecker factor).  An image of the classical Rauzy fractal going back to \cite{Rauzy:82} is depicted in Figure~\ref{fig:1}. This Rauzy fractal is associated with the \emph{Tribonacci substitution} defined on the alphabet $\mathcal{A}=\{1,2,3\}$ by $\sigma(1)=12$, $\sigma(2)=13$, and $\sigma(3)=1$.

We extend classical notions, results, and problems studied in the Pisot substitutive case to the $S$-adic framework. We are able to define Rauzy fractals associated with $S$-adic symbolic dynamical systems, with the Pisot assumption being extended to the $S$-adic framework by requiring the second Lyapunov exponent to be negative. In other words, we work with $S$-adic shifts whose associated cocyles (provided by the incidence matrices of the substitutions) display strong convergence properties
analogous to the Pisot case. Combinatorially, this reflects in certain balancedness properties of the associated language. This also  allows us to define analogs of the  stable/unstable splitting in the Pisot substitution case. In order to prove discrete spectrum, we associate with any Pisot $S$-adic shift a Rauzy fractal that lives in 
the analog of the stable space. 

We then introduce a family of coverings and multiple tilings, including periodic and aperiodic ones, that comes together with set equations playing the role of the graph-directed iterated function system in the Pisot substitutive case.
A~particular choice of a periodic tiling yields number theoretic applications and the isomorphism with a toral translation, whereas other (aperiodic) choices allow the  study of the associated coverings. We then express a criterion for the multiple tilings  to be indeed tilings, which yields pure discrete spectrum.
This criterion is a coincidence type condition in the same vein as the various coincidence conditions (algebraic, combinatorial, overlap, etc.) introduced in the substitutive framework (first in~\cite{Dekking} for substitutions of constant length and then extended to the most general substitutive framework, see e.g.\ \cite{Solomyak:97,AL11}). 
 
The idea of constructing Rauzy fractals associated with multidimensional continued fractions is already present in \cite{Ito:89,Ito:95b},  but the problem remained to prove tiling properties, and even the question whether subpieces of the Rauzy fractal do not overlap could not be answered. Furthermore, although there exist results for the generation of discrete hyperplanes in connection with continued fraction algorithms \cite{Ito-Ohtsuki:93,Ito-Ohtsuki:94,Arnoux-Berthe-Ito:02,BBJS13,BBJS14},  more information on convergence and renormalization  properties is needed in order to deduce spectral properties. In~\cite{Arnoux-Mizutani-Sellami}, $S$-adic sequences are considered  where the substitutions all have the same Pisot irreducible unimodular matrix; in our case, the matrices are allowed to be different at each step.

\subsection*{Main results}
In our first result  (see Theorem~\ref{t:1}) we describe geometric and dynamical properties of an $S$-adic shift $(X,\Sigma)$ under very general combinatorial conditions. In particular, we are able to associate Rauzy fractals with $(X,\Sigma)$ that are compact, the closure of their interior, and have a boundary of zero measure. We deduce covering and (multiple) tiling properties of these Rauzy fractals and, subject to a combinatorial condition (a coincidence type condition), we are able to show that they form a periodic tiling. Due to the lack of a dominant eigenvector and the fact that we lose the self-similarity properties present for substitutive systems, these proofs require new ideas and do not run along the lines of the substitutive setting.
In particular, a crucial point is  to prove that the  boundary of  Rauzy fractals   has measure zero (see Proposition \ref{p:boundary}). The tiling property of the Rauzy fractals is then used to prove that $(X,\Sigma)$ is  measurably conjugate to a translation on a torus of suitable dimension. In this case, the subpieces of the Rauzy fractal turn out to be bounded remainder sets, and the elements of~$X$ are natural codings for this translation (as defined in Section~\ref{sec:coding}). Since the assumptions on the shift are very mild, a slight variation of this result can be used to establish a metric result stating that almost all shifts of certain families of $S$-adic shifts (under the above-mentioned Pisot condition in terms of Lyapunov exponents) have the above properties. We apply these constructions to two multidimensional continued fraction algorithms, the Arnoux-Rauzy and the Brun algorithm, that are proved to satisfy our Pisot assumptions as well as the combinatorial coincidence  condition.

Arnoux-Rauzy substitutions are known to be Pisot~\cite{Arnoux-Ito:01}. Purely substitutive Arnoux-Rauzy words are even natural codings of toral translations \cite{Berthe-Jolivet-Siegel:12,Barge-Stimac-Williams:13}. This is not true for arbitrary non-substitutive Arnoux-Rauzy words (see \cite{Cassaigne-Ferenczi-Zamboni:00,Cassaigne-Ferenczi-Messaoudi:08}), but we are able to show this property for large classes of them; to our knowledge, no such examples (on more than 2 letters) were known before. Moreover, we deduce from a recent result by Avila and Delecroix~\cite{AD13} that  almost every Arnoux-Rauzy shift (w.r.t.\ a measure on the sequences of substitutions generating these words) is a natural coding of a toral translation. This proves a conjecture of Arnoux and Rauzy that goes back to the early nineties (see e.g.\ \cite{Cassaigne-Ferenczi-Zamboni:00,Berthe-Ferenczi-Zamboni:05}) in a metric sense.
We also prove that any linearly recurrent Arnoux-Rauzy shift with recurrent directive sequence has pure discrete spectrum.

Brun's algorithm~\cite{BRUN} is one of the most classical multidimensional generalizations of the regular continued fraction expansion \cite{BRENTJES,SCHWEIGER}. This algorithm generates a sequence of simultaneous rational approximations to a given pair of points (each of these approximations is a pair of points having the same denominator). It is also closely related to the modified Jacobi-Perron algorithm introduced by Podsypanin in~\cite{POD77}, which is a two-point extension of the Brun algorithm. It is shown to be strongly convergent almost everywhere with exponential rate \cite{FUKE96,Schratzberger:98,Meester,Broise} and has an invariant ergodic  probability measure equivalent to the Lebesgue measure which is known explicitly~\cite{ArnouxNogueira93}. The substitutive case has been  handled in \cite{Barge14,BBJS14}: Brun substitutions  have pure discrete spectrum. Applying our theory, we prove that for almost all $(x_1,x_2) \in [0,1)^2$, there is an $S$-adic shift associated with a certain (explicitly given) Brun expansion which is measurably conjugate to the translation by $(x_1,x_2)$ on the torus~$\mathbb{T}^2$. This implies that Brun substitutions yield natural codings of almost all rotations on the two-dimensional torus. The subpieces of the associated Rauzy fractals provide  (measurable) bounded remainder sets for this rotation. 

\subsection*{Motivation}
Our motivation comes on the one hand from number theory. Indeed, Rauzy fractals are known to provide fundamental domains for Kronecker translations on the torus~$\mathbb{T}^d$ (together with Markov partitions for the corresponding toral automorphisms). They are also used to obtain best approximation results for cubic fields~\cite{Hubert-Messaoudi:06}, and serve as limit sets for simultaneous Diophantine approximation for cubic extensions in terms of self-similar ellipses provided by Brun's algorithm~\cite{ito:03,Ito:07}. Using our new theory, it is now possible to reach Kronecker translations with non-algebraic parameters, which extends the usual (Pisot) algebraic framework and the scope of potential number theoretic applications considerably.
 
On the other hand, the results of the present paper extend discrete spectrum results to a much wider framework. Furthermore, our theory enables us to give explicit constructions for higher dimensional non-stationary Markov partitions  for ``non-stationary hyperbolic toral automorphisms'', in the sense of~\cite{Arnoux-Fisher:05}. In~\cite{Arnoux-Fisher:05} such non-stationary Markov partitions are defined and $2$-dimen\-sio\-nal examples are given. Our new results (including the tilings by $S$-adic Rauzy fractals) will help in the quest for a convenient symbolic representation of the Weyl chamber flow; see e.g.~\cite[Section~6]{Gorodnik}. In the case of two letters this is performed in~\cite{Arnoux-Fisher:01}, the general case is the subject of the forthcoming paper~\cite{sadic3}. Another direction of research will be the investigation of the spectrum of $S$-adic shifts. In particular, it would be interesting to explore how Host's result~\cite{Host:86} on the continuity of eigenvalues extends to this more general setting. We will come back to these subjects in a forthcoming paper.

\section{Basic definitions and notation}\label{sec:miseenscene}

\subsection{Substitutions}
A~\emph{substitution} $\sigma$ over a finite alphabet $\mathcal{A} = \{1,2,\ldots,d\}$ is an endomorphism of the free monoid~$\mathcal{A}^*$ (that is endowed with the operation of concatenation). We assume here  that all our substitutions are non-erasing, i.e., they send non-empty words to non-empty words.
The \emph{incidence matrix} (or abelianization) of~$\sigma$ is the square matrix $M_\sigma = (|\sigma(j)|_i)_{i,j\in\mathcal{A}} \in \mathbb{N}^{d\times d}$.
Here, the notation $|w|_i$ stands for the number of occurrences of the letter~$i$ in $w \in \mathcal{A}^*$, and $|w|$ will denote the length of~$w$. 
We say that $\sigma$ is \emph{unimodular} if $|\!\det M_\sigma| = 1$.
The map
\[
\mathbf{l}:\ \mathcal{A}^* \to\mathbb{N}^d, \ w \mapsto \tr{(|w|_{1},|w|_2,\ldots, |w|_{d})}
\]
is called the \emph{abelianization map}. 
Note that $\mathbf{l}(\sigma(w)) = M_\sigma \mathbf{l}(w)$ for all $w\in \mathcal{A}^*$.
A~substitution is called \emph{Pisot irreducible} if the characteristic polynomial of its incidence matrix is the minimal polynomial of a Pisot number.

\subsection{$S$-adic words and languages}
Let $\boldsymbol{\sigma} = (\sigma_n)_{n\in\mathbb{N}}$ be a sequence of substitutions over the alphabet~$\mathcal{A}$.
To keep notation concise, we set $M_n = M_{\sigma_n}$ for $n \in \mathbb{N}$, and we abbreviate products of consecutive substitutions and their incidence matrices by
\[
\sigma_{[k,\ell)} = \sigma_k \sigma_{k+1} \cdots \sigma_{\ell-1} \quad \mbox{and} \quad M_{[k,\ell)} = M_k M_{k+1} \cdots M_{\ell-1} \quad (0 \le k \le \ell).
\]

The language associated with~$\boldsymbol{\sigma}$ is defined by $\mathcal{L}_{\boldsymbol{\sigma}} = \mathcal{L}_{\boldsymbol{\sigma}}^{(0)}$, where
\[
\mathcal{L}_{\boldsymbol{\sigma}}^{(m)} = \big\{w \in \mathcal{A}^*:\, \mbox{$w$ is a factor of $\sigma_{[m,n)}(i)$ for some $i \in\mathcal{A}$, $n\in\mathbb{N}$}\big\} \qquad (m \in \mathbb{N}).
\]
Here, $w$~is a \emph{factor} of $v \in \mathcal{A}^*$ if $v \in \mathcal{A}^* w \mathcal{A}^*$. 
Furthermore, $w$~is a \emph{prefix} of~$v$ if $v \in w \mathcal{A}^*$.
Similarly, $w$~is a factor and a prefix of an infinite word $\omega \in \mathcal{A}^\mathbb{N}$ if $\omega \in \mathcal{A}^* w \mathcal{A}^\mathbb{N}$ and $\omega \in w \mathcal{A}^\mathbb{N}$, respectively. 

The sequence~$\boldsymbol{\sigma}$ is said to be \emph{algebraically irreducible} if, for each $k \in \mathbb{N}$, the characteristic polynomial of $M_{[k,\ell)}$ is irreducible for all sufficiently large~$\ell$.
The sequence~$\boldsymbol{\sigma}$ is said to be \emph{primitive} if, for each $k \in \mathbb{N}$, $M_{[k,\ell)}$ is a positive matrix for \emph{some} $\ell > k$.
This notion extends primitivity of a single substitution~$\sigma$, where $M_\sigma^\ell$ is required to be positive for some $\ell > 0$, to sequences. Note that \cite{Durand:00a,Durand:00b,Durand-Leroy-Richomme:13} use a more restrictive definition of primitive sequences of substitutions.

Following \cite{Arnoux-Mizutani-Sellami}, we say that an infinite word $\omega \in \mathcal{A}^\mathbb{N}$ is a \emph{limit word} of $\boldsymbol{\sigma} = (\sigma_n)_{n\in\mathbb{N}}$ if there is a sequence of infinite words $(\omega^{(n)})_{n\in\mathbb{N}}$ with 
\[
\omega^{(0)} = \omega, \quad \omega^{(n)} = \sigma_n\big(\omega^{(n+1)}\big) \quad \mbox{for all}\ n \in \mathbb{N},
\]
where the substitutions~$\sigma_n$ are naturally extended to infinite words.
We also say that $\omega$ is an \emph{$S$-adic limit word} with \emph{directive sequence} $\boldsymbol{\sigma}$ and $S = \{\sigma_n:\, n \in \mathbb{N}\}$.  
We can write 
\[
\omega = \lim_{n\to\infty} \sigma_{[0,n)}(i_n),
\] 
where $i_n$ denotes the first letter of~$\omega^{(n)}$, provided that $\lim_{n\to\infty} |\sigma_{[0,n)}(i_n)| = \infty$ (which holds in particular when $\boldsymbol{\sigma}$ is primitive). 
In case that~$\boldsymbol{\sigma}$ is a periodic sequence, there exists a limit word~$\omega$ such that $\omega^{(n)} = \omega$ for some $n \ge 1$, i.e., $\omega$~is the fixed point of the substitution~$\sigma_{[0,n)}$. 
We will refer to this case as the \emph{periodic case}. 

Note that we do not require $S$ to be finite since we want to include $S$-adic shifts issued from (multiplicative) multidimensional continued fraction expansions.
For more on $S$-adic sequences, see e.g.\ \cite{Berthe-Delecroix,Durand-Leroy-Richomme:13,Arnoux-Mizutani-Sellami}.

\subsection{Symbolic dynamics, $S$-adic shifts, and $S$-adic graphs}
An infinite word~$\omega$ is said to be \emph{recurrent} if each factor of~$\omega$ occurs infinitely often in~$\omega$.
It is is said to be \emph{uniformly recurrent} if each factor occurs at an infinite number of positions with bounded gaps. The recurrence function $R(n)$ of a uniformly recurrent word~$\omega$ is defined for any~$n$ as the smallest positive integer~$k$ for which every factor of size~$k$ of~$\omega$ contains every factor of size~$n$. An infinite word~$\omega$ is said to be \emph{linearly recurrent} if there exists a constant~$C$ such that $R(n) \leq Cn$, for all~$n$. 

The \emph{shift operator}~$\Sigma$ maps $(\omega_n)_{n\in\mathbb{N}}$ to  $(\omega_{n+1})_{n\in\mathbb{N}}$.
A~dynamical system $(X,\Sigma)$ is a \emph{shift space} if $X$ is a closed shift invariant set of infinite words over a finite alphabet, with the product topology of the discrete topology. 
The system $(X,\Sigma)$ is \emph{minimal} if every non-empty closed shift invariant subset equals the whole set; it is called \emph{uniquely ergodic} if there exists a unique shift invariant probability measure on~$X$.
The \emph{symbolic dynamical system generated by an infinite word~$\omega$} is defined as~$(X_\omega,\Sigma)$, where $X_\omega = \overline{\{\Sigma^n(\omega) :\, n \in \mathbb{N}\}}$ is the closure of the $\Sigma$-orbit of~$\omega$.
This system is \emph{minimal} if and only if $\omega$ is uniformly recurrent \cite[Proposition~4.7]{Queffelec:10}. 

Let $\mu$ be a shift invariant measure defined on $(X,\Sigma)$. A~measurable eigenfunction of the system $(X,\Sigma,\mu)$ with associated eigenvalue
$\lambda \in \mathbb{R}$ is an $L^2(X,\mu)$ function that satisfies $f(\Sigma^n(\omega)) = e^{2\pi i\lambda n} f(\omega)$ for all $n \in \mathbb{N}$ and $\omega \in X$. The system $(X,\Sigma)$ is said to be \emph{weakly mixing} if there are no non-trivial measurable eigenvalues. It has \emph{pure discrete spectrum} if $L^2(X,\mu)$ is spanned by the measurable eigenfunctions. 

In the present paper we consider two types of symbolic dynamical systems in which the previous definitions make sense. Namely,  \emph{$S$-adic shifts} and  edge shifts associated with \emph{$S$-adic graphs}.   They will be defined now.

Let $S$ be a  set of substitutions, $S$ can be finite or infinite. The \emph{$S$-adic shift} or \emph{$S$-adic system} with directive sequence~$\boldsymbol{\sigma}\in S^{\mathbb{N}}$ is $(X_{\boldsymbol{\sigma}}, \Sigma)$, where $X_{\boldsymbol{\sigma}}$ denotes the set of infinite words~$\omega$ such that each factor of~$\omega$ is an element of~$\mathcal{L}_{\boldsymbol{\sigma}}$.
If $\boldsymbol{\sigma}$ is primitive, then one checks that $(X_{\boldsymbol{\sigma}}, \Sigma) = (X_\omega,\Sigma)$ for any limit word~$\omega$ of~$\boldsymbol{\sigma}$; see e.g. \cite[Theorem 5.2]{Berthe-Delecroix}.

Let $S$ be  a set of substitutions and let $G=(V,E)$ be a strongly connected directed graph with set of vertices $V$ and set of edges $E$.  The graph $G$ may be an infinite graph with multiple edges. Let $\tau: E \to S$ be a map that associates a substitution $\tau(e)=\sigma_e$ with each edge $e\in E$ and call $(G,\tau)$ an \emph{$S$-adic graph} (see~\cite[Section~3.3]{Berthe-Delecroix}).  Let $s(e)$ and $r(e)$ be the source and the range of an edge $e\in E$. Then with $G$  we associate the edge shift $(E_G,\Sigma)$ with
\[
E_G = \{ 
(\gamma_n) \in E^\mathbb{N} \;:\;r(\gamma_n)=s(\gamma_{n+1}) \hbox{ for each }n\in \mathbb{N} 
\}.
\]
The set $D_G = \{ 
(\sigma_n)=(\tau(\gamma_n)) \;:\; (\gamma_n) \in E_G$
consists of all directive sequences corresponding to labelings of infinite walks in $G$. 
If $V$ is finite, then we will speak of an $S$-adic graph with finitely many vertices; in this case, the fact that $E$ may be infinite (which allows $S$-adic shifts with infinitely many different substitutions) constitutes the difference between the edge shift $E_G$ and a shift of finite type or a Markov system.
It is important to allow $E$ to be infinite because to define classes of $S$-adic words related to multiplicative continued fraction algorithms requires infinitely many substitutions and, hence, infinitely many edges in the $S$-adic graph (a prominent example in this context is the Jacobi-Perron algorithm, see e.g.~\cite[Section~2.3]{BBJS14}). Very often we will deal with the full shift $E_G\cong D_G=S^\mathbb{N}$ corresponding to a graph with one vertex and finitely or countably infinitely many self loops (each of which is identified with a different substitution). Note that we use the same notation for the shift map $\Sigma$ acting on~$\mathcal{A}^\mathbb{N}$ and on~$E_G$. The \emph{cylinder} of a finite sequence $(\gamma_0, \gamma_1, \ldots, \gamma_{\ell-1}) \in E^\ell$ is 
\[
\mathcal{Z}(\gamma_0, \gamma_1, \ldots, \gamma_{\ell-1}) = \big\{(\tau_n)_{n\in\mathbb{N}} \in E_G:\, (\tau_0, \tau_1, \ldots, \tau_{\ell-1}) = (\gamma_0, \gamma_1, \ldots, \gamma_{\ell-1})\big\}.
\]
In what follows we will often identify an edge $e\in E$ with the substitution $\tau(e)$, a finite walk with the associated product of substitutions, and an element of $(\gamma_n)\in E_G$ with the associated directive sequence $(\tau(\gamma_n))$. In particular, we will talk about the incidence matrix of an edge or a walk, and about primitivity, algebraic irreducibility, and the language $\mathcal{L}_{\boldsymbol{\gamma}}$ of an element $\boldsymbol{\gamma}\in E_G$.

\subsection{Balance and letter frequencies} \label{subsec:letterfreq}
A~pair of words $u, v \in \mathcal{A}^*$ with $|u| = |v|$ is \emph{$C$-balanced} if 
\[
-C \le |u|_j - |v|_j \le C \quad \mbox{for all}\ j \in \mathcal{A}.
\]
A~language~$\mathcal{L}$ is $C$-balanced if each pair of words $u, v \in \mathcal{L}$ with $|u| = |v|$ is $C$-balanced.
The language~$\mathcal{L}$ is said to be \emph{balanced} if there exists~$C$ such that $\mathcal{L}$ is $C$-balanced. (In previous works, this property was sometimes called \emph{finitely balanced}, and balancedness referred to the case  $C = 1$.)
A~(finite or infinite) word is $C$-balanced or balanced if the language of its factors has this property.

Note that the language  of a Pisot irreducible substitution is balanced; see e.g.~\cite{Adamczewski:03,Adamdis}.

The \emph{frequency} of a letter $i \in \mathcal{A}$ in $\omega \in \mathcal{A}^\mathbb{N}$ is defined as $f_i = \lim_{|p|\to\infty} |p|_i/|p|$, where the limit is taken over the prefixes~$p$ of~$\omega$, if the limit exists. 
The vector $\tr{(f_1,f_2,\ldots,f_d)}$ is then called the \emph{letter frequency vector}.
Balancedness implies the existence of letter frequencies; see~\cite{Berthe-Tijdeman:02}.

\subsection{Generalized Perron-Frobenius eigenvectors} \label{sec:gener-perr-frob}
A~natural way to endow a shift space with a shift invariant measure is to consider its factor frequencies (defined analogously as for letters). In the primitive substitutive case, letter frequencies are given by the Perron-Frobenius eigenvector.
More generally, for a sequence of matrices $(M_n)_{n\in\mathbb{N}}$, we have by \cite[pp.~91--95]{Furstenberg:60} that
\begin{equation} \label{e:topPF}
\bigcap_{n\in\mathbb{N}} M_{[0,n)}\, \mathbb{R}^d_+ = \mathbb{R}_+ \mathbf{u} \quad \mbox{for some positive vector}\ \mathbf{u} \in \mathbb{R}_+^d,
\end{equation}
provided there are indices $k_1 < \ell_1 \le k_2 < \ell_2 \le \cdots$ and a positive matrix~$B$ such that $B = M_{[k_1,\ell_1)} = M_{[k_2,\ell_2)} = \cdots$.
In particular, \eqref{e:topPF} holds for the sequence of incidence matrices of a  primitive and recurrent sequence of substitutions $\boldsymbol{\sigma} = (\sigma_n)_{n\in\mathbb{N}}$ (even if $S$ is infinite).
We call~$\mathbf{u}$ a \emph{generalized right eigenvector} of~$\boldsymbol{\sigma}$.
Note that \eqref{e:topPF} is called \emph{topological Perron-Frobenius condition} in~\cite{Fisher:09}.
In particular, the letter frequency vector $\mathbf{u}=\tr{(f_1,f_2,\ldots,f_d)}$ is a generalized right eigenvector when $\omega$ is a limit word of a primitive and recurrent sequence of substitutions.

\subsection{Lyapunov exponents and Pisot condition} \label{sec:lyap-expon-pisot}
Let $S$ be a finite or  infinite set of substitutions with invertible incidence matrices and let $(G,\tau)$ be an $S$-adic graph with associated edge shift $(E_G, \Sigma, \nu)$, where $\nu$ is a (ergodic) probability measure. With each $\boldsymbol{\gamma} = (\gamma_n)_{n\in\mathbb{N}} \in E_G$, associate the linear \emph{cocycle} operator $A(\boldsymbol{\gamma}) = \tr{\!M}_0$ (where $M_0$ is the incidence matrix of the substitution~$\sigma_0=\tau(\gamma_0)$ associated with the first edge $\gamma_0$ of the walk $\boldsymbol{\gamma}$). Assume that this cocycle is \emph{log-integrable} in the sense that
\[
\int_{E_G} \log\max\{     \|A(x) \|,  \|A(x)^{-1} \| \} d\nu(x) < \infty
\]
(this condition is always satisfied if $G$ is a finite graph, that is, when  $G$  has finitely many edges). Then the \emph{Lyapunov exponents} $\theta_1, \theta_2, \ldots, \theta_d$ of $(E_G, \Sigma, \nu)$ are recursively defined by 
\begin{align}
\theta_1 + \theta_2 + \cdots + \theta_k & = \lim_{n\to\infty} \frac{1}{n} \int_{E_G} \log \|\wedge^k \big(A(\Sigma^{n-1}(x)) \cdots A(\Sigma(x)) A(x)\big)\|\, d\nu(x) \nonumber \\
& = \lim_{n\to\infty} \frac{1}{n} \int_{E_G} \log \|\wedge^k (\tr{M}_{[0,n)})\|\, d\nu = \lim_{n\to\infty} \frac{1}{n} \int_{E_G} \log \|\wedge^k M_{[0,n)}\|\, d\nu \label{eq:transposeequal}
\end{align}
for $1 \le k \le d$, where $\wedge^k$ denotes the $k$-fold wedge product.
Here and in the following, $\|\cdot\|$ denotes the maximum norm~$\|\cdot\|_\infty$. 
Following \cite[\S 6.3]{Berthe-Delecroix}, we say that $(E_G, \Sigma, \nu)$ satisfies the \emph{Pisot condition} if 
\[
\theta_1 > 0 > \theta_2 \ge \theta_3 \ge \cdots \ge \theta_d.
\]

\subsection{Natural codings and bounded remainder sets}\label{sec:coding}
Let $\Lambda$ be a full-rank lattice in~$\mathbb{R}^d$ and $T_\mathbf{t}: \mathbb{R}^d/\Lambda \to \mathbb{R}^d/\Lambda$, $\mathbf{x} \mapsto \mathbf{x} + \mathbf{t}$ a given toral translation. Let $R\subset\mathbb{R}^d$ be a fundamental domain for $\Lambda$ and $\tilde T_\mathbf{t}:R\to R$ the mapping induced by $T_\mathbf{t}$ on $R$. If $R = R_1 \cup \cdots \cup R_k$ is a partition of $R$ 
(up to measure zero) such that for each $1\le i\le k$ the restriction~$\tilde T_\mathbf{t}|_{R_i}$  is given by the translation ${\mathbf x}\mapsto{\mathbf x}+{\mathbf t}_i$ for some~${\mathbf t}_i\in\mathbb{R}^d$, and $\omega$ is the coding of a point $\mathbf{x} \in R$ with respect to this partition, we call~$\omega$ a \emph{natural coding} of $T_\mathbf{t}$.
A~symbolic dynamical system $(X,\Sigma)$ is a \emph{natural coding} of $(\mathbb{R}^d/\Lambda, T_\mathbf{t})$ if $(X,\Sigma)$ and $(\mathbb{R}^d/\Lambda, T_\mathbf{t})$ are measurably conjugate and every element of~$X$ is a natural coding of the orbit of some point of the $d$-dimensional torus $\mathbb{R}^d/\Lambda$ (with respect to some fixed partition).

A~subset~$A$ of $\mathbb{R}^d/\Lambda$ with Lebesgue measure~$\lambda(A)$ is said to be a~\emph{bounded remainder set} for the translation~$T_\mathbf{t}$ if there exists $C > 0$ such that, for  a.e.\ $x \in  \mathbb{R}^d/\Lambda$,
\[
|\#\{n  < N:\, T _\mathbf{t}^n(x) \in A \} -  N  \lambda (A)/\lambda(R) |  <  C \qquad \mbox{for all}\ N \in \mathbb{N}.
\]

Observe that if $(X,\Sigma)$ is a natural coding of a minimal translation $(\mathbb{R}^d/\Lambda, T _\mathbf{t})$ with balanced language, then the elements of its associated partition are bounded remainder sets \cite[Proposition~7]{Adamczewski:03}. Moreover, $A$~is a bounded remainder set if it is an atom of a partition that gives rise to a natural coding of a translation whose induced mapping on~$A$ is again a (particular) translation; see \cite{Rauzy:84} (we also refer to \cite{Ferenczi92} for an analogous characterization of bounded remainder sets). 

\subsection{(Multiple) tilings} \label{sec:multiple-tilings}
We call a collection~$\mathcal{K}$ of compact subsets of a Euclidean space~$\mathcal{E}$ a \emph{multiple tiling} of~$\mathcal{E}$ if each element of~$\mathcal{K}$ is the closure of its interior and if there exists a positive integer~$m$ such that almost every point of~$\mathcal{E}$ (with respect to the Lebesgue measure) is contained in exactly $m$ elements of~$\mathcal{K}$. 
The integer~$m$ is called the \emph{covering degree} of the multiple tiling~$\mathcal{K}$. 
If $m=1$, then $\mathcal{K}$ is called a \emph{tiling} of~$\mathcal{E}$.
A~point in~$\mathcal{E}$ is called \emph{$m$-exclusive} if it is contained in the interior of exactly $m$ tiles of~$\mathcal{K}$; it is called \emph{exclusive} if $m = 1$.

\subsection{Rauzy fractals} \label{subsec:FR}
For a vector $\mathbf{w} \in \mathbb{R}^d \setminus \{\mathbf{0}\}$, let 
\[
\mathbf{w}^\bot = \{\mathbf{x} \in \mathbb{R}^d:\, \langle \mathbf{w}, \mathbf{x} \rangle = 0\}
\]
be the hyperplane orthogonal to~$\mathbf{w}$ containing the origin, equipped with the $({d\!-\!1})$-dimensional Lebesgue measure~$\lambda_\mathbf{w}$. 
In particular, for $\mathbf{1} = \tr{(1,\ldots,1)}$, $\mathbf{1}^\bot$~is the hyperplane of vectors whose entries sum up to~$0$.

The \emph{Rauzy fractal} (in the representation space~$\mathbf{w}^\bot$, $\mathbf{w} \in \mathbb{R}_{\ge0}^d \setminus \{\mathbf{0}\}$) associated with a sequence of substitutions $\boldsymbol{\sigma} = (\sigma_n)_{n\in\mathbb{N}}$ over the alphabet~$\mathcal{A}$ with generalized right eigenvector~$\mathbf{u}$ is 
\[
\mathcal{R}_\mathbf{w} = \overline{\{\pi_{\mathbf{u},\mathbf{w}}\, \mathbf{l}(p):\, p \in \mathcal{A}^*,\ \mbox{$p$ is a prefix of a limit word of $\boldsymbol{\sigma}$}\}},
\]
where $\pi_{\mathbf{u},\mathbf{w}}$ denotes the projection along the direction of~$\mathbf{u}$ onto~$\mathbf{w}^\bot$.
The Rauzy fractal has natural \emph{subpieces} (or \emph{subtiles}) defined by
\[
\mathcal{R}_\mathbf{w}(i) = \overline{\{\pi_{\mathbf{u},\mathbf{w}}\, \mathbf{l}(p):\, p \in \mathcal{A}^*,\ \mbox{$p\hspace{.1em}i$ is a prefix of a limit word of $\boldsymbol{\sigma}$}\}},
\]
We set $\mathcal{R} = \mathcal{R}_\mathbf{1}$ and $\mathcal{R}(i) = \mathcal{R}_\mathbf{1}(i)$. 

If $\omega\in\mathcal{A}^\mathbb{N}$ then $\{\mathbf{l}(p):\, \mbox{$p$ is a prefix of $\omega$}\}$ can be regarded as the set of vertex points of the \emph{broken line} corresponding  to $\omega$ (see e.g.\ \cite[Section~5.2.2]{CANTBST}). The Rauzy fractal $\mathcal{R}_\mathbf{w}$ is the closure of the projection of the vertices of all broken lines corresponding to a limit word. When $\boldsymbol{\sigma}$ is a primitive, algebraically irreducible, and recurrent sequence of substitutions with balanced language~$\mathcal{L}_{\boldsymbol{\sigma}}$, then it follows from Proposition~\ref{p:strongconvergence} below that it is sufficient to take a single (arbitrary) limit word in the definition of the Rauzy fractal.  

The \emph{Rauzy boxes} (or suspensions of the Rauzy fractals) are
\[
\widehat{\mathcal{R}}_\mathbf{w}(i) = \big\{x\, (\mathbf{e}_i - \pi_{\mathbf{u},\mathbf{w}}\, \mathbf{e}_i) - \mathbf{y}:\, x \in [0,1),\ \mathbf{y} \in \mathcal{R}_\mathbf{w}(i)\big\},
\]
where $\mathbf{e}_i = \mathbf{l}(i)$ denotes the $i$-th standard unit vector in~$\mathbb{R}^d$.

\subsection{Discrete hyperplanes and collections of tiles} \label{sec:discr-hyperpl-coll}
Let $\boldsymbol{\sigma}$ be a sequence of substitutions over the alphabet~$\mathcal{A}$ with generalized right eigenvector~$\mathbf{u}$.
For any vector $\mathbf{w} \in \mathbb{R}_{\ge0}^d \setminus \{\mathbf{0}\}$, we consider the collections of tiles 
\[
\mathcal{C}_\mathbf{w} = \big\{\pi_{\mathbf{u},\mathbf{w}}\, \mathbf{x} + \mathcal{R}_\mathbf{w}(i):\, [\mathbf{x},i] \in \Gamma(\mathbf{w})\big\} \quad \mbox{and} \quad \widehat{\mathcal{C}}_{\mathbf{w}} = \big\{\mathbf{z} + \widehat{\mathcal{R}}_{\mathbf{w}}(i):\, i\in\mathcal{A},\, \mathbf{z}\in \mathbb{Z}^d\big\},
\]
where 
\[
\Gamma(\mathbf{w}) = \big\{[\mathbf{x}, i] \in \mathbb{Z}^d \times \mathcal{A}:\, 0 \le \langle \mathbf{w}, \mathbf{x}\rangle < \langle \mathbf{w}, \mathbf{e}_i \rangle\big\}
\]
is the \emph{discrete hyperplane}\footnote{A~geometric interpretation can be given to the notation $[\mathbf{x},i] \in \mathbb{Z}^d \times \mathcal{A}$ by setting
$[\mathbf{x},i] = \{\mathbf{x} + \sum_{j\in\mathcal{A},\, j \neq i} \lambda_j  \mathbf{e}_j:\, \lambda_j \in [0,1],\ j \in \mathcal{A}\}$, which turns $\Gamma(\mathbf{w})$ into a \emph{stepped hyperplane}.} approximating~$\mathbf{w}^\bot$.
We endow $\Gamma(\mathbf{w})$ with a product metric of the distance induced by $||\cdot||=||\cdot||_\infty$ on $\mathbb{Z}^d$ and some metric on~$\mathcal{A}$. This notion of discrete hyperplane corresponds to the notion of standard discrete hyperplane in discrete geometry; see~\cite{Reveilles:91}. 

In the particular case $\mathbf{w} = \mathbf{1}$, the collection
\[
\mathcal{C}_\mathbf{1} = \{\mathbf{x} + \mathcal{R}(i):\, \mathbf{x} \in \mathbb{Z}^d \cap \mathbf{1}^\bot,\, i \in \mathcal{A}\}
\]
consists of the translations of (the subtiles of) the Rauzy fractal by vectors in the lattice $\mathbb{Z}^d \cap \mathbf{1}^\bot$. 
The collection~$\mathcal{C}_\mathbf{1} $ generalizes the periodic tiling introduced for unimodular Pisot (irreducible) substitutions.
For particular vectors~$\mathbf{v}$ that will be specified in Section~\ref{subsec:choicev}, the collection~$\mathcal{C}_\mathbf{v}$ generalizes the corresponding aperiodic tiling that is obtained in the Pisot case by taking for $\mathbf{v}$ a left  Perron-Frobenius eigenvector of~$M_\sigma$; see e.g.~\cite{Ito-Rao:06}. 

We also recall the formalism of \emph{dual substitutions} introduced in~\cite{Arnoux-Ito:01}.
For $[\mathbf{x}, i] \in \mathbb{Z}^d \times \mathcal{A}$ and a unimodular substitution $\sigma$ on~$\mathcal{A}$, let
\begin{equation}\label{eq:dualsubst}
E_1^*(\sigma)[\mathbf{x}, i] = \big\{[M_\sigma^{-1} (\mathbf{x} + \mathbf{l}(p)), j]:\, \mbox{$j \in \mathcal{A}$, $p\in\mathcal{A}^*$ such that $p\hspace{.1em}i$ is a prefix of $\sigma(j)$}\big\}.
\end{equation}
We will recall basic properties of $E_1^*$ in Section~\ref{sec:dualsubst}.
In order to make this formalism work, we assume that our substitutions are unimodular. Observe that a non-unimodular theory in the Pisot substitutive case has also been developed; see e.g.\ \cite{MineThus} and the references therein.

\subsection{Coincidences and geometric finiteness}
A~sequence of substitutions $\boldsymbol{\sigma} = (\sigma_n)_{n\in\mathbb{N}}$ satisfies the \emph{strong coincidence condition} if there is $\ell \in \mathbb{N}$ such that, for each pair $(j_1,j_2) \in \mathcal{A} \times \mathcal{A}$, there are $i \in \mathcal{A}$ and $p_1, p_2 \in \mathcal{A}^*$ with $\mathbf{l}(p_1) = \mathbf{l}(p_2)$ such that $\sigma_{[0,\ell)}(j_1) \in p_1\hspace{.1em}i\hspace{.1em}\mathcal{A}^*$ and $\sigma_{[0,\ell)}(j_2) \in p_2\hspace{.1em}i\hspace{.1em}\mathcal{A}^*$. 
As in the periodic case, this condition will ensure that the subtiles~$\mathcal{R}(i)$ are disjoint in measure and, hence, define an exchange of domains on~$\mathcal{R}$
(see Proposition \ref{p:strongcoincidence}; the same conclusion is true for a suffix version of strong coincidence, see Remark~\ref{rem:-}).

We say that $\boldsymbol{\sigma} = (\sigma_n)_{n\in\mathbb{N}}$ satisfies the \emph{geometric coincidence condition} if for each $R > 0$ there is $\ell \in \mathbb{N}$ such that, for all $n \ge \ell$, $E_1^*(\sigma_{[0,n)})[\mathbf{0},i_n]$ contains a ball of radius~$R$ of the discrete hyperplane $\Gamma(\tr{(M_{[0,n)})}\, \mathbf{1})$ for some $i_n \in \mathcal{A}$. This condition can be seen as an $S$-adic  dual analogue to the geometric coincidence condition (or super-coincidence condition) in \cite{Barge-Kwapisz:06,Ito-Rao:06,CANTBST}, which provides a tiling criterion.  Recall that  the periodic tiling yields the isomorphism with a toral translation and  thus pure discrete spectrum.
This criterion is a coincidence type condition in the same vein as the various coincidence conditions introduced in the usual Pisot framework; see e.g.\ \cite{Solomyak:97,AL11}.
In Proposition~\ref{p:gcc}, we give a variant of the geometric coincidence condition that can be checked algorithmically; see also Proposition~\ref{p:gccvariant}.

A~more restrictive condition is the \emph{geometric finiteness property}  stating that for each $R > 0$ there is $\ell \in \mathbb{N}$ such that $\bigcup_{i\in\mathcal{A}} E_1^*(\sigma_{[0,n)})[\mathbf{0},i]$ contains the ball $\{[\mathbf{x},i] \in \Gamma(\tr{(M_{[0,n)})}\, \mathbf{1}):\, \|\mathbf{x}\| \le R\}$ for all $n \ge \ell$. This implies that $\bigcup_{i\in\mathcal{A}} E_1^*(\sigma_{[0,n)})[\mathbf{0},i]$ generates a whole discrete plane if $n\to\infty$, and that $\mathbf{0}$ is an inner point of  the Rauzy fractal; see Proposition~\ref{p:gccvariant}. This condition is a geometric variant of the  finiteness property in the framework of beta-numeration~\cite{FS92}.

\section{Main results}\label{sec:mainresults}

\subsection{General results on $S$-adic shifts}
Theorem~\ref{t:1}, our first result, sets the stage for all the subsequent results. It gives a variety of properties of $S$-adic shifts $(X_{\boldsymbol{\sigma}},\Sigma)$ under general conditions. Indeed, the set $S$ of unimodular substitutions from which the directive sequence $\boldsymbol{\sigma}$ is formed may be finite or  infinite in this theorem. Primitivity and  algebraic irreducibility are the analogs of primitivity and irreducibility (of the characteristic polynomial of the incidence  matrix)  of a substitution $\sigma$ in the periodic case. To guarantee minimality of $(X_{\boldsymbol{\sigma}},\Sigma)$ in the $S$-adic setting, we require the directive sequence~$\boldsymbol{\sigma}$ to be primitive; to guarantee unique ergodicity, in our setting we also assume recurrence on top of this (see the proof of Lemma~\ref{lem:uniquelyergodic}). Moreover, we need to have balancedness of the language~$\mathcal{L}_{\boldsymbol{\sigma}}$ to ensure that the associated Rauzy fractal~$\mathcal{R}$ is bounded. To endow~$\mathcal{R}$ with a convenient subdivision structure (replacing the graph-directed self-affine structure of the periodic case), uniform balancedness properties of the ``desubstituted'' languages~$\mathcal{L}_{\boldsymbol{\sigma}}^{(n)}$ are needed for infinitely many (but not all)~$n$. 
These assumptions are not very restrictive in the sense that they will enable us to prove metric results valid for almost all sequences of $S$-adic shifts under the Pisot condition as specified in Theorem~\ref{t:3}.

\begin{theorem} \label{t:1}
Let $S$ be a finite or  infinite set of unimodular substitutions over the finite alphabet~$\mathcal{A}$ and let $\boldsymbol{\sigma} = (\sigma_n)_{n\in\mathbb{N}} \in S^{\mathbb{N}}$ be a primitive and algebraically irreducible directive sequence. Assume that there is $C > 0$ such that for each $\ell \in \mathbb{N}$, there is $n \ge 1$ with $(\sigma_{n},\ldots,\sigma_{n+\ell-1}) = (\sigma_{0},\ldots,\sigma_{\ell-1})$ and the language $\mathcal{L}_{\boldsymbol{\sigma}}^{(n+\ell)}$ is $C$-balanced.
Then the following results are true.
\renewcommand{\theenumi}{\roman{enumi}}
\begin{enumerate}
\itemsep1ex
\item \label{i:11}
The $S$-adic shift $(X_{\boldsymbol{\sigma}},\Sigma)$ is minimal and uniquely ergodic with unique invariant measure~$\mu$. 
\item \label{i:12}
Each subtile~$\mathcal{R}(i)$, $i \in \mathcal{A}$, of the Rauzy fractal~$\mathcal{R}$ is a compact set that is the closure of its interior; its boundary has zero Lebesgue measure~$\lambda_{\mathbf{1}}$.
\item \label{i:13}
The collection~$\mathcal{C}_\mathbf{1}$ forms a multiple tiling of~$\mathbf{1}^\bot$, and the $S$-adic shift $(X_{\boldsymbol{\sigma}},\Sigma,\mu)$ admits as a factor (with finite fiber) a translation on the torus~$\mathbb{T}^{d-1}$. As a consequence, it is not weakly mixing.
\item \label{i:14}
If $\boldsymbol{\sigma}$ satisfies the strong coincidence condition, then the subtiles~$\mathcal{R}(i)$, $i \in \mathcal{A}$, are mutually disjoint in measure, and the $S$-adic shift $(X_{\boldsymbol{\sigma}},\Sigma,\mu)$ is measurably conjugate  to an exchange of domains on~$\mathcal{R}$. 
\item\label{i:15}
The collection~$\mathcal{C}_\mathbf{1}$ forms a tiling of~$\mathbf{1}^\bot$ if and only if $\boldsymbol{\sigma}$ satisfies the geometric coincidence condition. 
\end{enumerate}
If moreover $\mathcal{C}_\mathbf{1}$ forms a tiling of~$\mathbf{1}^\bot$, then also the following results hold.
\begin{enumerate} \setcounter{enumi}{5}
\itemsep1ex
\item \label{i:16}
The $S$-adic shift $(X_{\boldsymbol{\sigma}},\Sigma,\mu)$ is measurably conjugate to a translation $T$ on the torus~$\mathbb{T}^{d-1}$; in particular, its 
measure-theoretic spectrum is purely discrete. 
\item\label{i:17}
Each $\omega\in X_{\boldsymbol{\sigma}}$ is a natural coding of the toral translation~$T$ with respect to the partition $\{\mathcal{R}(i):\,i \in \mathcal{A}\}$.
\item\label{i:18}
The set $\mathcal{R}(i)$ is a bounded remainder set for the toral translation~$T$ for each $i\in\mathcal{A}$.
\end{enumerate}
\end{theorem}

Note that the assumptions in Theorem~\ref{t:1} obviously imply that the sequence~$\boldsymbol{\sigma}$ is recurrent.

\begin{remark}\label{rem:w}
We will prove in Propositions~\ref{p:independentmultiple} and~\ref{p:tilingRd} that, under the conditions of Theorem~\ref{t:1}, for each $\mathbf{w} \in \mathbb{R}^d_{\ge0} \setminus \{\mathbf{0}\}$ the collection~$\mathcal{C}_\mathbf{w}$ forms a multiple tiling of~$\mathbf{w}^\bot$ with covering degree~$m$ not depending on~$\mathbf{w}$, and $\widehat{\mathcal{C}}_\mathbf{w}$ forms a multiple (lattice) tiling of~$\mathbb{R}^d$ with the same covering degree~$m$. 

In particular, if $m = 1$, then $\bigcup_{i\in\mathcal{A}} \widehat{\mathcal{R}}_\mathbf{w}(i)$ is a fundamental domain of $\mathbb{R}^d / \mathbb{Z}^d$. This will be the key result for defining non-stationary Markov partitions associated with two-sided Pisot $S$-adic systems (e.g., two-sided directive sequences in the framework of natural extensions of continued fraction algorithms), that we will investigate in the forthcoming paper~\cite{sadic3}.
The vector~$\mathbf{w}$ is then given by a sequence $(\sigma_n)_{n<0}$.

Moreover, taking $\mathbf{w} = \mathbf{e}_i$, we obtain that each subtile $\mathcal{R}(i)$ tiles periodically.
This result seems to be new even in the periodic case. 
\end{remark}

\begin{theorem}\hspace{-.5em}\footnote{With the new proof presented in this manuscript, the set of vertices of $G$ can be infinite. This means that every shift can be represented as $(E_G,\Sigma)$ for some graph $G$. Since we do not want this manuscript to differ too much from the one published in {\it Ann.\ Inst.\ Fourier (Grenoble)}, we have kept the presentation using $S$-adic graphs.}\label{t:3}
Let $S$ be a finite or  infinite set of unimodular substitutions and let $(G,\tau)$ be an $S$-adic graph with associated edge shift $(E_G, \Sigma, \nu)$. Assume that this shift is ergodic, the cocycle $A$ is log-integrable, and that it satisfies the Pisot condition. Assume further that $\nu$ assigns positive measure to each (non-empty) cylinder, and that  there exists a cylinder corresponding to a substitution with positive incidence matrix.
Then, for the directive sequence $\boldsymbol{\sigma}=(\tau(\gamma_n))$ of $\nu$-almost every walk $\boldsymbol{\gamma}=(\gamma_n) \in E_G$,
\renewcommand{\theenumi}{\roman{enumi}}
\begin{enumerate}
\itemsep1ex
\item Assertions (\ref{i:11})--(\ref{i:15}) of Theorem~\ref{t:1} hold;
\item Assertions (\ref{i:16})--(\ref{i:18}) of Theorem~\ref{t:1} hold provided that the collection $\mathcal{C}_\mathbf{1}$ associated with~$\boldsymbol{\sigma}$ forms a tiling of $\mathbf{1}^\bot$.
\end{enumerate}
\end{theorem}

\begin{remark}\label{rem:t:3}
The setting of Theorem~\ref{t:3} covers the (additive) Arnoux-Rauzy and Brun algorithms (see Sections~\ref{sec:arnoux-rauzy-words} and~\ref{sec:brun-words-natural}; recall that the assumption that the cocycle $A$ is log-integrable is always satisfied when the $S$-adic graph is finite), but also includes many multiplicative continued fraction algorithms (which correspond to infinite sets~$S$). Most prominently, according to \cite{Perron:07} (see also \cite[Proposition~8]{SCHWEIGER}) the admissible sequences of the Jacobi-Perron algorithm can be represented by an $S$-adic graph with finitely many vertices and log-integrability of the associated cocycle is proved in~\cite{Lagarias}. For the two-dimensional case an associated (infinite) set of substitutions can be found for instance in~\cite{BBJS14} (this can easily be generalized to higher dimensions). Also, the acceleration of the Arnoux-Rauzy algorithm  together with the  invariant measure proposed in~\cite{AvilaHubSkrip} fits into the framework of Theorem~\ref{t:3}.
\end{remark}

We think that the conditions of Theorem~\ref{t:1} are enough to get a tiling of~$\mathbf{1}^\bot$ by~$\mathcal{C}_\mathbf{1}$ and, hence, measurable conjugacy  of $(X_{\boldsymbol{\sigma}},\Sigma)$ to a toral translation. 
This extension of the well-known \emph{Pisot substitution conjecture} to the $S$-adic setting is made precise in the following statement.
(Here, we also replace uniform balancedness of~$\mathcal{L}_{\boldsymbol{\sigma}}^{(n+\ell)}$ by the weaker condition that $\mathcal{L}_{\boldsymbol{\sigma}}$ is balanced.) Note that  the word ``Pisot'' does not occur
in the statement of the conjecture but the generalization of the Pisot hypothesis is provided by  the balancedness assumption. 

\begin{conjecture}[$S$-adic Pisot conjecture] \label{c:1}
Let $S$ be a finite or  infinite set of unimodular substitutions over the finite alphabet~$\mathcal{A}$ 
and let $\boldsymbol{\sigma}\in S^\mathbb{N}$ be a primitive, algebraically irreducible, and recurrent directive sequence with balanced language~$\mathcal{L}_{\boldsymbol{\sigma}}$. 
Then $\mathcal{C}_\mathbf{1}$ forms a tiling of~$\mathbf{1}^\bot$, and the $S$-adic shift $(X_{\boldsymbol{\sigma}},\Sigma,\mu)$ is measurably conjugate to a translation on the torus~$\mathbb{T}^{d-1}$; in particular, its measure-theoretic spectrum is purely discrete. 
\end{conjecture}

\begin{remark}\label{rem:LR}
It would already be interesting to get this conjecture for sequences $\boldsymbol{\sigma}$ with linearly recurrent limit word $\omega$. In view of \cite[Proposition~1.1]{Durand:00b} this would entail to prove Conjecture~\ref{c:1} for $S$-adic shifts that are topologically conjugate to \emph{proper $S$-adic shifts}. A \emph{proper $S$-adic shift} is an $S$-adic shift for which $S$ is a set of \emph{proper} substitutions. Recall that a substitution $\sigma$ is \emph{proper} if there are letters $r,l\in \mathcal{A}$ such that the word $\sigma(a)$ starts with $l$ and ends with $r$ for each $a\in \mathcal{A}$. It is shown in \cite[Proposition~25]{DHS:99} that in the substitutive case we always have linear recurrence (indeed, a substitutive dynamical system is always topologically conjugate to a proper substitutive system, see~\cite[Section~5]{DHS:99}). Thus even this special case contains the classical Pisot substitution conjecture. 

We are able to prove that linearly recurrent Arnoux-Rauzy words with recurrent  directive   sequence  give rise to $S$-adic shifts that have pure discrete spectrum (see Corollary~\ref{cor:AR}).

In this context it would be also of interest to generalize Barge's result~\cite{Barge15,Barge14}, where pure discrete spectrum is proved for a large class of substitutive systems characterized by certain combinatorial properties (including beta substitutions), to the $S$-adic setting. We have provided a proof of the two-letter alphabet version of Conjecture~\ref{c:1} with the additional assumptions on uniform balancedness of Theorem~\ref{t:1} in \cite{BMST:15}.
\end{remark}

We work here with the $\mathbb{Z}$-action provided by the $S$-adic shift. However, under the assumptions of Theorem \ref{t:1} (with the balancedness assumption playing a crucial role), our results also apply to the $\mathbb{R}$-action of the associated tiling space (such as investigated e.g.\ in~\cite{CS:03}), according to~\cite{Sadun:15}.

\subsection{Arnoux-Rauzy words and the conjecture of Arnoux and Rauzy}\label{sec:arnoux-rauzy-words}

For certain sets $S$ of substitutions, we get the assertions of Theorems~\ref{t:1} and~\ref{t:3} unconditionally for a large collection of directive sequences in~$S^\mathbb{N}$. Arnoux and Rauzy~\cite{Arnoux-Rauzy:91} proposed a generalization of Sturmian words to three letters (which initiated an important literature around so-called  episturmian words, see e.g.\ \cite{Berstel:07}). They proved that these \emph{Arnoux-Rauzy words}  can be expressed as $S$-adic words if $S = \{\alpha_i:\, i \in \mathcal{A}\}$ is the set of \emph{Arnoux-Rauzy substitutions} over $\mathcal{A}=\{1,2,3\}$ defined by
\begin{equation}\label{eq:AR}
\alpha_i:\ i \mapsto i,\ j \mapsto ji\ \mbox{for}\ j \in \mathcal{A} \setminus \{i\}\qquad (i\in\mathcal{A})\,.
\end{equation}
It was conjectured since the early nineties (see e.g.~\cite[p.~1267]{Cassaigne-Ferenczi-Zamboni:00} or \cite[Section~3.3]{Berthe-Ferenczi-Zamboni:05}) that each Arnoux-Rauzy word is a natural coding of a translation on the torus. Cassaigne et al.~\cite{Cassaigne-Ferenczi-Zamboni:00} provided a counterexample to this conjecture by constructing unbalanced Arnoux-Rauzy words (unbalanced words cannot come from natural codings by a result of Rauzy~\cite{Rauzy:84}).
Moreover, Cassaigne et al.~\cite{Cassaigne-Ferenczi-Messaoudi:08} even showed that there exist Arnoux-Rauzy words $\omega$ on three letters such that $(X_\omega,\Sigma)$ is weakly mixing (w.r.t.\ the unique $\Sigma$-invariant probability measure on~$X_\omega$).

To our knowledge, positive examples for this conjecture so far existed only in the periodic case; cf.~\cite{Berthe-Jolivet-Siegel:12,Barge-Stimac-Williams:13}. The metric result in Theorem~\ref{t:3} allows us to prove  the following theorem which confirms the conjecture of Arnoux and Rauzy  almost everywhere.

\begin{theorem} \label{t:5}
Let $S$ be the  set of Arnoux-Rauzy substitutions over three letters and consider the shift $(S^\mathbb{N},\Sigma,\nu)$ for some shift invariant ergodic probability measure~$\nu$ which assigns positive measure to each cylinder. Then $(S^\mathbb{N},\Sigma,\nu)$ satisfies the Pisot condition. Moreover,  for $\nu$-almost all sequences $\boldsymbol{\sigma} \in S^\mathbb{N}$ the collection~$\mathcal{C}_\mathbf{1}$ forms a tiling, the $S$-adic shift $(X_{\boldsymbol{\sigma}},\Sigma)$ is measurably conjugate to a translation on the torus~$\mathbb{T}^2$, and
the words in~$X_{\boldsymbol{\sigma}}$ form natural codings of this translation.
\end{theorem}
As an example of   measure satisfying the assumptions of    Theorem \ref{t:5},  consider the   measure of maximal entropy for the suspension
flow of the Rauzy gasket    constructed  in \cite{AvilaHubSkripbis} (see also \cite{AvilaHubSkrip}).
Using Theorem~\ref{t:1} we are also able to provide a (uncountable) class of non-substitutive Arnoux-Rauzy words that give rise to translations on the torus $\mathbb{T}^{2}$.
To this end we introduce a terminology that comes from  the associated Arnoux-Rauzy continued fraction algorithm (which was also defined in \cite{Arnoux-Rauzy:91}).
A directive sequence  $\boldsymbol{\sigma}=(\sigma_n)\in S^\mathbb{N}$   that  contains each $\alpha_i$ ($i=1,2,3$) infinitely often
is said to have {\it bounded  weak partial quotients} if there is $h \in \mathbb{N}$ such that $\sigma_n = \sigma_{n+1} = \cdots = \sigma_{n+h}$ does not hold for any $n \in \mathbb{N}$,  and {\it bounded strong partial quotients}  if every  substitution in the directive sequence  $\boldsymbol{\sigma}$ occurs  with bounded gap.   We also recall   that  a directive sequence $\boldsymbol{\sigma}$   is said to be recurrent  if each factor of~$\boldsymbol{\sigma}$ occurs infinitely often in~$\boldsymbol{\sigma}$.

\begin{theorem} \label{t:4}
Let $S=\{\alpha_1,\alpha_2,\alpha_3\}$ be the set of Arnoux-Rauzy substitutions over three letters. Let $\boldsymbol{\sigma} \in S^\mathbb{N}$ be given and assume that it contains each of the substitutions  $\alpha_1,\alpha_2,\alpha_3$ at least once. If $\boldsymbol{\sigma}$ is recurrent and has bounded weak partial quotients, then the collection~$\mathcal{C}_\mathbf{1}$ forms a tiling, the $S$-adic shift $(X_{\boldsymbol{\sigma}},\Sigma)$ is measurably conjugate to a translation on the torus~$\mathbb{T}^2$, and the words in~$X_{\boldsymbol{\sigma}}$ form natural codings of this translation. 
\end{theorem}

Note that examples of uniformly  balanced words    (for which $\omega^{(n)}$ is $C$-balanced for each~$n$) for the $S$-adic shifts generated by  Arnoux-Rauzy substitutions are provided  in \cite{Berthe-Cassaigne-Steiner}. In particular, boundedness of the  strong partial quotients provides a nice characterization of  linear recurrence for Arnoux-Rauzy words (see Proposition~\ref{prop:LR} below). We believe that linear recurrence of an Arnoux-Rauzy word is enough to obtain pure discrete spectrum. However, in view of Proposition~\ref{prop:LR} our Theorem~\ref{t:4} yields this result only subject to the additional assumption of recurrence on the directive sequence. This conditional result is the content of the following corollary.

\begin{corollary}\label{cor:AR}
Any linearly recurrent Arnoux-Rauzy word~$\omega$ with recurrent directive sequence  generates a symbolic dynamical system $(X_{\omega}, \Sigma)$ that has  pure discrete spectrum.
\end{corollary}

It is well known that Arnoux-Rauzy words can be defined also for $d>3$ letters (see e.g.~\cite{Berthe-Cassaigne-Steiner}). To apply our theory to these classes of words and prove the results of this section in this more general setting it would be necessary to extend the combinatorial results from \cite{Berthe-Jolivet-Siegel:12} to higher dimensions. Although this should be possible we expect it to be very tedious already in the case of 4 letters.

\subsection{Brun words and natural codings of rotations with linear complexity}\label{sec:brun-words-natural}

Let $\Delta_2:=\{(x_1,x_2)\in\mathbb{R}^2\,:\, 0\le x_1\le x_2 \le 1\}$ be equipped with the Lebesgue measure $\lambda_{2}$.
Brun~\cite{BRUN} devised a generalized continued fraction algorithm for vectors $(x_1,x_2)\in\Delta_2$. This algorithm (in its ``projectivized'' additive form) is defined by the mapping $T_{\rm Brun}: \Delta_2\to\Delta_2$, 
\begin{equation}\label{eq:brunmap}
T_{\rm Brun}: (x_1,x_2) \mapsto 
\begin{cases}
\left(\frac{x_1}{1-x_2},\frac{x_2}{1-x_2}\right), & \hbox{for } x_2 \le \frac12, \\
\left(\frac{x_1}{x_2},\frac{1-x_2}{x_2}\right), & \hbox{for } \frac12 \le x_2 \le 1-x_1,\\
\left(\frac{1-x_2}{x_2},\frac{x_1}{x_2}\right), & \hbox{for }1-x_1 \le x_2 \leq 1;
\end{cases}
\end{equation}
for later use, we define $B(i)$ to be the set of $(x_1,x_2)\in\Delta_2$ meeting the restriction in the $i$-th line of \eqref{eq:brunmap}, for $1\le i\le 3$. An easy computation shows that the linear version of this algorithm is defined for vectors $\mathbf{w}^{(0)}=(w_1^{(0)},w_2^{(0)},w_3^{(0)})$ with $0\le w_1^{(0)} \le w_2^{(0)} \le w_3^{(0)}$ by the recurrence $M_{i_n}\mathbf{w}^{(n)}=\mathbf{w}^{(n-1)}$, where  $M_{i_n}$ is chosen among the matrices 
\begin{equation}\label{eq:brunmatrices}
\begin{pmatrix}1&0&0\\0&1&0\\0&1&1 \end{pmatrix},\quad
\begin{pmatrix}1&0&0\\0&0&1\\0&1&1 \end{pmatrix},\quad
\begin{pmatrix}0&1&0\\0&0&1\\1&0&1 \end{pmatrix}
\end{equation}
according to the magnitude of $w_3^{(n-1)}-w_2^{(n-1)}$ compared to $w_1^{(n-1)}$ and~$w_2^{(n-1)}$. 
More precisely, we have $T_{\rm Brun}\big(w_1^{(n-1)}/w_3^{(n-1)},w_2^{(n-1)}/w_3^{(n-1)}\big) = \big(w_1^{(n)}/w_3^{(n)},w_2^{(n)}/w_3^{(n)}\big)$.
We associate $S$-adic words with this algorithm by defining the \emph{Brun substitutions}
\begin{equation}\label{eq:brun}
\beta_1 : \begin{cases} 1 \mapsto 1  \\ 2 \mapsto 23 \\ 3 \mapsto 3 \end{cases} \quad
\beta_2 : \begin{cases} 1 \mapsto 1  \\ 2 \mapsto 3 \\ 3 \mapsto 23 \end{cases} \quad
\beta_3 : \begin{cases} 1 \mapsto 3  \\ 2 \mapsto 1 \\ 3 \mapsto 23 \end{cases}
\end{equation}
whose incidence matrices coincide with the three matrices in \eqref{eq:brunmatrices} associated with Brun's algorithm. Examples of uniformly  balanced words    (for which $\omega^{(n)}$ is $C$-balanced for each~$n$) for the $S$-adic shifts generated by  Brun substitutions are provided in  \cite{Delecroix-Hejda-Steiner}. We prove the following result on the related $S$-adic words.

\begin{theorem} \label{t:6}
Let $S = \{\beta_1,\beta_2,\beta_3\}$ be the set of Brun substitutions over  three letters, and consider the shift $(S^\mathbb{N},\Sigma,\nu)$ for some shift invariant ergodic probability measure~$\nu$ that assigns positive measure to each cylinder. Then $(S^\mathbb{N},\Sigma,\nu)$ satisfies the Pisot condition. Moreover,  for $\nu$-almost all sequences $\boldsymbol{\sigma} \in S^\mathbb{N}$ the collection~$\mathcal{C}_\mathbf{1}$ forms a tiling, the $S$-adic shift $(X_{\boldsymbol{\sigma}},\Sigma)$ is measurably conjugate to a translation on the torus~$\mathbb{T}^2$, and
the words in~$X_{\boldsymbol{\sigma}}$ form natural codings of this translation.\end{theorem}

We will now show that this result implies that the $S$-adic shifts associated with Brun's algorithm provide a natural coding of almost all rotations on the torus~$\mathbb{T}^2$. Indeed, by the (weak) convergence of Brun's algorithm for almost all $(x_1,x_2)\in\Delta_2$ (w.r.t.\ to the two-dimensional Lebesgue measure; see e.g.\ \cite{BRUN}), there is a bijection $\Phi$ defined for almost all $(x_1,x_2)\in\Delta_2$ that makes the diagram
\begin{equation}\label{CDPhi}
\begin{CD}
\Delta_2 @>T_{\rm Brun}>> \Delta_2 \\
@VV \Phi V  @ VV \Phi V \\
S^{\mathbb{N}} @>\Sigma>> S^\mathbb{N}
\end{CD}
\end{equation}
commutative and that provides a measurable conjugacy between $(\Delta_2,T_{\rm Brun},\lambda_2)$ and $(S^{\mathbb{N}} ,\Sigma,\nu)$; the measure $\nu$ is specified in the proof of Theorem~\ref{t:7}.

\begin{theorem}\label{t:7}
For almost all $(x_1,x_2) \in \Delta_2$, the $S$-adic shift $(X_{\boldsymbol{\sigma}},\Sigma)$ with $\boldsymbol{\sigma} = \Phi(x_1,x_2)$ is measurably conjugate to the translation by $\big(\frac{x_1}{1+x_1+x_2},\frac{x_2}{1+x_1+x_2}\big)$ on~$\mathbb{T}^2$; then each $\omega\in X_{\boldsymbol{\sigma}}$ is a natural coding for this translation, $\mathcal{L}_{\boldsymbol{\sigma}}$~is balanced, and the subpieces of the Rauzy fractal  provide bounded remainder sets for this translation.
\end{theorem} 

This result has the following consequence.

\begin{corollary}\label{cor:8}
For almost all $\mathbf{t} \in \mathbb{T}^2$, there is $(x_1,x_2) \in \Delta_2$ such that the $S$-adic shift $(X_{\boldsymbol{\sigma}},\Sigma)$ with $\boldsymbol{\sigma} = \Phi(x_1,x_2)$ is measurably conjugate to the translation by~$\mathbf{t}$ on~$\mathbb{T}^2$. Moreover, the words in~$X_{\boldsymbol{\sigma}}$ form natural codings of the translation by~$\mathbf{t}$.
\end{corollary}

We believe that the codings mentioned in Theorem~\ref{t:7} have linear factor complexity, that is, for each such coding, there is $C>0$ such that the number of its factors of length~$n$ is less than~$Cn$. Indeed, S.~Labb\'e and J.~Leroy informed us that they are currently working on a proof of the fact that $S$-adic words with $S = \{\beta_1,\beta_2,\beta_3\}$ have linear  factor complexity \cite{LabLer}. 
We thus get bounded remainder sets whose characteristic infinite words have linear factor complexity, contrarily to the examples provided e.g.\ in~\cite{Chevallier,Grepstad-Lev}.

\section{Convergence properties}\label{sec:conv}
In this section, we show that the Rauzy fractal~$\mathcal{R}$ corresponding to a sequence~$\boldsymbol{\sigma}$ is bounded if $\mathcal{L}_{\boldsymbol{\sigma}}$ is balanced. Moreover, we prove that under certain conditions the letter  frequency vector of an $S$-adic word has rationally independent entries and give a criterion that ensures the strong convergence of the matrix products $M_{[0,n)}$ to one single direction (defined by  a generalized  right eigenvector provided by the letter frequency vector). All these results will be needed in the sequel.

Throughout this section $S$ is a (finite or  infinite)  set of substitutions over the finite alphabet~$\mathcal{A}$ and $\boldsymbol{\sigma} \in S^{\mathbb{N}}$ is a directive sequence.

\subsection{Boundedness of the Rauzy fractal}
Recall that the Rauzy fractal~$\mathcal{R}$ is the closure of the projection of the vertices of the broken lines defined by limit words of~$\boldsymbol{\sigma}$; see Section~\ref{subsec:FR}. Therefore, $\mathcal{R}$~is compact if and only if the broken lines remain at bounded distance from the generalized right eigendirection~$\mathbb{R} \mathbf{u}$.
The following result shows that this is equivalent with balancedness and establishes a connection between the degree of balancedness and the diameter of~$\mathcal{R}$; see also \cite[Proposition~7]{Adamczewski:03} and~\cite[Lemma~3]{Delecroix-Hejda-Steiner}.
Recall that $\|\cdot\|$ denotes the maximum norm. 

\begin{lemma} \label{l:bounded}
Let $\boldsymbol{\sigma}$ be a primitive sequence of substitutions with generalized right eigenvector~$\mathbf{u}$. 
Then $\mathcal{R}$ is bounded if and only if $\mathcal{L}_{\boldsymbol{\sigma}}$ is balanced. 
If $\mathcal{L}_{\boldsymbol{\sigma}}$ is $C$-balanced, then $\mathcal{R} \subset [-C,C]^d \cap \mathbf{1}^\bot$.
\end{lemma}

 \begin{proof}
Assume first that $\mathcal{R}$ is bounded. 
Then there exists $C$ such that $\|\pi_{\mathbf{u},\mathbf{1}}\, \mathbf{l}(p)\| \le C$ for all prefixes $p$ of limit words of~$\boldsymbol{\sigma}$.
Let $u, v \in \mathcal{L}_{\boldsymbol{\sigma}}$ with $|u| = |v|$.
By the primitivity of $\boldsymbol{\sigma}$, $u$~and $v$ are factors of a limit word, hence, $\|\pi_{\mathbf{u},\mathbf{1}}\, \mathbf{l}(u)\|,\, \|\pi_{\mathbf{u},\mathbf{1}}\, \mathbf{l}(v)\| \le 2C$.
As $\mathbf{l}(u) - \mathbf{l}(v) \in \mathbf{1}^\bot$, we obtain
\[
\|\mathbf{l}(u) - \mathbf{l}(v)\| = \|\pi_{\mathbf{u},\mathbf{1}}\, (\mathbf{l}(u) - \mathbf{l}(v))\| \le \|\pi_{\mathbf{u},\mathbf{1}}\, \mathbf{l}(u)\| + \|\pi_{\mathbf{u},\mathbf{1}}\, \mathbf{l}(v)\| \le 4C,
\]
i.e., $\mathcal{L}_{\boldsymbol{\sigma}}$ is $4C$-balanced.

Assume now that $\mathcal{L}_{\boldsymbol{\sigma}}$ is $C$-balanced and let $p$ be a prefix of a limit word~$\omega$.
Write $\omega$ as concatenation of words~$v_k$, $k \in \mathbb{N}$, with $|v_k| = |p|$. 
Then $C$-balancedness yields $\|\pi_{\mathbf{u},\mathbf{1}}\, \mathbf{l}(v_k) - \pi_{\mathbf{u},\mathbf{1}}\, \mathbf{l}(p)\| \le C$ for all $k \in \mathbb{N}$, thus $\|\frac1n \sum_{k=0}^{n-1} \pi_{\mathbf{u},\mathbf{1}}\, \mathbf{l}(v_k) - \pi_{\mathbf{u},\mathbf{1}}\, \mathbf{l}(p)\| \le C$ for all $n \in \mathbb{N}$.
As $M_{[0,n)}\, \mathbf{e}_i = \mathbf{l}(\sigma_{[0,n)}(i)) \in \mathbf{l}(\mathcal{L}_{\boldsymbol{\sigma}})$ for all $n \in \mathbb{N}$, $i \in \mathcal{A}$, the letter frequency vector of~$\omega$ (which exists because of balancedness~\cite{Berthe-Tijdeman:02}) is in~$\mathbb{R} \mathbf{u}$. 
Therefore, we have $\lim_{n\to\infty} \frac1n \sum_{k=0}^{n-1} \mathbf{l}(v_k) \in \mathbb{R} \mathbf{u}$, hence $\lim_{n\to\infty} \frac1n \sum_{k=0}^{n-1} \pi_{\mathbf{u},\mathbf{1}}\, \mathbf{l}(v_k) = \mathbf{0}$, and consequently 
\begin{equation}\label{lem41end}
\|\pi_{\mathbf{u},\mathbf{1}}\, \mathbf{l}(p)\| = \|\lim_{n\to\infty} \frac1n \sum_{k=0}^{n-1} \pi_{\mathbf{u},\mathbf{1}}\, \mathbf{l}(v_k) - \pi_{\mathbf{u},\mathbf{1}}\, \mathbf{l}(p)\| \le C. \qedhere
\end{equation}
 \end{proof}

\subsection{Irrationality and strong convergence}
In the periodic case with a unimodular irreducible Pisot substitution~$\sigma$, the  incidence matrix~$M_\sigma$ has an expanding right eigenline and a contractive right hyperplane (that is orthogonal to an expanding left eigenvector), i.e., the matrix~$M_\sigma$ contracts the space~$\mathbb{R}^d$ towards the expanding eigenline.
Moreover, irreducibility implies that the coordinates of the expanding eigenvector are rationally independent.
These properties are crucial for proving that the Rauzy fractal~$\mathcal{R}$ has positive measure and induces a (multiple) tiling of the representation space~$\mathbf{1}^\bot$.
In the $S$-adic setting, the cones $M_{[0,n)}\, \mathbb{R}_+^d$ converge ``weakly'' to the direction of the generalized right eigenvector~$\mathbf{u}$; see Section~\ref{sec:gener-perr-frob}. We give a criterion for $\mathbf{u}$ to have rationally independent coordinates in Lemma~\ref{l:independent}. As the weak convergence of $M_{[0,n)}\, \mathbb{R}_+^d$ to~$\mathbf{u}$ is not sufficient for our purposes, in Proposition~\ref{p:strongconvergence} we will provide a strong convergence property.

\begin{lemma} \label{l:independent}
Let $\boldsymbol{\sigma}$ be an algebraically irreducible sequence of substitutions with generalized right eigenvector~$\mathbf{u}$ and balanced language~$\mathcal{L}_{\boldsymbol{\sigma}}$.
Then the coordinates of~$\mathbf{u}$ are rationally independent.
\end{lemma}

\begin{proof}
Suppose that $\mathbf{u}$ has rationally dependent coordinates, i.e., there is $\mathbf{x} \in \mathbb{Z}^d \setminus \{\mathbf{0}\}$ such that $\langle \mathbf{x}, \mathbf{u}  \rangle= 0$.
Then $\langle \tr{(M_{[0,n)})}\, \mathbf{x},  \mathbf{e}_i  \rangle = \langle \mathbf{x},  M_{[0,n)}\, \mathbf{e}_i \rangle = \langle \mathbf{x},  \mathbf{l}(\sigma_{[0,n)}(i)) \rangle$ is bounded (uniformly in~$n$) for each $i \in \mathcal{A}$, by the balancedness of~$\mathcal{L}_{\boldsymbol{\sigma}}$; cf.\ the proof of Lemma~\ref{l:bounded}. Therefore, $\tr{(M_{[0,n)})}\, \mathbf{x} \in \mathbb{Z}^d$ is bounded, and there is $k \in \mathbb{N}$ such that $\tr{(M_{[0,\ell)})}\, \mathbf{x} = \tr{(M_{[0,k)})}\, \mathbf{x}$ for infinitely many $\ell > k$. The matrix $M_{[0,k)}$ is regular since otherwise $M_{[0,\ell)}$ would have the eigenvalue~$0$ for all $\ell \ge k$, contradicting algebraic irreducibility. Thus $\tr{(M_{[0,k)})}\, \mathbf{x} \not= \mathbf{0}$ is an eigenvector of~$\tr{(M_{[k,\ell)})}$ to the eigenvalue~$1$, contradicting that $M_{[k,\ell)}$ has irreducible characteristic polynomial for large~$\ell$.
\end{proof}

\begin{proposition} \label{p:strongconvergence}
Let $\boldsymbol{\sigma} = (\sigma_n)_{n\in\mathbb{N}}$ be a primitive, algebraically irreducible, and recurrent  sequence of substitutions with balanced language~$\mathcal{L}_{\boldsymbol{\sigma}}$. 
Then
\begin{equation} \label{e:Lsn}
\lim_{n\to\infty} \sup\big\{ \|\pi_{\mathbf{u},\mathbf{1}} M_{[0,n)}\,  \mathbf{l}(v)\|:\,  v \in \mathcal{L}_{\boldsymbol{\sigma}}^{(n)}\big\} = 0.
\end{equation}
In particular, 
\begin{equation} \label{e:contraction}
\lim_{n\to\infty} \pi_{\mathbf{u},\mathbf{1}} M_{[0,n)}\, \mathbf{e}_i = \mathbf{0} \quad \mbox{for all}\ i \in \mathcal{A}.
\end{equation}
\end{proposition}

Note that \eqref{e:contraction} is the \emph{strong convergence} property used in the theory of multidimensional continued fraction algorithms; see e.g.\ \cite[Definition~19]{SCHWEIGER}.

\begin{proof}
First note that \eqref{e:contraction} follows from~\eqref{e:Lsn} since $i \in \mathcal{L}_{\boldsymbol{\sigma}}^{(n)}$ for all $i \in \mathcal{A}$, $n \in \mathbb{N}$, by primitivity.

Let $\omega$ be a limit word of~$\boldsymbol{\sigma}$. 
Then, again by primitivity, for each $v \in \mathcal{L}_{\boldsymbol{\sigma}}^{(n)}$ we have $\mathbf{l}(v) = \mathbf{l}(p) - \mathbf{l}(q)$ for some prefixes $p, q$ of~$\omega^{(n)}$. 
Therefore, it is sufficient to prove that
\begin{equation} \label{e:Lsn2}
\lim_{n\to\infty} \sup\big\{ \|\pi_{\mathbf{u},\mathbf{1}} M_{[0,n)}\,  \mathbf{l}(p)\|:\,  \mbox{$p$ is a prefix of $\omega^{(n)}$}\big\} = 0.
\end{equation}

Choose $\varepsilon >0$ arbitrary but fixed. For all $n \in \mathbb{N}$, let $i_n$ be the first letter of $\omega^{(n)}$ and set
\begin{align*}
\mathcal{S}_n & = \{\pi_{\mathbf{u},\mathbf{1}}\, \mathbf{l}(p):\, \mbox{$p$ is a prefix of $\sigma_{[0,n)}(i_n)$}\} \\
\tilde{\mathcal{R}} & = \overline{\{\pi_{\mathbf{u},\mathbf{1}}\, \mathbf{l}(p):\, \mbox{$p$ is a prefix of $\omega$}\}}.
\end{align*}
Then $\lim_{n\to\infty} \mathcal{S}_n = \tilde{\mathcal{R}}$ (in Hausdorff metric) and $\pi_{\mathbf{u},\mathbf{1}} M_{[0,n)}\, \mathbf{l}(p) + \mathcal{S}_n \subset \tilde{\mathcal{R}}$ for all $p \in \mathcal{A}^*$ such that $p\, i_n$ is a prefix of~$\omega^{(n)}$. These two facts yield that 
\begin{equation}\label{uniform}
\|\pi_{\mathbf{u},\mathbf{1}} M_{[0,n)}\, \mathbf{l}(p)\| < \varepsilon 
\end{equation}
for all $p \in \mathcal{A}^*$ such that $p\hspace{.1em}i_n$  is a prefix of $\omega^{(n)}$ for $n$ large enough.
For $p \in \mathcal{A}^*$, let $N(p) = \{{n \in \mathbb{N}}:\, \mbox{$p\hspace{.1em}i_n$ is a prefix of $\omega^{(n)}$}\}$. If $N(p)$ is infinite, then \eqref{uniform} immediately implies that
\begin{equation}\label{Npinfinite}
\lim_{n\in N(p),\,n\to\infty} \pi_{\mathbf{u},\mathbf{1}} M_{[0,n)}\, \mathbf{l}(p) =\mathbf{0}.
\end{equation}
Our next aim is to find a set of prefixes~$p$ spanning~$\mathbb{R}^d$ that all yield an infinite set~$N(p)$. 

By recurrence of $(\sigma_n)_{n\in\mathbb{N}}$, there is an increasing sequence of integers $(n_k)_{k\in\mathbb{N}}$ such that 
\begin{equation}\label{sigmank}
(\sigma_{n_k}, \sigma_{n_k+1}, \ldots, \sigma_{n_k+k-1}) = (\sigma_0, \sigma_1, \ldots, \sigma_{k-1})
\end{equation}
for all $k \in \mathbb{N}$.
Using a Cantor diagonal argument we can choose a sequence of letters $(j_\ell)_{\ell\in\mathbb{N}}$ such that, for each $\ell \in \mathbb{N}$, we have that
\begin{equation}\label{jell1}
(i_{n_k},i_{n_k+1}, i_{n_k+2}, \ldots, i_{n_k+\ell}) = (j_0,j_1, j_2, \ldots, j_\ell)
\end{equation}
holds for infinitely many $k \in \mathbb{N}$; denote the set of these~$k$ by~$K_\ell$. By the definition of~$i_n$, we have that $\sigma_{n-1}(i_n) \in i_{n-1} \mathcal{A}^*$. For $k \in K_\ell$, we gain thus
\begin{equation}\label{jell2}
\sigma_{\ell-1}(j_\ell) = \sigma_{n_k+\ell-1}(j_\ell) = \sigma_{n_k+\ell-1}(i_{n_k+\ell}) \in i_{n_k+\ell-1} \mathcal{A}^* = j_{\ell-1} \mathcal{A}^*.
\end{equation}
 
Let $P_\ell$ be the set of all $p \in \mathcal{A}^*$ such that $p\hspace{.1em}j_0$ is a prefix of $\sigma_{[0,\ell)}(j_\ell)$. Then, \eqref{jell2} implies that $P_0 \subset P_1 \subset \cdots$. 
Consider the lattice $L \subset \mathbb{Z}^d$ generated by $\bigcup_{\ell\in\mathbb{N}} \mathbf{l}(P_\ell)$. 
The set $\bigcup_{\ell\in\mathbb{N}} \mathbf{l}(P_\ell)$ contains arbitrarily large vectors. 
Therefore, if the lattice~$L$ does not have full rank, then the rational independence of the coordinates of~$\mathbf{u}$ (Lemma~\ref{l:independent}) implies that the maximal distance of elements of $\bigcup_{\ell\in\mathbb{N}} \mathbf{l}(P_\ell)$ from the line $\mathbb{R}\mathbf{u}$ is unbounded. Since $P_\ell \subset \mathcal{L}_{\boldsymbol{\sigma}}$, this contradicts the fact that $\mathcal{L}_{\boldsymbol{\sigma}}$ is balanced; cf.~\eqref{lem41end} in the proof of Lemma~\ref{l:bounded}.
Hence, there is $\ell \in \mathbb{N}$ such that $\mathbf{l}(P_\ell)$ contains a basis of~$\mathbb{R}^d$.

We now fix~$\ell$ such that $\mathbf{l}(P_\ell)$ contains a basis of~$\mathbb{R}^d$.
If $p \in P_\ell$, i.e., if $p\hspace{.1em}j_0$ is a prefix of~$\sigma_{[0,\ell)}(j_\ell)$, then \eqref{sigmank} and~\eqref{jell1} imply that $p\hspace{.1em}j_0\, ({=}p\hspace{.1em}i_{n_k})$ is a prefix of~$\omega^{(n_k)}$ for all $k \in K_\ell$, thus ${\{n_k:\, k \in K_\ell\}} \subset N(p)$, 
which shows that $N(p)$ is infinite. Therefore we may apply \eqref{Npinfinite} to obtain that 
\[
\lim_{k\in K_\ell,\,k\to\infty} \pi_{\mathbf{u},\mathbf{1}} M_{[0,n_k)}\, \mathbf{l}(p) = \lim_{k\in N(p),\,k\to\infty} \pi_{\mathbf{u},\mathbf{1}} M_{[0,n_k)}\, \mathbf{l}(p) =\mathbf{0}.
\]
Since $\mathbf{l}(P_\ell)$ contains a basis of~$\mathbb{R}^d$, this yields that 
\begin{equation} \label{eq:st2}
\lim_{k\in K_\ell,\,k\to\infty} \pi_{\mathbf{u},\mathbf{1}} M_{[0,n_k)}\, \mathbf{x} = \mathbf{0} \mbox{  for all } \mathbf{x} \in \mathbb{R}^d.
\end{equation}

Let $h \in \mathbb{N}$ be such that $M_{[0,h)}$ is a positive matrix.
Then there is a finite set $Q \subset \mathcal{A}^*$ such that, for each $i \in \mathcal{A}$, $q\hspace{.1em}j_0$ is a prefix of $\sigma_{[0,h)}(i)$ for some $q \in Q$. 
Thus, for all sufficiently large $k \in K_\ell$, 
\begin{itemize}
\item[(i)] $\|\pi_{\mathbf{u},\mathbf{1}} M_{[0,n_k)}\, \mathbf{l}(p)\| < \varepsilon$ for all $p \in \mathcal{A}^*$ such that $p\hspace{.1em}j_0 = p\hspace{.1em}i_{n_k}$ is a prefix of~$\omega^{(n_k)}$, using~\eqref{uniform},
\item[(ii)] and $\|\pi_{\mathbf{u},\mathbf{1}} M_{[0,n_k)}\, \mathbf{l}(q)\| < \varepsilon$ for all $q \in Q$, using \eqref{eq:st2} and the fact that $Q$ is finite.
\end{itemize}

Finally, let $p$ be a prefix of~$\omega^{(n)}$, $n \ge  n_k+h$.
Choose $i \in \mathcal{A}$ in a way that $\sigma_{[n_k,n)}(p)\sigma_{[n_k,n_k+h)}(i) = \sigma_{[n_k,n)}(p)\sigma_{[0,h)}(i)$ is a prefix of~$\omega^{(n_k)}$.
Then $\sigma_{[n_k,n)}(p)\hspace{.1em}q\hspace{.1em}j_0 = \sigma_{[n_k,n)}(p)\hspace{.1em}q\hspace{.1em}i_{n_k}$ is a prefix of~$\omega^{(n_k)}$ for some $q \in Q$. Therefore, by (i) we have $\|\pi_{\mathbf{u},\mathbf{1}} M_{[0,n_k)} \mathbf{l}({q})\|< \varepsilon$ and (ii) implies that
\[
\|\pi_{\mathbf{u},\mathbf{1}} M_{[0,n)}\, \mathbf{l}(p) + \pi_{\mathbf{u},\mathbf{1}} M_{[0,n_k)}\, \mathbf{l}(q)\| = \|\pi_{\mathbf{u},\mathbf{1}} M_{[0,n_k)} \, \mathbf{l}(\sigma_{[n_k,n)}(p)\hspace{.1em}q)\| < \varepsilon,
\]
if $k \in K_\ell$ is sufficiently large.
Combining these two inequalities yields that $\|\pi_{\mathbf{u},\mathbf{1}} M_{[0,n)}\, \mathbf{l}(p)\| < 2\varepsilon$ for all prefixes~$p$ of~$\omega^{(n)}$, if $n \in \mathbb{N}$ is sufficiently large. 
As $\varepsilon$ was chosen arbitrary, this proves~\eqref{e:Lsn2} and thus the proposition.
\end{proof}

\begin{remark}
The assumption of algebraic irreducibility cannot be omitted in Proposition~\ref{p:strongconvergence}.
E.g., taking the primitive substitution $\sigma_n(1) = 121$, $\sigma_n(2) = 212$ for all~$n$, we have $M_{[0,n)} = \begin{pmatrix}2&1\\1&2\end{pmatrix}^n$, $\mathbf{u} = \tr(1,1)$, thus $\pi_{\mathbf{u},\mathbf{1}} M_{[0,n)} \mathbf{l}(1) = \tr(1/2,-1/2)$ and $\pi_{\mathbf{u},\mathbf{1}} M_{[0,n)} \mathbf{l}(2) = \tr(-1/2,1/2)$ for all~$n$; the limit words are the periodic words $1212\cdots$ and $2121\cdots$, hence, $\mathcal{L}_{\boldsymbol{\sigma}}$ is clearly balanced.
\end{remark}

\section{Set equations for Rauzy fractals and the recurrent left eigenvector}\label{sec:dual}
The classical Rauzy fractal associated with a unimodular Pisot substitution $\sigma$ can be defined in terms of the dual substitution $E_1^*(\sigma)$ given in~\eqref{eq:dualsubst}. This dual substitution acts on the discrete hyperplane~$\Gamma(\mathbf{v})$ of the stable hyperplane~$\mathbf{v}^\bot$ of~$\sigma$; cf.\ e.g.~\cite{Arnoux-Ito:01}. Carrying this over to a sequence~$\boldsymbol{\sigma}\in S^{\mathbb{N}}$ requires considering an infinite sequence of hyperplanes~$(\mathbf{w}^{(n)})^\bot$, where, for each $n \in \mathbb{N}$, the dual substitution $E_1^*(\sigma_n)$ of $\sigma_n$ maps $\Gamma(\mathbf{w}^{(n)})$ to~$\Gamma(\mathbf{w}^{(n+1)})$. In Section~\ref{sec:dualsubst}, we formalize these concepts and relate them to the Rauzy fractals defined in Section~\ref{subsec:FR}. We  first define Rauzy fractals on any hyperplane~$\mathbf{w}^\bot$, $\mathbf{w} \in \mathbb{R}_{\ge0}^d \setminus \{\mathbf{0}\}$, in order to obtain set equations that reflect the combinatorial properties of $S$-adic words geometrically.
In Section \ref{subsec:choicev}, we specify the vector~$\mathbf{w}$ by defining a ``recurrent left eigenvector''~$\mathbf{v}$. This vector will allow us to obtain an associated sequence of hyperplanes $(\mathbf{v}^{(n)})^\bot$ such that the Rauzy fractals defined on these hyperplanes converge w.r.t.\ the Hausdorff metric; see Proposition~\ref{p:close}. It is this convergence property that will later enable us to derive topological as well as tiling properties of our ``$S$-adic Rauzy fractals''.

Throughout this section we assume that $S$ is a finite or  infinite set of unimodular substitutions over the finite alphabet~$\mathcal{A}$ and $\boldsymbol{\sigma} \in S^{\mathbb{N}}$ is a directive sequence.

\subsection{The dual substitution and set equations}\label{sec:dualsubst}  We now give some properties of the dual substitution $E_1^*(\sigma)$ defined in \eqref{eq:dualsubst}. Let  $\mathbf{u}$ be a generalized right eigenvector, $\mathbf{w} \in \mathbb{R}_{\ge0}^d \setminus \{\mathbf{0}\}$. 
To simplify notation, we use the abbreviations
\begin{equation} \label{e:projabb}
\pi_{\mathbf{u},\mathbf{w}}^{(n)} = \pi_{(M_{[0,n)})^{-1}\mathbf{u},\tr{(M_{[0,n)})}\,\mathbf{w}} \qquad (n \in \mathbb{N}).
\end{equation}
Note that $\pi_{\mathbf{u},\mathbf{w}}^{(0)} =\pi_{\mathbf{u},\mathbf{w}}$.
Moreover, we set
\[
\mathbf{w}^{(n)} = \tr(M_{[0,n)})\, \mathbf{w} \qquad (n \in \mathbb{N}).
\]

The dual substitution~$E_1^*(\sigma)$ can be extended to subsets of discrete hyperplanes in the obvious way. 
Moreover, by direct calculation, one obtains that $E_1^*(\sigma \tau) = E_1^*(\tau) E_1^*(\sigma)$; cf.~\cite{Arnoux-Ito:01}. 
The following lemma contains further relevant properties of~$E_1^*$.

\begin{lemma} \label{l:e1star}
Let $\boldsymbol{\sigma} = (\sigma_n)\in S^{\mathbb{N}}$ be a sequence of unimodular substitutions. Then for all $k < \ell$, we~have
\renewcommand{\theenumi}{\roman{enumi}}
\begin{enumerate}
\itemsep1ex
\item \label{62i}
$M_{[k,\ell)}\, (\mathbf{w}^{(\ell)})^\bot = (\mathbf{w}^{(k)})^\bot$,
\item \label{62ii}
$E_1^*(\sigma_{[k,\ell)})\, \Gamma(\mathbf{w}^{(k)}) = \Gamma(\mathbf{w}^{(\ell)})$,
\item \label{62iii}
for distinct $[\mathbf{x},i], [\mathbf{x}',i'] \in \Gamma(\mathbf{w}^{(k)})$, the sets $E_1^*(\sigma_{[k,\ell)})[\mathbf{x},i]$ and $E_1^*(\sigma_{[k,\ell)})[\mathbf{x}',i']$ are disjoint.
\end{enumerate}
\end{lemma}

\begin{proof}
The first assertion follows directly from the fact that $\mathbf{w}^{(\ell)} = \tr{(M_{[k,\ell)})}\, \mathbf{w}^{(k)}$.
By the same fact, the other assertions are special cases of \cite[Theorem~1]{Fernique:06}.
\end{proof}

We need the following auxiliary result on the projections~$\pi_{\mathbf{u},\mathbf{w}}^{(n)}$.

\begin{lemma}\label{lem:projections} 
Let $\boldsymbol{\sigma} = (\sigma_n)\in S^{\mathbb{N}}$ be a sequence of unimodular substitutions. Then for all $n\in \mathbb{N}$, we~have
\begin{equation*}
\pi_{\mathbf{u},\mathbf{w}}^{(n)}\, M_n = M_n\, \pi_{\mathbf{u},\mathbf{w}}^{(n+1)}.
\end{equation*}
\end{lemma}

\begin{proof}
Consider the linear mapping $M_n^{-1} \pi_{\mathbf{u},\mathbf{w}}^{(n)}\, M_n$. This mapping is idempotent, its kernel is $M_n^{-1} \mathbb{R}\, (M_{[0,n)})^{-1} \mathbf{u} = \mathbb{R}\, (M_{[0,n+1)})^{-1} \mathbf{u}$, and by Lemma~\ref{l:e1star}~(\ref{62i}) its image is~$(\mathbf{w}^{(n+1)})^\bot$. Thus $M_n^{-1} \pi_{\mathbf{u},\mathbf{w}}^{(n)}\, M_n$ is the projection to~$(\mathbf{w}^{(n+1)})^\bot$ along $(M_{[0,n+1)})^{-1} \mathbf{u}$.
\end{proof}

The following lemma gives an alternative definition of~$\mathcal{R}(i)$.

\begin{lemma}\label{ref:alterdef}
Let $\boldsymbol{\sigma}=(\sigma_n)_{n\in\mathbb{N}}\in S^{\mathbb{N}}$ be a primitive, algebraically irreducible, and recurrent sequence of  unimodular substitutions with balanced language~$\mathcal{L}_{\boldsymbol{\sigma}}$.
For each $i \in \mathcal{A}$ we have
\[
\mathcal{R}(i) = \lim_{n\to\infty} \pi_{\mathbf{u},\mathbf{1}} M_{[0,n)}\, E_1^*(\sigma_{[0,n)})[\mathbf{0},i],
\]
where each $[\mathbf{y},j] \in E_1^*(\sigma_{[0,n)})[\mathbf{0},i]$ is identified with its first component $\mathbf{y} \in \mathbb{Z}^d$ and the limit is taken with respect to the Hausdorff metric.
\end{lemma}

\begin{proof}
By the definition of $E_1^*(\sigma_{[0,n)})$ in~\eqref{eq:dualsubst}, we have
\[
\pi_{\mathbf{u},\mathbf{1}} M_{[0,n)}\, E_1^*(\sigma_{[0,n)})[\mathbf{0},i] = \{\pi_{\mathbf{u},\mathbf{1}}\, \mathbf{l}(p):\, p \in \mathcal{A}^*,\ \mbox{$p\hspace{.1em}i$ is a prefix of $\sigma_{[0,n)}(j)$ for some $j \in \mathcal{A}$}\}.
\]
If $p\hspace{.1em}i$ is a prefix of a limit word, we have thus $\pi_{\mathbf{u},\mathbf{1}}\, \mathbf{l}(p) \in \pi_{\mathbf{u},\mathbf{1}} M_{[0,n)}\, E_1^*(\sigma_{[0,n)})[\mathbf{0},i]$ for all sufficiently large~$n$, hence $\mathcal{R}(i) \subset \lim_{n\to\infty} \pi_{\mathbf{u},\mathbf{1}} M_{[0,n)}\, E_1^*(\sigma_{[0,n)})[\mathbf{0},i]$.

On the other hand, choose a limit word~$\omega$.
Then for each~$n$ and for each $j \in \mathcal{A}$, there is a prefix~$p$ of $\omega^{(n)}$ such that $\omega$ starts with $\sigma_{[0,n)}({p}\hspace{.1em}j)$.
Since $\|\pi_{\mathbf{u},\mathbf{1}} \mathbf{l}(\sigma_{[0,n)}({p}))\|$ is small for large~$n$ by Proposition~\ref{p:strongconvergence}, we obtain that $\pi_{\mathbf{u},\mathbf{1}} M_{[0,n)}\, E_1^*(\sigma_{[0,n)})[\mathbf{0},i]$ is close to $\mathcal{R}(i)$ for large~$n$.
\end{proof}

We now associate with a directive sequence $\boldsymbol{\sigma} = (\sigma_n)$ a sequence of Rauzy fractals~$\mathcal{R}^{(n)}_\mathbf{w}$ obtained by taking projections of each ``desubstituted'' limit word~$\omega^{(n)}$  to $(\tr{(M_{[0,n)})}\, \mathbf{w})^\bot$ along the direction $(M_{[0,n)})^{-1} \mathbf{u}$, which is the generalized right eigenvector of the shifted sequence $(\sigma_{m+n})_{m\in\mathbb{N}}$.
 
For $\mathbf{w} \in \mathbb{R}_{\ge0}^d \setminus \{\mathbf{0}\}$, let $\mathcal{R}^{(n)}_\mathbf{w} = \bigcup_{i\in\mathcal{A}} \mathcal{R}^{(n)}_\mathbf{w}(i)$ with
\begin{equation} \label{e:defRn}
\mathcal{R}^{(n)}_\mathbf{w}(i) = \overline{\{\pi_{\mathbf{u},\mathbf{w}}^{(n)}\, \mathbf{l}(p):\, p \in \mathcal{A}^*,\ \mbox{$p\hspace{.1em}i$ is a prefix of $\omega^{(n)}$},\ \mbox{$\sigma_{[0,n)}(\omega^{(n)})$ is a limit word of $\boldsymbol{\sigma}$}\}}.
\end{equation}
Note that $\mathcal{R}^{(0)}_\mathbf{w}(i) = \mathcal{R}_\mathbf{w}(i)$. With the above notation, $\mathcal{R}^{(n)}_\mathbf{w}$ lives on the hyperplane~$(\mathbf{w}^{(n)})^\bot$.

Similarly to Lemma~\ref{l:bounded}, we can give explicit bounds for these subtiles.

\begin{lemma} \label{l:convhull2}
Let $\mathbf{w} \in \mathbb{R}_{\ge 0}^d \setminus \{\mathbf{0}\}$.
If $\mathcal{L}_{\boldsymbol{\sigma}}^{(n)}$ is $C$-balanced, then $\mathcal{R}^{(n)}_\mathbf{w} \subset \pi_{\mathbf{u},\mathbf{w}}^{(n)} \big([-C,C]^d \cap \mathbf{1}^\bot\big)$.
\end{lemma}

\begin{proof}
By Lemma~\ref{l:bounded}, we have $\pi_{(M_{[0,n)})^{-1}\mathbf{u},\mathbf{1}}\, \mathcal{R}^{(n)}_\mathbf{w} \subset [-C,C]^d \cap \mathbf{1}^\bot$.
Projecting by~$\pi_{\mathbf{u},\mathbf{w}}^{(n)}$, we obtain the result.
\end{proof}

The following lemma shows that the Rauzy fractals~$\mathcal{R}^{(n)}_\mathbf{w}$ mapped back via $M_{[0,n)}$ to the representation space~$\mathbf{w}^\bot$ tend to be smaller and smaller.

\begin{lemma} \label{l:smallsubtiles}
Let $\boldsymbol{\sigma} = (\sigma_n)\in S^{\mathbb{N}}$ be a primitive, algebraically irreducible, and recurrent sequence of unimodular substitutions with balanced language $\mathcal{L}_{\boldsymbol{\sigma}}$, and let $\mathbf{w} \in \mathbb{R}_{\ge 0}^d \setminus \{\mathbf{0}\}$. Then 
\[
\lim_{n\to\infty} M_{[0,n)} \mathcal{R}^{(n)}_\mathbf{w} = \{\mathbf{0}\}.
\]
\end{lemma}

\begin{proof}
As $M_{[0,n)} \pi_{\mathbf{u},\mathbf{w}}^{(n)} = \pi_{\mathbf{u},\mathbf{w}}\, M_{[0,n)}$ by Lemma~\ref{lem:projections} and $\pi_{\mathbf{u},\mathbf{w}} = \pi_{\mathbf{u},\mathbf{w}}\, \pi_{\mathbf{u},\mathbf{1}}$, we conclude that $M_{[0,n)} \pi_{\mathbf{u},\mathbf{w}}^{(n)}\, \mathbf{l}(p) = \pi_{\mathbf{u},\mathbf{w}}\, \pi_{\mathbf{u},\mathbf{1}}\, M_{[0,n)}\, \mathbf{l}({p})$ for all prefixes~$p$ of~$\omega^{(n)}$. 
Now, the result follows from Proposition~\ref{p:strongconvergence}.
\end{proof}

For the Rauzy fractals $\mathcal{R}^{(n)}_\mathbf{w}$, we obtain a hierarchy of set equations, which replaces the self-affine structure present in the periodic case. As $\mathcal{R}^{(n)}_\mathbf{w}$ lives on the  hyperplane~$(\mathbf{w}^{(n)})^\bot$, the decomposition below involves Rauzy fractals living in different hyperplanes.

\begin{proposition} \label{p:setequation}
Let $\boldsymbol{\sigma} = (\sigma_n)\in S^{\mathbb{N}}$ be a sequence of unimodular substitutions with generalized right eigenvector~$\mathbf{u}$.
Then for each $[\mathbf{x},i]\in\mathbb{Z}^d\times \mathcal{A}$ and all $k < \ell$, we have the set equation
\begin{equation}\label{e:setequationkl}
\pi_{\mathbf{u},\mathbf{w}}^{(k)}\, \mathbf{x}+ \mathcal{R}^{(k)}_\mathbf{w}(i) = \bigcup_{[\mathbf{y},j] \in E_1^*(\sigma_{[k,\ell)})[\mathbf{x},i]} M_{[k,\ell)}\big(\pi_{\mathbf{u},\mathbf{w}}^{(\ell)}\, \mathbf{y} +  \mathcal{R}^{(\ell)}_\mathbf{w}(j)\big).
\end{equation}
\end{proposition}

\begin{proof}
Let $\omega$ be a limit word.
Each prefix~$p$ of~$\omega^{(k)}$ has a unique decomposition $p = \sigma_{[k,\ell)}(\tilde{p})\, q$ with $\tilde{p}\hspace{.1em}j$ a prefix of~$\omega^{(\ell)}$ and $q$ a proper prefix of $\sigma_{[k,\ell)}(j)$. 
Since $\mathbf{l}(\sigma_{[k,\ell)}(\tilde{p})) = M_{[k,\ell)}\, \mathbf{l}(\tilde{p})$, Lemma~\ref{lem:projections} implies that 
$\pi_{\mathbf{u},\mathbf{w}}^{(k)}\, \mathbf{l}(p) = \pi_{\mathbf{u},\mathbf{w}}^{(k)}\, \mathbf{l}(q) +  M_{[k,\ell)}\, \pi_{\mathbf{u},\mathbf{w}}^{(\ell)}\, \mathbf{l}(\tilde{p})$. 
We gain that
\[
\big\{ \pi_{\mathbf{u},\mathbf{w}}^{(k)}\, \mathbf{l}(p):\, \mbox{$p\hspace{.1em}i$ is a prefix of $\omega^{(k)}$}\big\} = \hspace{-1em} \bigcup_{\substack{q\in\mathcal{A}^*,\, j\in \mathcal{A}:\\[.2ex] \sigma_{[k,\ell)}(j)\in qi\mathcal{A}^*}} \hspace{-1em} \pi_{\mathbf{u},\mathbf{w}}^{(k)}\, \mathbf{l}(q) + M_{[k,\ell)}\, \big\{ \pi_{\mathbf{u},\mathbf{w}}^{(\ell)}\, \mathbf{l}(\tilde{p}):\, \mbox{$\tilde{p}\hspace{.1em}j$ is a prefix of $\omega^{(\ell)}$}\big\}.
\]
By the definition of $E_1^*(\sigma_{[k,\ell)})$, taking closures and translating by $\pi_{\mathbf{u},\mathbf{w}}^{(k)}\, \mathbf{x}$ yields the result.
\end{proof}

\subsection{Recurrent left eigenvector } \label{subsec:choicev}

In the case of a single substitution~$\sigma$, choosing $\mathbf{w} = \mathbf{v}$, where $\tr{\mathbf{v}}$ is the Perron-Frobenius left eigenvector of~$M_\sigma$, the set equations give a graph-directed iterated function system for the subtiles $\mathcal{R}_\mathbf{v}(i)$; see~\cite{CANTBST}. 
For $\boldsymbol{\sigma} = (\sigma_n)\in S^{\mathbb{N}}$, the Rauzy fractals~$\mathcal{R}^{(n)}_\mathbf{w}$ are different from~$\mathcal{R}^{(0)}_\mathbf{w}$ and even live on different hyperplanes~$(\mathbf{w}^{(n)})^\bot$. Thus, in general \eqref{e:setequationkl} is an infinite system of set equations. Also, the construction of an analog of the left Perron-Frobenius eigenvector needs some work. 
Contrary to the cones $M_{[0,n)}\, \mathbb{R}_+^d$, there is no reason for the cones $\tr{(M_{[0,n)})}\, \mathbb{R}_+^d$ to be nested.
Therefore, the intersection of these cones does not define a generalized left eigenvector of~$\boldsymbol{\sigma}$ and cannot be used to get a stable space. 
However, for a suitable choice of~$\mathbf{v}$, we have a subsequence $(n_k)_{k\in\mathbb{N}}$ such that the directions of $\mathbf{v}^{(n_k)} = \tr{(M_{[0,n_k)})}\, \mathbf{v}$ tend to that of~$\mathbf{v}$; in this case $\mathbf{v}$ is called a \emph{recurrent left eigenvector}.
Using the assumptions of Theorem~\ref{t:1}, we can even guarantee that $\mathcal{R}^{(n_k)}_\mathbf{v}$ converges to~$\mathcal{R}_\mathbf{v}$ in Hausdorff limit for a suitable choice of $(n_k)$. 

The following lemma shows that, under the assumptions of primitivity and recurrence, one can easily exhibit recurrent left eigenvectors~$\mathbf{v}$. The precise statement involving a subsequence of a given sequence $(n_k)_{k\in\mathbb{N}}$ will be useful in the proof of Lemma~\ref{lem:th1star}.

\begin{lemma} \label{l:findv}
Let $\boldsymbol{\sigma} = (\sigma_n)\in S^{\mathbb{N}}$ be a primitive and recurrent sequence of substitutions and $(n_k)$ a strictly increasing sequence of non-negative integers.
Then there is $\mathbf{v} \in \mathbb{R}_{\ge0}^d \setminus \{\mathbf{0}\}$ such that 
\begin{equation}\label{eq:recurrentcandidate}
\lim_{k\in K,\,k\to\infty} \frac{\mathbf{v}^{(n_k)}}{\|\mathbf{v}^{(n_k)}\|} = \lim_{k\in K,\,k\to\infty}\frac{\tr{(M_{[0,n_k)})}\mathbf{v}}{\|\tr{(M_{[0,n_k)})}\mathbf{v}\|}= \mathbf{v}
\end{equation} 
for some infinite set $K \subset \mathbb{N}$. Such a vector $\mathbf{v}$ is called a \emph{recurrent left eigenvector}.
\end{lemma}

\begin{proof}
Recall that a non-negative matrix is non-expanding and positive matrix is contractive w.r.t.\ the Hilbert metric on the projective space~$\mathbb{P}(\mathbb{R}^{d-1})$ (see \cite{Birkhoff} or \cite[Appendix~A]{Fisher:09} for details on this metric). Thus primitivity and recurrence imply that the diameter of the cones $\tr{(M_{[0,n_k)})}\, \mathbb{R}_{\ge0}^d$ converges to zero. By the compactness of $\mathbb{P}(\mathbb{R}^{d-1})$, we can choose an infinite set $K \subset \mathbb{N}$ such that $\bigcap_{k\in K} \tr{(M_{[0,n_k)})}\, \mathbb{R}^d_{\ge 0} = \mathbb{R}_{\ge0} \mathbf{v}$ for some $\mathbf{v} \in \mathbb{R}_{\ge0}^d \setminus \{\mathbf{0}\}$. For this choice of $\mathbf{v}$, \eqref{eq:recurrentcandidate} holds.
\end{proof}

In the sequel, we will work with directive sequences that satisfy a list of conditions gathered in the following Property PRICE (which stands for Primitivity, Recurrence, algebraic Irreducibility, $C$-balancedness, and recurrent left Eigenvector). 
By Lemma~\ref{lem:th1star} below, this property is a consequence of the assumptions of Theorem~\ref{t:1}. 
Nevertheless, we prefer referring to Property PRICE because we will frequently use the sequences $(n_k)$, $(\ell_k)$, and the recurrent left eigenvector~$\mathbf{v}$ involved in the definition.

\begin{definition}[Property PRICE]\label{def:star}
Assume that $S$ is a finite or  infinite  set of substitutions over the finite alphabet~$\mathcal{A}$.
We say that a directive sequence $\boldsymbol{\sigma} = (\sigma_n)\in S^{\mathbb{N}}$ has Property \emph{PRICE} w.r.t.\ the strictly increasing sequences $(n_k)_{k\in\mathbb{N}}$ and $(\ell_k)_{k\in\mathbb{N}}$ and a vector $\mathbf{v} \in \mathbb{R}_{\ge0}^d \setminus \{\mathbf{0}\}$ if the following conditions hold\footnote{Note that (P) doesn't merely mean that the matrix $B$ occurs infinitely often but that it has to occur at the end of the recurring blocks defined in (R). So (P) doesn't follow from (R).}.
\begin{itemize}
\labitem{(P)}{defP}
There exists $h \in \mathbb{N}$ and a positive matrix~$B$ such that $M_{[\ell_k-h,\ell_k)} = B$ for all $k \in \mathbb{N}$.
\labitem{(R)}{defR}
We have $(\sigma_{n_k}, \sigma_{n_k+1}, \ldots,\sigma_{n_k+\ell_k-1}) = (\sigma_0, \sigma_1, \ldots,\sigma_{\ell_k-1})$ for all $k\in\mathbb{N}$.
\labitem{(I)}{defI}
The directive sequence~$\boldsymbol{\sigma}$ is algebraically irreducible.
\labitem{(C)}{defC}
There is $C > 0$ such that $\mathcal{L}_{\boldsymbol{\sigma}}^{(n_k+\ell_k)}$ is $C$-balanced for all $k\in\mathbb{N}$.
\labitem{(E)}{defE}
We have $\lim_{k\to\infty} \mathbf{v}^{(n_k)}/\|\mathbf{v}^{(n_k)}\|= \mathbf{v}$.
\end{itemize}
We also simply say that $\boldsymbol{\sigma}$ satisfies Property PRICE if the five conditions hold for some not explicitly specified strictly increasing sequences $(n_k)_{k\in\mathbb{N}}$ and $(\ell_k)_{k\in\mathbb{N}}$ and some $\mathbf{v} \in \mathbb{R}_{\ge0}^d \setminus \{\mathbf{0}\}$.
\end{definition}

Note that Properties~\ref{defP}, \ref{defR} and~\ref{defC} in Definition~\ref{def:star} imply that $\boldsymbol{\sigma}$ is a primitive and recurrent directive sequence with balanced language~$\mathcal{L}_{\boldsymbol{\sigma}}$, and \ref{defE} means that $\mathbf{v}$ is a recurrent left eigenvector.

The conditions of the following lemma are (apart from unimodularity, which we do not need here) that of Theorem~\ref{t:1}.

\begin{lemma} \label{lem:th1star}
Let $\boldsymbol{\sigma} = (\sigma_n)\in S^{\mathbb{N}}$ be a primitive and algebraically irreducible sequence of substitutions over the finite alphabet~$\mathcal{A}$. Assume that there is $C > 0$ such that for each $\ell \in \mathbb{N}$, there is $n \ge 1$ with $(\sigma_n, \ldots, \sigma_{n+\ell-1}) = (\sigma_0, \ldots, \sigma_{\ell-1})$ and $\mathcal{L}_{\boldsymbol{\sigma}}^{(n+\ell)}$ is $C$-balanced.
Then Property PRICE holds. 
\end{lemma}

\begin{proof}
First observe that~\ref{defI} holds by assumption. By primitivity of~$\boldsymbol{\sigma}$, we can choose $\ell_0$ and~$h$ in a way that $M_{[\ell_0-h,\ell_0)}$ is positive. 
As the assumptions of the lemma imply that $\boldsymbol{\sigma}$ is recurrent, there exists a strictly increasing sequence $(\ell_k)$ of non-negative integers such that \ref{defP} holds. By assumption, there is an associated sequence $(n_k)$ of non-negative integers such that \ref{defR} and \ref{defC} hold. In view of Lemma~\ref{l:findv}, we can choose appropriate subsequences of $(\ell_k)$ and $(n_k)$, again called $(\ell_k)$ and~$(n_k)$, such that \ref{defE} holds. As taking subsequences doesn't affect \ref{defP}, \ref{defR}, \ref{defI}, and~\ref{defC}, this proves the lemma.
\end{proof}

We will use the following simple observation. 

\begin{lemma} \label{l:shiftedPRICE}
Assume that the directive sequence $\boldsymbol{\sigma} = (\sigma_n)_{n\in\mathbb{N}}\in S^{\mathbb{N}}$ has Property PRICE w.r.t.\ the sequences $(n_k)_{k\in\mathbb{N}}$ and $(\ell_k)_{k\in\mathbb{N}}$ and the vector~$\mathbf{v}$.  
Then for each $h \in \mathbb{N}$ there is $k_0 \in \mathbb{N}$ such that the shifted sequence $(\sigma_{n+h})_{n\in\mathbb{N}}$ has Property PRICE w.r.t.\ the sequences $(n_{k+k_0})_{k\in\mathbb{N}}$ and $(\ell_{k+k_0}{-}h)_{k\in\mathbb{N}}$, and the vector~$\mathbf{v}^{(h)}$.  
\end{lemma}

Property PRICE implies the following uniform convergence result for the projections~$\pi_{\mathbf{u},\mathbf{v}}^{(n_k)}$.

\begin{lemma}\label{lem:projectionconvergence}
Assume that the directive sequence $\boldsymbol{\sigma}\in S^{\mathbb{N}}$ has Property PRICE w.r.t.\ the sequences $(n_k)$ and $(\ell_k)$ and the vector~$\mathbf{v}$.  
Then
\[
\lim_{k\to\infty} \max \big\{\|\pi_{\mathbf{u},\mathbf{v}}^{(n_k)}\, \mathbf{x} - \pi_{\mathbf{u},\mathbf{v}}\, \mathbf{x}\|:\, \|\mathbf{x}\| \le \max _{i\in \mathcal{A}} \|M_{[0,\ell_k)} \mathbf{e}_i\|,\, \|\pi_{\mathbf{u},\mathbf{v}}\, \mathbf{x}\| \le 1\big\} = 0.
\]
In particular, $\pi_{\mathbf{u},\mathbf{v}}^{(n_k)} \to \pi_{\mathbf{u},\mathbf{v}}$ for $k\to\infty$ in compact-open topology.
\end{lemma}

\begin{proof}
Since \ref{defP}, \ref{defR}, \ref{defI}, and \ref{defC} hold, we obtain from Proposition~\ref{p:strongconvergence} that $\|\pi_{\mathbf{u},\mathbf{1}} M_{[0,\ell_k)} \mathbf{e}_i\| \to 0$ for each $i \in \mathcal{A}$ when $k\to\infty$. Since $\pi_{\mathbf{u},\mathbf{v}} = \pi_{\mathbf{u},\mathbf{v}} \pi_{\mathbf{u},\mathbf{1}}$, this implies that $\|\pi_{\mathbf{u},\mathbf{v}} M_{[0,\ell_k)} \mathbf{e}_i\| \to 0$. As $M_{[\ell_k-h,\ell_k)} = B$ is a positive matrix (that does not depend on~$k$), there is $c > 0$ such that $\max_{i\in \mathcal{A}} \|M_{[0,\ell_k)} \mathbf{e}_i\| \le c\, \min_{i\in\mathcal{A}} \|M_{[0,\ell_k)} \mathbf{e}_i||$ for all $k \in \mathbb{N}$. Thus the cone $M_{[0,\ell_k)} \mathbb{R}_+^d$ has small diameter at ``height'' $\max_{i\in \mathcal{A}} \|M_{[0,\ell_k)}\mathbf{e}_i\|$, hence, $\pi_{\tilde{\mathbf{u}}, \mathbf{v}}\, \mathbf{x}$ is close to $\pi_{\mathbf{u},\mathbf{v}}\, \mathbf{x}$ for all $\tilde{\mathbf{u}} \in M_{[0,\ell_k)}\mathbb{R}_+^d$ and $\mathbf{x}$ in the cylinder $\|\mathbf{x}\| \le \max _{i\in \mathcal{A}} \|M_{[0,\ell_k)} \mathbf{e}_i\|$, $\|\pi_{\mathbf{u},\mathbf{v}}\, \mathbf{x}\| \le 1$. More precisely, 
\[
\pi_{\mathbf{u},\mathbf{v}}\, \mathbf{x} - \pi_{\tilde{\mathbf{u}},\mathbf{v}}\, \mathbf{x} = \pi_{\mathbf{u},\mathbf{v}}(\mathbf{x} - \pi_{\tilde{\mathbf{u}},\mathbf{v}}\, \mathbf{x}) = \pi_{\mathbf{u},\mathbf{v}}\bigg( \frac{\|\mathbf{x} - \pi_{\tilde{\mathbf{u}},\mathbf{v}}\, \mathbf{x}\|}{\|\tilde{\mathbf{u}}\|}\, \tilde{\mathbf{u}} \bigg) = \frac{\|\mathbf{x} - \pi_{\tilde{\mathbf{u}},\mathbf{v}}\, \mathbf{x}\|}{\|\tilde{\mathbf{u}}\|}\, \pi_{\mathbf{u},\mathbf{v}}\, \tilde{\mathbf{u}} 
\]
gives that 
\[
||\pi_{\mathbf{u},\mathbf{v}}\, \mathbf{x} - \pi_{\tilde{\mathbf{u}},\mathbf{v}}\, \mathbf{x}||
 \le \frac{\|\mathbf{x} - \pi_{\mathbf{u},\mathbf{v}}\, \mathbf{x}\| + \|\pi_{\mathbf{u},\mathbf{v}}\, \mathbf{x} - \pi_{\tilde{\mathbf{u}},\mathbf{v}}\, \mathbf{x}\|}{\|\tilde{\mathbf{u}}\|}\,|| \pi_{\mathbf{u},\mathbf{v}}\, \tilde{\mathbf{u}}||
\]
and, hence,
\[
\big\|\pi_{\mathbf{u},\mathbf{v}}\, \mathbf{x} - \pi_{\tilde{\mathbf{u}},\mathbf{v}}\, \mathbf{x}\big\| \le \frac{\|\mathbf{x} - \pi_{\mathbf{u},\mathbf{v}}\, \mathbf{x}\|}{\|\tilde{\mathbf{u}}\| - \|\pi_{\mathbf{u},\mathbf{v}}\, \tilde{\mathbf{u}}\|}\, \|\pi_{\mathbf{u},\mathbf{v}}\, \tilde{\mathbf{u}}\|.
\]
Thus we obtain for $\tilde{\mathbf{u}}$ and $\mathbf{x}$ with the above properties that for each $\varepsilon > 0$
\begin{align*}
\big\|\pi_{\mathbf{u},\mathbf{v}}\, \mathbf{x} - \pi_{\tilde{\mathbf{u}},\mathbf{v}}\, \mathbf{x}\big\| & \le \frac{\max _{i\in \mathcal{A}} \|M_{[0,\ell_k)} \mathbf{e}_i\|+1}{\min _{i\in \mathcal{A}} \|M_{[0,\ell_k)} \mathbf{e}_i\|-\max _{i\in \mathcal{A}} \|\pi_{\mathbf{u},\mathbf{v}} M_{[0,\ell_k)} \mathbf{e}_i\|}\, \max _{i\in \mathcal{A}} \|\pi_{\mathbf{u},\mathbf{v}} M_{[0,\ell_k)} \mathbf{e}_i\| \\
& \le 2c\, \max _{i\in \mathcal{A}} \|\pi_{\mathbf{u},\mathbf{v}} M_{[0,\ell_k)} \mathbf{e}_i\| < \varepsilon,
\end{align*}
holds for sufficiently large~$k$.
Moreover, the facts that $\lim_{k\to\infty} \mathbf{v}^{(n_k)}/\|\mathbf{v}^{(n_k)}\|= \mathbf{v}$, that $\|\pi_{\tilde{\mathbf{u}},\mathbf{v}}\, \mathbf{x}\|$ is bounded (by $1+\varepsilon$), and that $\langle \tilde{\mathbf{u}}, \mathbf{v} \rangle$ is bounded away from~$0$, yield that
\[
\big\|\pi_{\tilde{\mathbf{u}},\mathbf{v}^{(n_k)}}\, \mathbf{x} - \pi_{\tilde{\mathbf{u}},\mathbf{v}}\, \mathbf{x}\big\| < \varepsilon
\]
for sufficiently large~$k$, thus $\|\pi_{\tilde{\mathbf{u}},\mathbf{v}^{(n_k)}}\, \mathbf{x} - \pi_{\mathbf{u},\mathbf{v}}\, \mathbf{x}\| < 2 \varepsilon$.
We can choose $\tilde{\mathbf{u}} = (M_{[0,n_k)})^{-1} \mathbf{u}$ because the recurrence assertion~\ref{defR} gives $(M_{[0,n_k)})^{-1} \mathbf{u} \in M_{[0,\ell_k)} \mathbb{R}_+^d$.
As $\pi_{\mathbf{u},\mathbf{v}}^{(n_k)} = \pi_{(M_{[0,n_k)})^{-1} \mathbf{u},\mathbf{v}^{(n_k)}}$, this proves the lemma.
\end{proof}

We can now prove the following convergence result for Rauzy fractals. 

\begin{proposition} \label{p:close}
Let $S$ be a finite or  infinite  set of unimodular substitutions over a finite alphabet $\mathcal{A}$.
Assume that the sequence $\boldsymbol{\sigma} = (\sigma_n)\in S^{\mathbb{N}}$ of unimodular substitutions has Property PRICE w.r.t.\ the sequences $(n_k)$ and $(\ell_k)$ and the vector~$\mathbf{v}$. 
Then, for each $i \in \mathcal{A}$ and $\ell\in\mathbb{N}$, 
\begin{equation} \label{e:tilesconvergence}
\lim_{k\to\infty} \mathcal{R}^{(n_k+\ell)}_\mathbf{v}(i) = \mathcal{R}^{(\ell)}_\mathbf{v}(i),
\end{equation}
where the limit is taken w.r.t.\ the Hausdorff metric.
\end{proposition}

\begin{proof}
We first prove the result for $\ell=0$.
For each $\varepsilon > 0$ and each sufficiently large $k \in \mathbb{N}$, the following inequalities hold:
\renewcommand{\theenumi}{\roman{enumi}}
\begin{enumerate}
\itemsep1ex
\item \label{i:close1}
$\mathrm{diam} \big(M_{[0,\ell_k)} \mathcal{R}^{(\ell_k)}_\mathbf{v}(j)\big) < \varepsilon$ for each $j \in \mathcal{A}$,
\item \label{i:close2}
$\mathrm{diam} \big(M_{[0,\ell_k)} \mathcal{R}^{(n_k+\ell_k)}_\mathbf{v}(j)\big) < \varepsilon$ for each $j \in \mathcal{A}$,
\item \label{i:close3}
$\|\pi_{\mathbf{u},\mathbf{v}}^{(n_k)}\, M_{[0,\ell_k)}\, \mathbf{x} - \pi_{\mathbf{u},\mathbf{v}}\, M_{[0,\ell_k)}\, \mathbf{x}\| < \varepsilon$ for each $[\mathbf{x},j] \in E_1^*(\sigma_{[0,\ell_k)})[\mathbf{0},i]$.
\end{enumerate}
Inequality~(\ref{i:close1}) follows from Lemma~\ref{l:smallsubtiles}. To prove~(\ref{i:close2}), note first that, as $\mathcal{L}_{\boldsymbol{\sigma}}^{(n_k+\ell_k)}$ is $C$-balanced,
\[
M_{[0,\ell_k)} \mathcal{R}^{(n_k+\ell_k)}_\mathbf{v}(j) \subset M_{[0,\ell_k)}\, \pi_{\mathbf{u},\mathbf{v}}^{(n_k+\ell_k)} \big([-C,C]^d \cap \mathbf{1}^\bot\big) = \pi_{\mathbf{u},\mathbf{v}}^{(n_k)} M_{[0,\ell_k)} \big([-C,C]^d \cap \mathbf{1}^\bot\big)
\]
by  Lemmas~\ref{lem:projections} and~\ref{l:convhull2}.
For $\mathbf{y} \in M_{[0,\ell_k)}\, [-C,C]^d$ with sufficiently large~$k$, we have $\|\pi_{\mathbf{u},\mathbf{v}}\, \mathbf{y}\| < \varepsilon/2$ by Proposition~\ref{p:strongconvergence} and $\|\pi_{\mathbf{u},\mathbf{v}}^{(n_k)}\, \mathbf{y} - \pi_{\mathbf{u},\mathbf{v}}\, \mathbf{y}\| < \varepsilon/2$ by Lemma~\ref{lem:projectionconvergence}, where we have used that $\|\mathbf{y}\| \le C \sum_{j\in\mathcal{A}} \|M_{[0,\ell_k)}\, \mathbf{e}_j\|$. 
This implies that $\|\pi_{\mathbf{u},\mathbf{v}}^{(n_k)}\, \mathbf{y}\| < \varepsilon$, and (\ref{i:close2}) follows.
Finally, (\ref{i:close3}) is a consequence of Lemma~\ref{lem:projectionconvergence} because the definition of $E_1^*$ in~(\ref{eq:dualsubst}) yields for $[\mathbf{x},j] \in E_1^*(\sigma_{[0,\ell_k)})[\mathbf{0},i]$ that $M_{[0,\ell_k)}\, \mathbf{x} = \mathbf{l}(p)$ for some prefix~$p$ of $\sigma_{[0,\ell_k)}(j)$, $j \in \mathcal{A}$, hence $\|M_{[0,\ell_k)}\, \mathbf{x}\| \le \max_{j\in\mathcal{A}} \|M_{[0,\ell_k)}\, \mathbf{e}_j\|$ and $\|\pi_{\mathbf{u},\mathbf{v}}\, M_{[0,\ell_k)}\, \mathbf{x}\|$ is bounded by the balancedness of~$\mathcal{L}_{\boldsymbol{\sigma}}$.

By~\eqref{e:setequationkl}, we have
\[
\mathcal{R}_\mathbf{v}(i) = \bigcup_{[\mathbf{x},j] \in E_1^*(\sigma_{[0,\ell_k)})[\mathbf{0},i]} \big(\pi_{\mathbf{u},\mathbf{v}} M_{[0,\ell_k)}\, \mathbf{x} + M_{[0,\ell_k)} \mathcal{R}^{(\ell_k)}_\mathbf{v}(j)\big)
\]
and
\[
\mathcal{R}^{(n_k)}_\mathbf{v}(i) = \bigcup_{[\mathbf{x},j] \in E_1^*(\sigma_{[n_k,n_k+\ell_k)})[\mathbf{0},i]} \big(\pi_{\mathbf{u},\mathbf{v}}^{(n_k)}\, M_{[n_k,n_k+\ell_k)}\, \mathbf{x} + M_{[n_k,n_k+\ell_k)} \mathcal{R}^{(n_k+\ell_k)}_\mathbf{v}(j)\big).
\]
As $\sigma_{[n_k,n_k+\ell_k)} = \sigma_{[0,\ell_k)}$ and $M_{[n_k,n_k+\ell_k)} = M_{[0,\ell_k)}$, the result for the case $\ell=0$ now follows from (\ref{i:close1})--(\ref{i:close3}) by an application of the triangle inequality.

The case of $\ell > 0$ is equivalent to proving that $\lim_{k\to\infty} \mathcal{R}^{(n_k)}_{\mathbf{v}^{(\ell)}}(i) = \mathcal{R}_{\mathbf{v}^{(\ell)}}(i)$ for the Rauzy fractals defined by the shifted sequence $(\sigma_{n+\ell})_{n\in\mathbb{N}}$.
It is thus an immediate consequence of Lemma~\ref{l:shiftedPRICE}.
\end{proof}

\section{Some properties of Rauzy fractals} \label{sec:results}

In this section, we introduce the collections~$\mathcal{C}^{(n)}_\mathbf{w}$ of translates of $\mathcal{R}_\mathbf{w}^{(n)}(i)$, $i \in \mathcal{A}$, and prove  their covering properties. Moreover, we show that under certain conditions the set~$\mathcal{R}(i)$ is the closure of its interior and $\partial\mathcal{R}(i)$ has measure zero for each $i\in \mathcal{A}$; the proof of the latter property is the main task of this section. In the substitutive case, the proofs of the analogous results are based on the graph-directed iterated function system satisfied by the subtiles of the Rauzy fractal; see e.g.~\cite{CANTBST}. Since we do not have a graph-directed structure in our case, we rely on the infinite family of set equations in~\eqref{e:setequationkl}.

Again, throughout this section we assume that $S$ is a finite or  infinite set of unimodular substitutions over the finite alphabet~$\mathcal{A}$ and $\boldsymbol{\sigma} \in S^{\mathbb{N}}$ is a directive sequence.

\subsection{Covering properties}
For $\mathbf{w} \in \mathbb{R}^d_{\ge 0}\setminus\{\mathbf{0}\}$ and $n \in \mathbb{N}$, define the collection of tiles in $(\mathbf{w}^{(n)})^\bot$
\[
\mathcal{C}^{(n)}_\mathbf{w} = \{\pi_{\mathbf{u},\mathbf{w}}^{(n)}\, \mathbf{x} + \mathcal{R}^{(n)}_\mathbf{w}(i):\, [\mathbf{x},i] \in \Gamma(\mathbf{w}^{(n)})\}\,,
\]
where $\mathcal{R}^{(n)}_\mathbf{w}(i)$ are the Rauzy fractals defined in~\eqref{e:defRn} and $\pi_{\mathbf{u},\mathbf{w}}^{(n)}$ is as in~\eqref{e:projabb}.
Note that $\mathcal{C}^{(0)}_\mathbf{w} = \mathcal{C}_\mathbf{w}$. 

The following simple lemma will be used frequently in the sequel.

\begin{lemma} \label{l:covering}
Let $\boldsymbol{\sigma} = (\sigma_n)\in S^{\mathbb{N}}$ be a sequence of unimodular substitutions with generalized right eigenvector~$\mathbf{u}$, let $\mathbf{w} \in \mathbb{R}_{\ge0}^d \setminus \{\mathbf{0}\}$, and $k < \ell$.
If $\mathbf{z} \in (\mathbf{w}^{(k)})^\bot$ lies in $m$ distinct tiles of~$\mathcal{C}^{(k)}_\mathbf{w}$, then $(M_{[k,\ell)})^{-1} \mathbf{z}$ lies in at least~$m$ distinct tiles of~$\mathcal{C}^{(\ell)}_\mathbf{w}$.
If moreover there are distinct $[\mathbf{y},j], [\mathbf{y}',j'] \in E_1^*(\sigma_{[k,\ell)})[\mathbf{x},i]$, with $[\mathbf{x},i] \in \Gamma(\mathbf{w}^{(k)})$, such that $(M_{[k,\ell)})^{-1} \mathbf{z} \in {\big(\pi_{\mathbf{u},\mathbf{w}}^{(\ell)}\, \mathbf{y} + \mathcal{R}^{(\ell)}_\mathbf{w}(j)\big)} \cap {\big(\pi_{\mathbf{u},\mathbf{w}}^{(\ell)}\, \mathbf{y}' + \mathcal{R}^{(\ell)}_\mathbf{w}(j')\big)}$, then $(M_{[k,\ell)})^{-1} \mathbf{z}$ lies in at least $m+1$ distinct tiles of~$\mathcal{C}^{(\ell)}_\mathbf{w}$.
\end{lemma}

\begin{proof}
This is an immediate consequence of the set equations~\eqref{e:setequationkl}, the fact that $E_1^*(\sigma_{[k,\ell)})[\mathbf{x},i] \subset \Gamma(\mathbf{w}^{(\ell)})$ for $[\mathbf{x},i] \in \Gamma(\mathbf{w}^{(k)})$ by Lemma~\ref{l:e1star}~(\ref{62ii}) and that $E_1^*(\sigma_{[k,\ell)})[\mathbf{x},i] \cap E_1^*(\sigma_{[k,\ell)})[\mathbf{x}',i'] = \emptyset$ for distinct $[\mathbf{x},i], [\mathbf{x}',i'] \in \Gamma(\mathbf{w}^{(k)})$ by Lemma~\ref{l:e1star}~(\ref{62iii}).
\end{proof}

In particular, Lemma~\ref{l:covering} implies that the covering degree of~$\mathcal{C}^{(n)}_\mathbf{w}$ is less than or equal to that of~$\mathcal{C}^{(n+1)}_\mathbf{w}$, where the \emph{covering degree} of a collection of sets~$\mathcal{K}$ in a Euclidean space~$\mathcal{E}$ is the maximal number~$m$ such that each point of~$\mathcal{E}$ lies in at least $m$ distinct elements of~$\mathcal{K}$. (For locally finite multiple tilings, this agrees with the definition of the covering degree in Section~\ref{sec:multiple-tilings}.)

\begin{proposition} \label{p:covering}
Let $S$ be a finite or  infinite  set of unimodular substitutions over a finite alphabet and
let $\boldsymbol{\sigma} = (\sigma_n)_{n\in\mathbb{N}}\in S^{\mathbb{N}}$ be a primitive, algebraically irreducible, and recurrent directive sequence with balanced language~$\mathcal{L}_{\boldsymbol{\sigma}}$.
Then for each $n \in \mathbb{N}$ and $\mathbf{w} \in \mathbb{R}_{\ge0}^d \setminus \{\mathbf{0}\}$, the collection of tiles~$\mathcal{C}^{(n)}_\mathbf{w}$ covers~$(\mathbf{w}^{(n)})^\bot$ with finite covering degree. For fixed~$\mathbf{w}$, the covering degree of~$\mathcal{C}^{(n)}_\mathbf{w}$ increases monotonically with~$n$.
\end{proposition}

\begin{proof}
By the set equations~\eqref{e:setequationkl} and Lemma~\ref{l:e1star}~(\ref{62ii}), we have
\begin{equation} \label{e:Cv}
\bigcup_{\mathcal{T}\in\mathcal{C}_\mathbf{w}} \mathcal{T} = \bigcup_{[\mathbf{x},i] \in \Gamma(\mathbf{w})} \big(\pi_{\mathbf{u},\mathbf{w}}\, \mathbf{x} + \mathcal{R}_\mathbf{w}(i)\big) = \bigcup_{[\mathbf{x},i] \in \Gamma(\mathbf{w}^{(n)})} M_{[0,n)}\, \big(\pi_{\mathbf{u},\mathbf{w}}^{(n)}\, \mathbf{x} + \mathcal{R}^{(n)}_\mathbf{w}(i)\big)
\end{equation}
for each $n \in \mathbb{N}$.
Moreover, $\mathbf{w}^{(n)} = \tr{(M_{[0,n)})}\, \mathbf{w}$ and $M_{[0,n)}\, \mathbb{Z}^d = \mathbb{Z}^d$ (by unimodularity) imply that
\begin{align*}
\{M_{[0,n)}\, \pi_{\mathbf{u},\mathbf{w}}^{(n)}\, \mathbf{x}:\, [\mathbf{x},i] \in \Gamma(\mathbf{w}^{(n)})\} & = \{\pi_{\mathbf{u},\mathbf{w}}\, M_{[0,n)}\, \mathbf{x}:\, \mathbf{x} \in \mathbb{Z}^d,\, 0 \le \langle \mathbf{w}^{(n)}, \mathbf{x} \rangle < \max_{i\in\mathcal{A}} \langle \mathbf{w}^{(n)}, \mathbf{e}_i \rangle\} \\
& = \{\pi_{\mathbf{u},\mathbf{w}}\, \mathbf{y}:\, \mathbf{y} \in \mathbb{Z}^d,\, 0 \le \langle \mathbf{w}, \mathbf{y} \rangle < \max_{i\in\mathcal{A}} \langle \mathbf{w}, M_{[0,n)}\, \mathbf{e}_i \rangle\}.
\end{align*}
As $\mathbf{u}$ has rationally independent coordinates by Lemma~\ref{l:independent}, the set $\{\pi_{\mathbf{u},\mathbf{w}}\, \mathbf{y}:\, \mathbf{y} \in \mathbb{Z}^d,\, 0 \le \langle \mathbf{w}, \mathbf{y} \rangle\}$ is dense in $\mathbf{w}^\bot$. 
Observing that $\lim_{n\to\infty} \max_{i\in\mathcal{A}} \langle \mathbf{w}, M_{[0,n)}\, \mathbf{e}_i \rangle = \infty$ by the primitivity of $(\sigma_n)_{n\in\mathbb{N}}$, we obtain that 
\[
\lim_{n\to\infty} \{M_{[0,n)}\, \pi_{\mathbf{u},\mathbf{w}}^{(n)}\, \mathbf{x}:\, [\mathbf{x},i] \in \Gamma(\mathbf{w}^{(n)})\} = \overline{\{\pi_{\mathbf{u},\mathbf{w}}\, \mathbf{y}:\, \mathbf{y} \in \mathbb{Z}^d,\, 0 \le \langle \mathbf{w}, \mathbf{y} \rangle\}} = \mathbf{w}^\bot,
\]
where the limit is taken with respect to the Hausdorff metric. 
Since $\lim_{n\to\infty} M_{[0,n)} \mathcal{R}^{(n)}_\mathbf{w}(i) = \{\mathbf{0}\}$ by Lemma~\ref{l:smallsubtiles}, this implies together with \eqref{e:Cv} that $\overline{\bigcup_{\mathcal{T}\in\mathcal{C}_\mathbf{w}} \mathcal{T}} = \mathbf{w}^\bot$.
As $\mathcal{C}_\mathbf{w}$ is a locally finite collection of compact sets, this proves that $\mathcal{C}_\mathbf{w}$ covers~$\mathbf{w}^\bot$ and, hence, $\mathcal{C}^{(n)}_\mathbf{w}$ covers~$(\mathbf{w}^{(n)})^\bot$.

As $\pi_{\mathbf{u},\mathbf{w}}\, \Gamma(\mathbf{w})$ is uniformly discrete in~$\mathbf{w}^\bot$ and the elements of~$\mathcal{C}_\mathbf{w}$ are translations of the subtiles~$\mathcal{R}_\mathbf{w}(i)$, which are compact by Lemma~\ref{l:bounded}, $\mathcal{C}^{(0)}_\mathbf{w}$~has finite covering degree.
By Lemma~\ref{l:covering}, the covering degree of~$\mathcal{C}^{(n)}_\mathbf{w}$ is a monotonically increasing function in~$n$.
By the set equations~\eqref{e:setequationkl}, Lemma~\ref{l:e1star}~(\ref{62ii}) and the definition of~$E_1^*$ in~\eqref{eq:dualsubst}, we also see that the covering degree of~$\mathcal{C}^{(n+1)}_\mathbf{w}$ is bounded by $\max_{i\in\mathcal{A}} \sum_{j\in\mathcal{A}} |\sigma_n(j)|_i$ times the covering degree of~$\mathcal{C}^{(n)}_\mathbf{w}$.
\end{proof}

We also need the following result about locally finite compact coverings (its proof is easy).

\begin{lemma}\label{lem:intcovering}
Let $\mathcal{K}$ be a locally finite covering of~$\mathbb{R}^k$ by compact sets. If $\mathcal{K}$ has covering degree~$m$ and $\mathbf{z} \in \mathbb{R}^k$ is contained in exactly $m$ elements of~$\mathcal{K}$, then $\mathbf{z}$ is contained in the interior of each of these $m$ elements.
\end{lemma}

\subsection{Interior of Rauzy fractals}
We are now in a position to show that the Rauzy fractals are the closure of their interior.

\begin{proposition} \label{p:closint}
Let $S$ be a finite or  infinite set of unimodular substitutions over a finite alphabet $\mathcal{A}$ and let $\boldsymbol{\sigma}\in S^{\mathbb{N}}$ be a primitive, algebraically irreducible, and recurrent directive sequence with balanced language~$\mathcal{L}_{\boldsymbol{\sigma}}$.
Then each~$\mathcal{R}(i)$, $i \in \mathcal{A}$, is the closure of its interior.
\end{proposition}

\begin{proof}
By Proposition~\ref{p:covering} and Baire's theorem, for each $n \in \mathbb{N}$, we have $\mathrm{int}(\mathcal{R}^{(n)}(i)) \ne \emptyset$ for some $i \in \mathcal{A}$. 
By the set equation in \eqref{e:setequationkl} and primitivity, we get that $\mathrm{int}(\mathcal{R}^{(n)}(i)) \ne \emptyset$ for all $i \in \mathcal{A}$, $n \in \mathbb{N}$. 
Therefore, again the set equation~\eqref{e:setequationkl} yields subdivisions of $\mathcal{R}_\mathbf{w}(i)$, $i \in \mathcal{A}$, into tiles with non-empty interior whose diameters tend to~$0$ by Lemma~\ref{l:smallsubtiles}.
This proves the result.
\end{proof}

\subsection{Boundary of Rauzy fractals} 
Our next task is to show that the boundary of~$\mathcal{R}(i)$ has zero measure for each $i \in \mathcal{A}$.
The proof of this result is quite technical and requires several preparatory lemmas. 
First, we show that each ``patch'' of $\Gamma(\mathbf{w})$ occurs relatively densely in each discrete hyperplane $\Gamma(\tilde{\mathbf{w}})$ with $\tilde{\mathbf{w}}$ sufficiently close to~$\mathbf{w}$.
Note that this  statement is a crucial step  and requires  the use of   new  ideas with respect to the substitutive case, since we lose here  the possibility of using  a classical Perron-Frobenius argument (see   e.g. in \cite{ST09,CANTBST}).

\begin{lemma}\label{lemma0b}\label{l:relativelydense}
Let $r > 0$, $\mathbf{w} \in \mathbb{R}_{\ge0}^d\setminus\{\mathbf{0}\}$, and define the patch
\[
P = \{[\mathbf{x},i] \in \Gamma(\mathbf{w}):\, \|\mathbf{x}\|\le r \}.
\]
There exist $\delta, R > 0$ such that, for each $\tilde{\mathbf{w}} \in \mathbb{R}_{\ge0}^d \setminus \{\mathbf{0}\}$ with $\|\tilde{\mathbf{w}} - \mathbf{w}\| \le \delta$ and each $[\mathbf{z},j] \in \Gamma(\tilde{\mathbf{w}})$, \begin{equation}\label{eq:samepatch}
\{[\mathbf{x},i] \in \Gamma(\tilde{\mathbf{w}}):\, \|\mathbf{x} - \mathbf{y}\| \le r\} = P + \mathbf{y}
\end{equation}
for some $\mathbf{y} \in \mathbb{Z}^d$ with $\|\mathbf{y} - \mathbf{z}\| \le R$.
\end{lemma}

\begin{proof}
The set $\{[\mathbf{x},i] \in \mathbb{Z}^d \times \mathcal{A}:\, \|\mathbf{x}\| \le r\}$ admits the partition $\{P, P^+, P^-\}$, with
\begin{align*}
P^+ & = \{[\mathbf{x},i] \in \mathbb{Z}^d \times \mathcal{A}:\, \|\mathbf{x}\| \le r,\, \langle \mathbf{w},\mathbf{x} \rangle \ge \langle \mathbf{w},\mathbf{e}_i \rangle\}, \\
P^- & = \{[\mathbf{x},i] \in \mathbb{Z}^d \times \mathcal{A}:\, \|\mathbf{x}\| \le r,\, \langle \mathbf{w},\mathbf{x} \rangle < 0\}.
\end{align*}

Let $\eta_1 = \min_{[\mathbf{x},i] \in P} \langle \mathbf{w}, \mathbf{e}_i - \mathbf{x} \rangle > 0$, $\eta_2 = \min_{[\mathbf{x},i] \in P^-} \langle \mathbf{w}, - \mathbf{x} \rangle > 0$, and set $\eta=\min\{\eta_1,\eta_2\}$.
Choose $\delta > 0$ such that for all $\tilde{\mathbf{w}} \in \mathbb{R}_{\ge0}^d$ with $\|\tilde{\mathbf{w}} - \mathbf{w}\| \le \delta$ we have
\begin{equation} \label{e:eta}
\min_{[\mathbf{x},i] \in P} \langle \tilde{\mathbf{w}}, \mathbf{e}_i - \mathbf{x} \rangle \ge 2\eta/3 \qquad \mbox{and} \qquad  \min_{[\mathbf{x},i] \in P^-} \langle \tilde{\mathbf{w}}, -\mathbf{x} \rangle \ge 2\eta/3,
\end{equation}
as well as 
\begin{equation} \label{e:eta_b}
\min_{[\mathbf{x},i] \in P} \langle \tilde{\mathbf{w}}, \mathbf{x} \rangle \ge -\eta/3 \qquad \mbox{and} \qquad  \min_{[\mathbf{x},i] \in P^+} \langle \tilde{\mathbf{w}}, \mathbf{x} -\mathbf{e}_i\rangle \ge - \eta/3,
\end{equation}
and set $R = 6\, (r+1)\, (\|\mathbf{w}\|+\delta) / \eta$.

Let now $[\mathbf{z},j] \in \Gamma(\tilde{\mathbf{w}})$ with $\|\tilde{\mathbf{w}} - \mathbf{w}\| \le \delta$. 
To find $\mathbf{y} \in \mathbb{Z}^d$ satisfying $\|\mathbf{y} - \mathbf{z}\| \le R$ and~\eqref{eq:samepatch}, choose $\mathbf{x}', \mathbf{x}'' \in\mathbb{Z}^d$ with $\|\mathbf{x}'\|, \|\mathbf{x}''\| \le r+1$ such that $\langle \tilde{\mathbf{w}}, \mathbf{x}' \rangle$ is equal to the smaller of the two minima in~\eqref{e:eta}, and $\langle \tilde{\mathbf{w}}, \mathbf{x}'' \rangle$ is equal to the smaller of the two minima in~\eqref{e:eta_b}; this choice is possible by the definition of the minima.
Let $\mathbf{y} = \mathbf{z} -  h\, (\mathbf{x}' + \mathbf{x}'')$ with $h \in \mathbb{Z}$ such that 
\begin{equation} \label{e:yh}
- \langle \tilde{\mathbf{w}}, \mathbf{x}'' \rangle \le \langle \tilde{\mathbf{w}}, \mathbf{z} - h\, (\mathbf{x}' + \mathbf{x}'') \rangle < \langle \tilde{\mathbf{w}}, \mathbf{x}' \rangle\,;
\end{equation}
such an $h$ exists (uniquely) since $\langle \tilde{\mathbf{w}}, \mathbf{x}' + \mathbf{x}'' \rangle \ge \eta/3 > 0$ by \eqref{e:eta} and~\eqref{e:eta_b}.

Let $[\mathbf{x},i] \in \mathbb{Z}^d \times \mathcal{A}$ with $\|\mathbf{x}\| \le r$.
By \eqref{e:yh} and the definition of $\mathbf{x}'$ and~$\mathbf{x}''$, we have
\begin{align*}
\langle \tilde{\mathbf{w}}, \mathbf{x} + \mathbf{y} \rangle < \langle \tilde{\mathbf{w}}, \mathbf{x} + \mathbf{x}' \rangle & \le \begin{cases}\langle \tilde{\mathbf{w}}, \mathbf{e}_i \rangle & \mbox{if}\ [\mathbf{x},i] \in P, \\ 0 & \mbox{if}\ [\mathbf{x},i] \in P^-,\end{cases} \\
\langle \tilde{\mathbf{w}}, \mathbf{x} + \mathbf{y} \rangle \ge \langle \tilde{\mathbf{w}}, \mathbf{x} - \mathbf{x}'' \rangle & \ge \begin{cases}0 & \mbox{if}\ [\mathbf{x},i] \in P, \\ \langle \tilde{\mathbf{w}}, \mathbf{e}_i \rangle & \mbox{if}\ [\mathbf{x},i] \in P^+,\end{cases}
\end{align*}
thus $[\mathbf{x}+\mathbf{y},i] \in \Gamma(\tilde{\mathbf{w}})$ if $[\mathbf{x},i] \in P$ and $[\mathbf{x}+\mathbf{y},i] \notin \Gamma(\tilde{\mathbf{w}})$ if $[\mathbf{x},i] \in P^- \cup P^+$, i.e., \eqref{eq:samepatch} holds. 

To show that $\|\mathbf{y} - \mathbf{z}\| \le R$, note that $\frac{\langle\tilde{\mathbf{w}},\mathbf{z}+\mathbf{x}''\rangle}{\langle\tilde{\mathbf{w}},\mathbf{x}'+\mathbf{x}''\rangle} - 1 < h \le \frac{\langle\tilde{\mathbf{w}},\mathbf{z}+\mathbf{x}''\rangle}{\langle\tilde{\mathbf{w}},\mathbf{x}'+\mathbf{x}''\rangle}$.
Using the equalities $-\eta/3 \le \langle \tilde{\mathbf{w}}, \mathbf{x}'' \rangle \le 0$ (given by~\eqref{e:eta_b} and since $[\mathbf{0},i] \in P$), $0 \le \langle \tilde{\mathbf{w}}, \mathbf{z} \rangle \le \langle \tilde{\mathbf{w}}, \mathbf{e}_j \rangle \le \|\tilde{\mathbf{w}}\| \le \|\mathbf{w}\| + \delta$, and $\langle \tilde{\mathbf{w}}, \mathbf{x}' + \mathbf{x}'' \rangle \ge \eta/3$, we obtain that $-2 < h \le 3\,(\|\mathbf{w}\| + \delta)/\eta$, thus $|h| \le 3\,(\|\mathbf{w}\| + \delta)/\eta$ and
\[
\|\mathbf{y} - \mathbf{z}\| \le |h|\, (\|\mathbf{x}'\| + \|\mathbf{x}''\|) \le 6\, (r+1)\, (\|\mathbf{w}\|+\delta) / \eta = R. \qedhere
\]
\end{proof}

\begin{lemma} \label{l:interiornk}
Assume that the sequence $\boldsymbol{\sigma} = (\sigma_n)\in S^{\mathbb{N}}$ of unimodular substitutions has Property PRICE w.r.t.\ the sequences $(n_k)$ and $(\ell_k)$ and the vector~$\mathbf{v}$.
Then there exists $\ell \in \mathbb{N}$ such that for each pair $i, j \in \mathcal{A}$, there is $[\mathbf{y},j] \in E_1^*(\sigma_{[0,\ell)})[\mathbf{0},i]$ such that
\renewcommand{\theenumi}{\roman{enumi}}
\begin{enumerate}
\itemsep1ex
\item \label{i:66i}
$M_{[0,\ell)}\, \big(\pi_{\mathbf{u},\mathbf{v}}^{(\ell)}\, \mathbf{y} + \mathcal{R}^{(\ell)}_\mathbf{v}(j)\big) \subset \mathrm{int}\big(\mathcal{R}_\mathbf{v}(i)\big)$ and 
\item \label{i:66ii}
$M_{[0,\ell)}\, \big(\pi_{\mathbf{u},\mathbf{v}}^{(n_k+\ell)}\, \mathbf{y} + \mathcal{R}^{(n_k+\ell)}_\mathbf{v}(j)\big) \subset \mathrm{int}\big(\mathcal{R}^{(n_k)}_\mathbf{v}(i)\big)$ for all sufficiently large $k \in \mathbb{N}$.
\end{enumerate}
Moreover, the covering degree of~$\mathcal{C}_\mathbf{v}^{(n)}$ is equal to that of~$\mathcal{C}_\mathbf{v}$ for all $n \in \mathbb{N}$. 
\end{lemma}

\begin{proof}
We first show that (\ref{i:66i}) and~(\ref{i:66ii}) hold for some $i \in \mathcal{A}$, $\ell \in \mathbb{N}$, $[\mathbf{y},j] \in E_1^*(\sigma_{[0,\ell)})[\mathbf{0},i]$.
Let $m$ be the covering degree of~$\mathcal{C}_{\mathbf{v}}$, which is positive and finite according to Proposition~\ref{p:covering}. Let $\mathbf{z} \in \mathbf{v}^\bot$ be a point lying in exactly $m$ tiles of~$\mathcal{C}_{\mathbf{v}}$. By Lemma~\ref{lem:intcovering}, $\mathbf{z}$~lies in the interior of each of these tiles, and the same is true for some open neighborhood~$U$ of~$\mathbf{z}$.  
Let $\pi_{\mathbf{u},\mathbf{v}}\, \tilde{\mathbf{x}} + \mathcal{R}(i)$ be one of these tiles. 
By the set equation~\eqref{e:setequationkl} and Lemma~\ref{l:smallsubtiles}, there is $\ell \in \mathbb{N}$ and $[\tilde{\mathbf{y}},j] \in E_1^*(\sigma_{[0,\ell)})[\tilde{\mathbf{x}},i]$ such that 
\[
M_{[0,\ell)}\, \big(\pi_{\mathbf{u},\mathbf{v}}^{(\ell)}\, \tilde{\mathbf{y}} + \mathcal{R}^{(\ell)}_\mathbf{v}(j)\big) \subset U \subset \mathrm{int}\big(\pi_{\mathbf{u},\mathbf{v}}\, \tilde{\mathbf{x}} + \mathcal{R}_\mathbf{v}(i)\big).
\]
Shifting by $-\pi_{\mathbf{u},\mathbf{v}}\, \tilde{\mathbf{x}}$, we see that (\ref{i:66i}) holds for $\mathbf{y} = \tilde{\mathbf{y}} - M_{[0,\ell)}^{-1}\, \tilde{\mathbf{x}}$.

By Lemma~\ref{lem:projectionconvergence}, Proposition~\ref{p:close} and since $\mathbf{u} \in \mathbb{R}_+^d$, $\mathbf{v} \in \mathbb{R}_{\ge0}^d \setminus \{\mathbf{0}\}$, we may choose $r > 0$ such that, for all $k \in \mathbb{N}$, $\pi_{\mathbf{u},\mathbf{v}}^{(n_k)}\, \mathbf{x} \in \pi_{\mathbf{u},\mathbf{v}}^{(n_k)}\, U - \mathcal{R}^{(n_k)}_\mathbf{v}$ with $|\langle \mathbf{v}^{(n_k)}, \mathbf{x} \rangle| < \|\mathbf{v}^{(n_k)}\|$ implies $\|\mathbf{x}\| \le r$.
In the following, assume that $k$ is sufficiently large.
Setting $P = \{[\mathbf{x},i] \in \Gamma(\mathbf{v}):\, \|\mathbf{x}\|\le r\big\}$, Lemma~\ref{l:relativelydense} yields that there is $\mathbf{y}_k \in \mathbb{Z}^d$ such that $\{[\mathbf{x}+\mathbf{y}_k,i] \in \Gamma(\mathbf{v}^{(n_k)}):\, \|\mathbf{x}\| \le r\} = P + \mathbf{y}_k$.  
Let $[\mathbf{x}+\mathbf{y}_k,i] \in \Gamma(\mathbf{v}^{(n_k)})$ be such that 
\begin{equation} \label{e:Unk}
\pi_{\mathbf{u},\mathbf{v}}^{(n_k)} (\mathbf{y}_k + U) \cap \big(\pi_{\mathbf{u},\mathbf{v}}^{(n_k)} (\mathbf{x} + \mathbf{y}_k) + \mathcal{R}^{(n_k)}_\mathbf{v}(i)\big) \ne \emptyset.
\end{equation}
Then we have $\pi_{\mathbf{u},\mathbf{v}}^{(n_k)}\, \mathbf{x} \in \pi_{\mathbf{u},\mathbf{v}}^{(n_k)}\, U - \mathcal{R}^{(n_k)}_\mathbf{v}$ and $|\langle \mathbf{v}^{(n_k)}, \mathbf{x} \rangle| < \|\mathbf{v}^{(n_k)}\|$ because both $\langle \mathbf{v}^{(n_k)}, \mathbf{x}+\mathbf{y}_k \rangle$ and $\langle \mathbf{v}^{(n_k)}, \mathbf{y}_k \rangle$ are in $[0,\|\mathbf{v}^{(n_k)}\|)$, hence, $\|\mathbf{x}\| \le r$.
This gives that $[\mathbf{x}+\mathbf{y}_k,i] \in P + \mathbf{y}_k$, i.e., $[\mathbf{x},i] \in P$. 
By~\eqref{e:Unk} and Proposition~\ref{p:close}, $\pi_{\mathbf{u},\mathbf{v}}\, \mathbf{x} + \mathcal{R}_\mathbf{v}(i)$ must be one of the $m$ tiles of~$\mathcal{C}_\mathbf{v}$ that contain~$U$.
In particular, the covering degree of~$\mathcal{C}_\mathbf{v}^{(n_k)}$ is at most~$m$.
By Proposition~\ref{p:covering}, the covering degree is at least $m$ and, hence, equal to~$m$.
Therefore, we have $\pi_{\mathbf{u},\mathbf{v}}^{(n_k)} (\mathbf{y}_k + U) \subset \pi_{\mathbf{u},\mathbf{v}}^{(n_k)} (\mathbf{x} + \mathbf{y}_k) + \mathcal{R}^{(n_k)}_\mathbf{v}(i)$ for all $[\mathbf{x},i] \in \Gamma(\mathbf{v})$ satisfying $U \subset \pi_{\mathbf{u},\mathbf{v}}\, \mathbf{x} + \mathcal{R}_\mathbf{v}(i)$.
By Lemma~\ref{lem:projectionconvergence} and Proposition~\ref{p:close}, we get that 
\[
M_{[0,\ell)}\, \big(\pi_{\mathbf{u},\mathbf{v}}^{(n_k+\ell)}\, \tilde{\mathbf{y}} + \mathcal{R}^{(n_k+\ell)}_\mathbf{v}(j)\big) \subset \pi_{\mathbf{u},\mathbf{v}}^{(n_k)}\, U \subset \mathrm{int}\big(\pi_{\mathbf{u},\mathbf{v}}^{(n_k)}\, \tilde{\mathbf{x}} + \mathcal{R}_\mathbf{v}^{(n_k)}(i)\big)
\]
with $\ell, [\tilde{\mathbf{x}},i], [\tilde{\mathbf{y}},j]$ as in the preceding paragraph, hence, (\ref{i:66ii}) holds for $\mathbf{y} = \tilde{\mathbf{y}} - M_{[0,\ell)}^{-1}\, \tilde{\mathbf{x}}$.

To prove the statements for arbitrary $i,j\in \mathcal{A}$, choose $h \in \mathbb{N}$ such that $M_{[0,h)}$ positive. 
Applying the results from the preceding paragraphs and using Lemma~\ref{l:shiftedPRICE}, there are  $i' \in \mathcal{A}$, $\ell' \in \mathbb{N}$, and $[\mathbf{y}',j'] \in E_1^*(\sigma_{[h,h+\ell')})[\mathbf{0},i']$ such that 
\begin{equation} \label{e:62i'}
M_{[h,h+\ell')}\, \big(\pi_{\mathbf{u},\mathbf{v}}^{(h+\ell')}\, \mathbf{y}' + \mathcal{R}^{(h+\ell')}_\mathbf{v}(j')\big) \subset \mathrm{int}\big(\mathcal{R}^{(h)}_\mathbf{v}(i')\big)
\end{equation}
and, for sufficiently large~$k$,
\begin{equation} \label{e:62i''}
M_{[h,h+\ell')}\, \big(\pi_{\mathbf{u},\mathbf{v}}^{(n_k+h+\ell')}\, \mathbf{y}' + \mathcal{R}^{(n_k+h+\ell')}_\mathbf{v}(j')\big) \subset \mathrm{int}\big(\mathcal{R}^{(n_k+h)}_\mathbf{v}(i')\big).
\end{equation}
Choose $\ell > h+\ell'$ such that $M_{[h+\ell',\ell)}$ is positive.
Then for each pair $i,j \in \mathcal{A}$, there are $\mathbf{x}', \mathbf{y} \in \mathbb{Z}^d$ such that $[\mathbf{x}',i'] \in E_1^*(\sigma_{[0,h)})[\mathbf{0},i]$ and $[\mathbf{y},j] \in E_1^*(\sigma_{[h+\ell',\ell)})[\mathbf{y}' + (M_{[h,h+\ell')})^{-1} \mathbf{x}',j']$. We get that
\[
[\mathbf{y}, j] \in E_1^*(\sigma_{[h+\ell',\ell)})[\mathbf{y}' + (M_{[h,h+\ell')})^{-1} \mathbf{x}', j'] \subset E_1^*(\sigma_{[h,\ell)})[\mathbf{x}', i'] \subset E_1^*(\sigma_{[0,\ell)})[\mathbf{0},i],
\]
and (\ref{i:66i}) and~(\ref{i:66ii}) are true by \eqref{e:62i'} and~\eqref{e:62i''}, respectively.

We have seen that the covering degree of~$\mathcal{C}_\mathbf{v}^{(n_k)}$ is equal to that of~$\mathcal{C}_\mathbf{v}$ for all sufficiently large~$k$. As the covering degree increases monotonically by Proposition~\ref{p:covering}, this holds also for all~$\mathcal{C}_\mathbf{v}^{(n)}$. 
\end{proof}

We are now able to prove that the boundary of $\mathcal{R}(i)$ has zero measure for each $i \in \mathcal{A}$.

\begin{proposition} \label{p:boundary}
Let $S$ be a finite or  infinite set of unimodular substitutions over a finite alphabet $\mathcal{A}$ and let $\boldsymbol{\sigma}\in S^{\mathbb{N}}$ be a directive sequence with Property PRICE. Then $\lambda_{\mathbf{1}}(\partial (\mathcal{R}(i)) = 0$ for each $i \in \mathcal{A}$.
\end{proposition}

\begin{proof}
Let the sequence~$(n_k)$ and the vector~$\mathbf{v}$ be as in Definition~\ref{def:star}, and set
\begin{align*}
C_{m,n}(i,j) & = \#\big\{\mathbf{y} \in \mathbb{Z}^d:\, [\mathbf{y},j] \in E_1^*(\sigma_{[m,n)}) [\mathbf{0},i]\big\}, \\
D_{m,n}(i,j) & = \#\big\{\mathbf{y} \in \mathbb{Z}^d:\, [\mathbf{y},j] \in E_1^*(\sigma_{[m,n)}) [\mathbf{0},i],\, M_{[m,n)} \big(\pi_{\mathbf{u},\mathbf{v}}^{(n)}\, \mathbf{y} + \mathcal{R}^{(n)}_\mathbf{v}(j)\big) \cap \partial \mathcal{R}^{(m)}_\mathbf{v}(i) \ne \emptyset\big\},
\end{align*}
for $i,j \in \mathcal{A}$, $m \le n$.
Our main task is to show that 
\begin{equation} \label{e:limDC}
\lim_{n\to\infty} \frac{D_{0,n}(i,j)}{C_{0,n}(i,j)} = 0 \qquad \mbox{for all}\ i,j \in \mathcal{A}.
\end{equation} 

Let $\ell\in \mathbb{N}$ be as in the statement of Lemma~\ref{l:interiornk}. 
We thus have, for each pair $i,j \in \mathcal{A}$, at least one~$\mathbf{y}$ such that $[\mathbf{y},j] \in E_1^*(\sigma_{[0,\ell)}) [\mathbf{0},i]$ and $M_{[0,\ell)} \big(\pi_{\mathbf{u},\mathbf{v}}^{(\ell)}\, \mathbf{y} + \mathcal{R}^{(\ell)}_\mathbf{v}(j)\big) \cap \partial \mathcal{R}_\mathbf{v}(i) = \emptyset$, i.e., $D_{0,\ell}(i,j) \le C_{0,\ell}(i,j) - 1$. 
Set $c = 1 - 1/\max_{i,j\in\mathcal{A}} C_{0,\ell}(i,j) < 1$.
Since all subtiles of $M_{[0,\ell)}\, (\pi_{\mathbf{u},\mathbf{v}}^{(\ell)}\, \mathbf{y} + \mathcal{R}^{(\ell)}_\mathbf{v}(j))$ are also contained in $\mathrm{int}(\mathcal{R}_\mathbf{v}(i))$, we obtain for each $n \ge \ell$ that 
\[
D_{0,n}(i,j) \le \sum_{j'\in\mathcal{A}} D_{0,\ell}(i,j')\, C_{\ell,n}(j',j) \le c \sum_{j'\in\mathcal{A}} C_{0,\ell}(i,j')\, C_{\ell,n}(j',j) = c\, C_{0,n}(i,j).
\]
Let us refine this inequality using Lemma~\ref{l:interiornk}~(\ref{i:66ii}). 
For sufficiently large~$k$, we have $D_{n_k,n_k+\ell}(i,j) \le C_{n_k,n_k+\ell}(i,j) - 1 = C_{0,\ell}(i,j) - 1$, and each subtile $\pi_{\mathbf{u},\mathbf{v}}^{(n_k+\ell)}\, \mathbf{y} + \mathcal{R}^{(n_k+\ell)}_\mathbf{v}(j)$ that is in the interior of a subtile $\pi_{\mathbf{u},\mathbf{v}}^{(n_k)}\, \mathbf{x} + \mathcal{R}^{(n_k)}_\mathbf{v}(i')$ of~$\mathcal{R}_\mathbf{v}(i)$ is clearly also in the interior of~$\mathcal{R}_\mathbf{v}(i)$.
Thus we have
\begin{align*}
D_{0,n}(i,j) & \le \sum_{j',i',j''\in\mathcal{A}} D_{0,\ell}(i,j')\, C_{\ell,n_k}(j',i') D_{n_k,n_k+\ell}(i',j'')\, C_{n_k+\ell,n}(j'',j) \\
& \le c^2 \sum_{j',i',j''\in\mathcal{A}} C_{0,\ell}(i,j')\, C_{\ell,n_k}(j',i')\, C_{0,\ell}(i',j'')\, C_{n_k+\ell,n}(j'',j) = c^2\, C_{0,n}(i,j)
\end{align*}
for $n \ge n_k+\ell$.
A~similar argument with $h$ different values of~$n_k$ yields for each $h \in \mathbb{N}$ that $D_{0,n}(i,j) \le c^{h+1}\, C_{0,n}(i,j)$ for sufficiently large~$n$, thus \eqref{e:limDC} is true.

By Lemma~\ref{lem:projectionconvergence}, Proposition~\ref{p:close} and since $\pi_{\mathbf{u},\mathbf{v}}\, \Gamma(\mathbf{v})$ is uniformly discrete, there exists $m \in \mathbb{N}$ such that, for all $k \in \mathbb{N}$, each point of~$(\mathbf{v}^{(n_k)})^\bot$ lies in at most $m$ tiles of~$\mathcal{C}^{(n_k)}_\mathbf{v}$.
Then
\begin{align*}
\lambda_\mathbf{v} \big(\partial \mathcal{R}_\mathbf{v}(i)\big) & \le \sum_{j\in\mathcal{A}} D_{0,n}(i,j)\, \lambda_\mathbf{v} \big(M_{[0,n)}\, \mathcal{R}^{(n)}_\mathbf{v}(j)\big) \qquad \mbox{for all}\ n \in\mathbb{N}, \\
\lambda_\mathbf{v} \big(\mathcal{R}_\mathbf{v}(i)\big) & \ge \frac{1}{m} \sum_{j\in\mathcal{A}} C_{0,n_k}(i,j)\, \lambda_\mathbf{v} \big(M_{[0,n_k)}\, \mathcal{R}^{(n_k)}_\mathbf{v}(j)\big) \qquad \mbox{for all}\ k \in\mathbb{N},
\end{align*}
by the set equations~\eqref{e:setequationkl}, thus
\[
\frac{\lambda_\mathbf{v}(\partial\mathcal{R}_\mathbf{v}(i)\big)}{\lambda_\mathbf{v}(\mathcal{R}_\mathbf{v}(i))} \le \frac{m\, \sum_{j\in\mathcal{A}} D_{0,n_k}(i,j)}{\sum_{j\in\mathcal{A}} C_{0,n_k}(i,j)}\, \frac{\max_{j\in\mathcal{A}} \lambda_\mathbf{v} (M_{[0,n_k)}\, \mathcal{R}^{(n_k)}_\mathbf{v}(j))}{\min_{j\in\mathcal{A}} \lambda_\mathbf{v} (M_{[0,n_k)}\, \mathcal{R}^{(n_k)}_\mathbf{v}(j))} \qquad \mbox{for all}\ k \in\mathbb{N}.
\]
It remains to show that the latter fraction is bounded.
Let $h \in \mathbb{N}$ be such that $M_{[0,h)}$ is a positive matrix. 
For sufficiently large~$k$, we have $M_{[n_k,n_k+h)} = M_{[0,h)}$ and thus
\begin{align*}
\frac{\max_{i\in\mathcal{A}} \lambda_\mathbf{v} (M_{[0,n_k)}\, \mathcal{R}^{(n_k)}_\mathbf{v}(i))}{\min_{i\in\mathcal{A}} \lambda_\mathbf{v} (M_{[0,n_k)}\, \mathcal{R}^{(n_k)}_\mathbf{v}(i))} & \le \frac{\max_{i\in\mathcal{A}} \sum_{j\in\mathcal{A}} C_{0,h}(i,j) \max_{j\in\mathcal{A}} \lambda_\mathbf{v} (M_{[0,n_k+h)}\, \mathcal{R}^{(n_k+h)}_\mathbf{v}(j))}{\max_{j\in\mathcal{A}} \lambda_\mathbf{v} (M_{[0,n_k+h)}\, \mathcal{R}^{(n_k+h)}_\mathbf{v}(j))} \\
& = \max_{i\in\mathcal{A}} \sum_{j\in\mathcal{A}} C_{0,h}(i,j).
\end{align*}
Together with~\eqref{e:limDC}, we obtain that $\lambda_\mathbf{v}(\partial \mathcal{R}_\mathbf{v}(i)) = 0$ and, hence, $\lambda_\mathbf{1}(\partial \mathcal{R}(i)) = 0$. 
\end{proof}

We also get the following strengthening of Proposition~\ref{p:close} for the difference between~$\mathcal{R}_\mathbf{v}^{(\ell)}$ and~$\pi_{\mathbf{u},\mathbf{v}}^{(\ell)}\, \mathcal{R}_\mathbf{v}^{(n_k+\ell)}$. 
One can prove in a similar way that $\lim_{k\to\infty} \lambda_{\mathbf{v}^{(\ell)}} \big(\pi_{\mathbf{u},\mathbf{v}}^{(\ell)}\, \mathcal{R}^{(n_k+\ell)}_\mathbf{v}(i) \setminus \mathcal{R}_\mathbf{v}^{(\ell)}(i)\big) = 0$, but we do not need this result.

\begin{lemma} \label{l:close2}
Assume that the sequence $\boldsymbol{\sigma} = (\sigma_n)\in S^{\mathbb{N}}$ of unimodular substitutions has Property PRICE w.r.t.\ the sequences $(n_k)$ and $(\ell_k)$ and the vector~$\mathbf{v}$. 
Then, for each $i \in \mathcal{A}$ and $\ell \in \mathbb{N}$,
\begin{equation} \label{e:close2}
\lim_{k\to\infty} \lambda_{\mathbf{v}^{(\ell)}} \big(\mathcal{R}_\mathbf{v}^{(\ell)}(i) \setminus \pi_{\mathbf{u},\mathbf{v}}^{(\ell)}\, \mathcal{R}^{(n_k+\ell)}_\mathbf{v}(i)\big) = 0.
\end{equation}
\end{lemma}

\begin{proof}
Let $\ell = 0$, the case $\ell > 0$ then being a consequence of Lemma~\ref{l:shiftedPRICE}.
For $\varepsilon > 0$ and $X \subset \mathbf{v}^\bot$, let $X_\varepsilon = \{\mathbf{x} \in \mathbf{v}^\bot:\, \|\mathbf{x} - \mathbf{y}\| \le \varepsilon\ \mbox{for some}\ \mathbf{y} \in X\}$.
With the notation of the proof of Proposition~\ref{p:boundary}, we obtain that
\[
\lambda_\mathbf{v} \Big((\mathcal{R}_\mathbf{v}(i))_\varepsilon \setminus \mathcal{R}_\mathbf{v}(i)\Big) \le \sum_{j\in\mathcal{A}} D_{0,n}(i,j)\, \lambda_\mathbf{v} \Big(\big(M_{[0,n)} \mathcal{R}^{(n)}_\mathbf{v}(j)\big)_\varepsilon\Big).
\]
Let $\varepsilon' > 0$ be arbitrary but fixed.
By the proof of Proposition~\ref{p:boundary}, we have some $n \in \mathbb{N}$ such that  $\sum_{j\in\mathcal{A}} D_{0,n}(i,j)\, \lambda_\mathbf{v} \big(M_{[0,n)} \mathcal{R}^{(n)}_\mathbf{v}(j)\big) < \varepsilon'$. 
Choose $\varepsilon > 0$ such that 
\[
\sum_{j\in\mathcal{A}} D_{0,n}(i,j)\, \lambda_\mathbf{v} \big(\big(M_{[0,n)} \mathcal{R}^{(n)}_\mathbf{v}(j)\big)_\varepsilon\big) < \varepsilon'.
\]
This is possible since, for compact $X \subset \mathbf{v}^\bot$, we have $\bigcap_{\varepsilon>0} X_\varepsilon = X$, thus ${\lim_{\varepsilon\to0} \lambda_\mathbf{v}(X_\varepsilon) = \lambda_\mathbf{v}(X)}$.
For sufficiently large~$k$, we have $\pi_{\mathbf{u},\mathbf{v}} \mathcal{R}^{(n_k)}_\mathbf{v}(i) \subset (\mathcal{R}_\mathbf{v}(i))_\varepsilon$ by Proposition~\ref{p:close}, which implies that $\lambda_\mathbf{v} \big(\pi_{\mathbf{u},\mathbf{v}}\, \mathcal{R}^{(n_k)}_\mathbf{v}(i) \setminus \mathcal{R}_\mathbf{v}(i)\big) < \varepsilon'$. 
As the choice of~$\varepsilon'$ was arbitrary, this yields~\eqref{e:close2}.
\end{proof}

\section{Tilings and coincidences} \label{sec:tilings}

Let $S$ be a finite or  infinite set of unimodular substitutions over the finite alphabet~$\mathcal{A}$ and let $\boldsymbol{\sigma}\in S^{\mathbb{N}}$ be a directive sequence. In this section, we prove several tiling results for Rauzy fractals associated with $\boldsymbol{\sigma}$. First we show that the collections $\mathcal{C}_\mathbf{w}$ form multiple tilings under general conditions and prove that the subdivision of the Rauzy fractals induced by the set equation consists of measure disjoint pieces. In the second part we deal with various coincidence conditions that imply further measure disjointness properties of Rauzy fractals and lead to criteria for $\mathcal{C}_\mathbf{w}$ to be a tiling.  

\subsection{Tiling properties}\label{sec:tilings2}
We start this section by giving a general criterion for the collection $\mathcal{C}_\mathbf{v}$ to be a multiple tiling.

\begin{lemma} \label{l:mtiling}
Assume that the sequence~$\boldsymbol{\sigma}\in S^{\mathbb{N}}$  of unimodular substitutions has Property PRICE with recurrent left eigenvector~$\mathbf{v}$. Then the collection~$\mathcal{C}_\mathbf{v}$ forms a multiple tiling of~$\mathbf{v}^\bot$.
\end{lemma}

\begin{proof}
Let $(n_k)$ be the strictly increasing sequence associated with~$\boldsymbol{\sigma}$ according to Definition~\ref{def:star}, let $m$ be the covering degree of~$\mathcal{C}_\mathbf{v}$, which is positive and finite by Proposition~\ref{p:covering}, and let $X$ be the set of points lying in at least $m+1$ tiles of~$\mathcal{C}_\mathbf{v}$.
We have to show that $X$ has zero measure.

By Lemma~\ref{l:interiornk}, each $(\mathbf{v}^{(n_k)})^\bot$ with sufficiently large~$k$ contains points lying in exactly $m$ tiles of~$\mathcal{C}^{(n_k)}_\mathbf{v}$.
Moreover, by Lemma~\ref{l:relativelydense}, there exists a constant $R > 0$ such that each ball of radius~$R$ in~$\Gamma(\mathbf{v}^{(n_k)})$ contains $\mathbf{y}_k$ as in the proof of Lemma~\ref{l:interiornk}.
Since $\|\mathbf{x} - \pi_{\mathbf{u},\mathbf{v}}^{(n_k)}\, \mathbf{x}\|$, with $[\mathbf{x},i] \in \Gamma(\mathbf{v}^{(n_k)})$, is bounded, we obtain that there exists $R' > 0$ such that each ball of radius $R'$ in~$(\mathbf{v}^{(n_k)})^\bot$ contains a point lying in exactly $m$ tiles of~$\mathcal{C}^{(n_k)}_\mathbf{v}$, for all sufficiently large~$k$.

On the other hand, by Lemma~\ref{l:covering}, each point in $(M_{[0,n_k)})^{-1} X \subset (\mathbf{v}^{(n_k)})^\bot$ is covered at least $m+1$ times by elements of~$\mathcal{C}_\mathbf{v}^{(n_k)}$.  
Assume that $X$ has positive measure.
Then, as the boundaries of~$\mathcal{R}(i)$ and thus of~$\mathcal{R}_\mathbf{v}(i)$ have zero measure by Proposition~\ref{p:boundary}, there are points in~$X$ that are not contained in the boundary of any element of~$\mathcal{C}_\mathbf{v}$. 
Thus $X$ contains a ball of positive diameter, and, by Proposition~\ref{p:strongconvergence}, $(M_{[0,n_k)})^{-1} X$ contains a ball of radius~$R'$ for all sufficiently large~$k$. 
This contradicts the fact that each ball of radius~$R'$ in $(\mathbf{v}^{(n_k)})^\bot$ contains a point that is covered at most $m$ times.
Therefore, $X$~has zero measure, i.e., $\mathcal{C}_\mathbf{v}$~forms a multiple tiling with covering degree~$m$.
\end{proof}

\begin{lemma} \label{l:mtilingPvn}
Assume that  the sequence~$\boldsymbol{\sigma}\in S^{\mathbb{N}}$ of unimodular substitutions has Property PRICE with recurrent left eigenvector~$\mathbf{v}$. 
Then, for each $n \in \mathbb{N}$, $\mathcal{C}^{(n)}_\mathbf{v}$~is a multiple tiling of~$(\mathbf{v}^{(n)})^\bot$, with covering degree equal to that of~$\mathcal{C}_\mathbf{v}$.
\end{lemma}

\begin{proof}
If $(\sigma_n)_{n\in\mathbb{N}}$ has Property PRICE w.r.t.\ the sequences $(n_k)$ and $(\ell_k)$ and the vector~$\mathbf{v}$, then there is $k_0\in\mathbb{N}$ such that $(\sigma_{m+n})_{m\in\mathbb{N}}$ has Property PRICE w.r.t.\ the sequences $(n_{k+k_0})$ and $(\ell_{k+k_0}{-}n)$ and the vector~$\mathbf{v}^{(n)}$ by Lemma~\ref{l:shiftedPRICE}, thus $\mathcal{C}^{(n)}_\mathbf{v}$~is a multiple tiling of~$(\mathbf{v}^{(n)})^\bot$ by Lemma~\ref{l:mtiling}.
By Lemma~\ref{l:interiornk}, the covering degree of~$\mathcal{C}^{(n)}_\mathbf{v}$ is equal to that of~$\mathcal{C}_\mathbf{v}$.
\end{proof}

\begin{proposition} \label{p:disjoint}
Let $S$ be a finite or  infinite set of unimodular substitutions and let $\boldsymbol{\sigma}\in S^{\mathbb{N}}$ be a directive sequence with Property PRICE.
Then the unions in the set equations \eqref{e:setequationkl} of Proposition~\ref{p:setequation} are disjoint in measure.
\end{proposition}

\begin{proof}
Let $\mathbf{v}$ be a recurrent left eigenvector as in Definition~\ref{def:star}, let $m$ be the covering degree of the multiple tilings~$\mathcal{C}^{(n)}_\mathbf{v}$, according to Lemma~\ref{l:mtilingPvn}, and $k < \ell$. 
Then the set of points in~$(\mathbf{v}^{(\ell)})^\bot$ lying in at least $m+1$ tiles of~$\mathcal{C}^{(\ell)}_\mathbf{v}$ has zero measure and each point in $(\mathbf{v}^{(k)})^\bot$ lies in at least $m$ tiles of~$\mathcal{C}^{(k)}_\mathbf{v}$.
Therefore, Lemma~\ref{l:covering} implies that the intersection of $\pi_{\mathbf{u},\mathbf{v}}^{(\ell)}\, \mathbf{y} +  \mathcal{R}^{(\ell)}_\mathbf{v}(j)$ and $\pi_{\mathbf{u},\mathbf{v}}^{(\ell)}\, \mathbf{y}' +  \mathcal{R}^{(\ell)}_\mathbf{v}(j')$ has zero measure for distinct $[\mathbf{y},j], [\mathbf{y}',j'] \in E_1^*(\sigma_{[k,\ell)})[\mathbf{x},i]$, with $[\mathbf{x},i] \in \Gamma(\mathbf{v}^{(k)})$.
By translation, this also holds for all $[\mathbf{x},i] \in \mathbb{Z}^d \times \mathcal{A}$ such that $\langle \mathbf{v}^{(k)}, \mathbf{e}_i \rangle > 0$.
Projecting by~$\pi_{\mathbf{u},\mathbf{w}}^{(\ell)}$, we obtain that $\pi_{\mathbf{u},\mathbf{w}}^{(\ell)}\, \mathbf{y} +  \mathcal{R}^{(\ell)}_\mathbf{w}(j)$ and $\pi_{\mathbf{u},\mathbf{w}}^{(\ell)}\, \mathbf{y}' +  \mathcal{R}^{(\ell)}_\mathbf{w}(j')$ are disjoint in measure for all $\mathbf{w} \in \mathbb{R}_{\ge0}^d \setminus \{\mathbf{0}\}$. 

It remains to consider the case that $\langle \mathbf{v}^{(k)}, \mathbf{e}_i \rangle = 0$.
By primitivity of~$\boldsymbol{\sigma}$, there is $h \in \mathbb{N}$ such that $\mathbf{v}^{(h)} \in \mathbb{R}_+^d$.
For sufficiently large~$\kappa$, we have thus $\mathbf{v}^{(n_\kappa+k)} \in \mathbb{R}_+^d$ and the previous paragraph implies that the intersection of $\pi_{\mathbf{u},\mathbf{v}}^{(n_\kappa+\ell)}\, \mathbf{y} +  \mathcal{R}^{(n_\kappa+\ell)}_\mathbf{v}(j)$ and $\pi_{\mathbf{u},\mathbf{v}}^{(n_\kappa+\ell)}\, \mathbf{y}' +  \mathcal{R}^{(n_\kappa+\ell)}_\mathbf{v}(j')$ has zero measure for distinct $[\mathbf{y},j], [\mathbf{y}',j'] \in E_1^*(\sigma_{[k,\ell)})[\mathbf{0},i]$.
As $\lim_{\kappa\to\infty} \pi_{\mathbf{u},\mathbf{v}}^{(\ell)}\, \pi_{\mathbf{u},\mathbf{v}}^{(n_\kappa+\ell)}\, \mathbf{y} = \pi_{\mathbf{u},\mathbf{v}}^{(\ell)}\, \mathbf{y}$ by Lemma~\ref{lem:projectionconvergence} and
$\lim_{\kappa\to\infty} \lambda_{\mathbf{v}^{(\ell)}} \big(\mathcal{R}_\mathbf{v}^{(\ell)}(j) \setminus \pi_{\mathbf{u},\mathbf{v}}^{(\ell)}\, \mathcal{R}^{(n_\kappa+\ell)}_\mathbf{v}(j)\big) = 0$ by Lemma~\ref{l:close2}, we obtain that the intersection of $\pi_{\mathbf{u},\mathbf{v}}^{(\ell)}\, \mathbf{y} +  \mathcal{R}^{(\ell)}_\mathbf{v}(j)$ and $\pi_{\mathbf{u},\mathbf{v}}^{(\ell)}\, \mathbf{y}' +  \mathcal{R}^{(\ell)}_\mathbf{v}(j')$ also has zero measure.
\end{proof}

\begin{lemma} \label{l:measure}
Let $S$ be a finite or  infinite set of unimodular substitutions.
Assume that the sequence  $\boldsymbol{\sigma}=(\sigma_n)\in S^{\mathbb{N}}$ of unimodular substitutions has Property PRICE with recurrent left eigenvector~$\mathbf{v}$.
Let $m$ be the covering degree of the multiple tiling~$\mathcal{C}_\mathbf{v}$, and identify $[\mathbf{0},i]$ with a face of the unit hypercube orthogonal to~$\mathbf{e}_i$.
Then 
\begin{equation} \label{e:lambdav}
\tr{\big(\lambda_\mathbf{v}(\mathcal{R}_\mathbf{v}(1)), \ldots, \lambda_\mathbf{v}(\mathcal{R}_\mathbf{v}(d))\big)} = m\, \tr{\big(\lambda_\mathbf{v}(\pi_{\mathbf{u},\mathbf{v}}\, [\mathbf{0},1]), \ldots, \lambda_\mathbf{v}(\pi_{\mathbf{u},\mathbf{v}}\, [\mathbf{0},d])\big)} \in \mathbb{R} \mathbf{u}\,.
\end{equation}
\end{lemma}

\begin{proof}
As in the proof of \cite[Lemma~2.3]{Ito-Rao:06}, we see that $\tr{\big(\lambda_\mathbf{v}(\pi_{\mathbf{u},\mathbf{v}}\, [\mathbf{0},1]), \ldots, \lambda_\mathbf{v} (\pi_{\mathbf{u},\mathbf{v}}\, [\mathbf{0},d])\big)} \in \mathbb{R} \mathbf{u}$.
Using the set equations~\eqref{e:setequationkl} and Proposition~\ref{p:disjoint}, we obtain that 
\[
\begin{pmatrix} \lambda_\mathbf{v}(\mathcal{R}_\mathbf{v}(1)) \\ \vdots \\ \lambda_\mathbf{v}(\mathcal{R}_\mathbf{v}(d)) \end{pmatrix} = M_{[0,n)} \begin{pmatrix} \lambda_\mathbf{v}(M_{[0,n)} \mathcal{R}^{(n)}_\mathbf{v}(1)) \\ \vdots \\ \lambda_\mathbf{v}(M_{[0,n)} \mathcal{R}^{(n)}_\mathbf{v}(d)) \end{pmatrix}
\]
for all $n \in \mathbb{N}$.
Then~\eqref{e:topPF} implies that $\tr{\big(\lambda_\mathbf{v}(\mathcal{R}_\mathbf{v}(1)), \ldots, \lambda_\mathbf{v}(\mathcal{R}_\mathbf{v}(d))\big)} \in \mathbb{R} \mathbf{u}$, hence, 
\[
\big(\lambda_\mathbf{v}(\mathcal{R}_\mathbf{v}(1)), \ldots, \lambda_\mathbf{v}(\mathcal{R}_\mathbf{v}(d))\big) = r\, \big(\lambda_\mathbf{v}(\pi_{\mathbf{u},\mathbf{v}}\, [\mathbf{0},1]), \ldots, \lambda_\mathbf{v}(\pi_{\mathbf{u},\mathbf{v}}\, [\mathbf{0},d])\big)
\]
for some $r \in \mathbb{R}$. 
Now, as $\{\pi_{\mathbf{u},\mathbf{v}} (\mathbf{x} + [\mathbf{0},i]):\, [\mathbf{x},i] \in \Gamma(\mathbf{v})\}$ forms a tiling of~$\mathbf{v}^\bot$, and $\mathcal{C}_\mathbf{v}$ has covering degree~$m$, we have $r = m$.
\end{proof}

The following result seems to be new even in the periodic case: Rauzy fractals induce tilings on any given hyperplane; in particular, $\mathcal{R}_{\mathbf{e}_i}(i)$ tiles ${\mathbf{e}_i}^\bot$ periodically for each $i \in \mathcal{A}$.

\begin{proposition}\label{p:independentmultiple}
Let $S$ be a finite or  infinite set of unimodular substitutions over a finite alphabet and assume that $\boldsymbol{\sigma}\in S^{\mathbb{N}}$ has Property PRICE. Then, for each $\mathbf{w} \in \mathbb{R}^d_{\ge0} \setminus \{\mathbf{0}\}$, the collection $\mathcal{C}_\mathbf{w}$ forms a multiple tiling of~$\mathbf{w}^\bot$, with covering degree not depending on~$\mathbf{w}$.
\end{proposition}

\begin{proof}
Let $\mathbf{v}$ be a recurrent left eigenvector as in Definition~\ref{def:star} and $\mathbf{w} \in \mathbb{R}^d_{\ge0} \setminus \{\mathbf{0}\}$.
Consider the collections $\mathcal{D}_\mathbf{w}^{(n)} = \{\mathcal{S}_\mathbf{w}^{(n)}(\mathbf{x},i):\, [\mathbf{x},i] \in \Gamma(\mathbf{w})\}$, $n \in \mathbb{N}$, with
\[
\mathcal{S}_\mathbf{w}^{(n)}(\mathbf{x},i) = \bigcup_{[\mathbf{y},j] \in E_1^*(\sigma_{[0,n)})[\mathbf{x},i] \cap \Gamma(\mathbf{v}^{(n)})} M_{[0,n)} \big(\pi_{\mathbf{u},\mathbf{w}}^{(n)}\, \mathbf{y} + \mathcal{R}_\mathbf{w}^{(n)}(j)\big).
\]
By Lemma~\ref{l:mtilingPvn}, the collections $\pi_{\mathbf{u},\mathbf{w}}^{(n)}\, \mathcal{C}^{(n)}_\mathbf{v} = \{\pi_{\mathbf{u},\mathbf{w}}^{(n)}\, \mathbf{y} + \mathcal{R}_\mathbf{w}^{(n)}(j):\, [\mathbf{y},j] \in \Gamma(\mathbf{v}^{(n)})\}$ are multiple tilings with covering degree~$m$ not depending on~$n$.
Therefore, for each $n \in \mathbb{N}$ by Lemma~\ref{l:e1star}~(\ref{62iii}), almost all points in~$\mathbf{w}^\bot$ lie in at most~$m$ sets of~$\mathcal{D}_\mathbf{w}^{(n)}$.

Next we show that $\mathcal{S}_\mathbf{w}^{(n)}(\mathbf{x},i)$ tends to $\pi_{\mathbf{u},\mathbf{w}}\, \mathbf{x} + \mathcal{R}_\mathbf{w}(i)$ in measure. 
For any $[\mathbf{y},j] \in E_1^*(\sigma_{[0,n)})[\mathbf{x},i]$, we have $p,s \in \mathcal{A}^*$ such that $\mathbf{y} = (M_{[0,n)})^{-1} (\mathbf{x} + \mathbf{l}(p))$, $\sigma_{[0,n)}(j) = p\hspace{.1em}i\hspace{.1em}s$.
Since
\[
\langle \mathbf{v}^{(n)}, \mathbf{y} \rangle = \langle \mathbf{v}, \mathbf{x} + \mathbf{l}(p) \rangle = \langle \mathbf{v}, \mathbf{x} - \mathbf{l}(i\hspace{.1em}s) \rangle + \langle \mathbf{v}^{(n)}, \mathbf{e}_j \rangle
\]
and $[\mathbf{y},j] \in \Gamma(\mathbf{v}^{(n)})$ if and only if $0 \le \langle \mathbf{v}^{(n)}, \mathbf{y} \rangle < \langle \mathbf{v}^{(n)}, \mathbf{e}_j \rangle$, we have $[\mathbf{y},j] \not\in \Gamma(\mathbf{v}^{(n)})$ if and only if
\[ 
\langle \mathbf{v}, \mathbf{l}(p) \rangle < - \langle \mathbf{v}, \mathbf{x} \rangle \quad \mbox{or} \quad \langle \mathbf{v}, \mathbf{l}(i\hspace{.1em}s) \rangle \le \langle \mathbf{v}, \mathbf{x} \rangle.
\]
As $\mathbf{v} \in \mathbb{R}_{\ge0}^d \setminus \{\mathbf{0}\}$ and each letter in~$\mathcal{A}$ occurs in $\sigma_{[0,n)}(j)$ with bounded gaps (by primitivity of~$\boldsymbol{\sigma}$), there is only a bounded number of faces $[\mathbf{y},j] \in E_1^*(\sigma_{[0,n)})[\mathbf{x},i] \setminus \Gamma(\mathbf{v}^{(n)})$ for each~$n$ (with the bound depending on~$\mathbf{x}$).
By \eqref{e:setequationkl} and Lemma~\ref{l:smallsubtiles}, we obtain that
\begin{align*}
& \lim_{n\to\infty} \lambda_\mathbf{w}\big(\big(\pi_{\mathbf{u},\mathbf{w}}\, \mathbf{x} + \mathcal{R}_\mathbf{w}(i)\big) \setminus \mathcal{S}_\mathbf{w}^{(n)}(\mathbf{x},i)\big) \\
& \qquad = \lim_{n\to\infty} \lambda_\mathbf{w} \Bigg( \bigcup_{[\mathbf{y},j] \in E_1^*(\sigma_{[0,n)})[\mathbf{x},i] \setminus \Gamma(\mathbf{v}^{(n)})} M_{[0,n)} \big(\pi_{\mathbf{u},\mathbf{w}}^{(n)}\, \mathbf{y} + \mathcal{R}_\mathbf{w}^{(n)}(j)\big) \Bigg) = 0
\end{align*}
for all $[\mathbf{x},i] \in \mathbb{Z}^d \times \mathcal{A}$. 
Therefore, almost all points in~$\mathbf{w}^\bot$ lie in at most~$m$ sets of~$\mathcal{C}_\mathbf{w}$.

Projecting the sets in~\eqref{e:lambdav} to~$\mathbf{w}^\bot$, we obtain that 
\[
\big(\lambda_\mathbf{w}(\mathcal{R}_\mathbf{w}(1)), \ldots, \lambda_\mathbf{w}(\mathcal{R}_\mathbf{w}(d))\big) = m\, \big(\lambda_\mathbf{w}(\pi_{\mathbf{u},\mathbf{w}}([\mathbf{0},1])), \ldots, \lambda_\mathbf{w}(\pi_{\mathbf{u},\mathbf{w}}([\mathbf{0},d]))\big)\,.
\]
As almost all points in~$\mathbf{w}^\bot$ lie in at most $m$ different sets $\pi_{\mathbf{u},\mathbf{w}}\, \mathbf{x} + \mathcal{R}_\mathbf{w}(i)$, this implies that $\mathcal{C}_\mathbf{w}$ forms a multiple tiling of~$\mathbf{w}^\bot$ with covering degree~$m$. 
\end{proof}

The following proposition generalizes a result of \cite{Ito-Rao:06}.

\begin{proposition} \label{p:tilingRd}
Let $\mathbf{w}\in \mathbb{R}^d_{\ge 0} \setminus \{\mathbf{0}\}$. 
Then $\widehat{\mathcal{C}}_{\mathbf{w}}$ forms a multiple tiling of~$\mathbb{R}^d$ with covering degree~$m$ if and only if $\mathcal{C}_{\mathbf{w}}$ forms a multiple tiling of~$\mathbf{w}^\bot$ with covering degree~$m$.
\end{proposition}

\begin{proof}
For $\mathbf{x} \in \mathbb{Z}^d$, we have
\[
\big({-}\mathbf{x} + \widehat{\mathcal{R}}_{\mathbf{w}}(i)\big) \cap \mathbf{w}^\bot =
\begin{cases} -(\pi_{\mathbf{u},\mathbf{w}}\, \mathbf{x} + \mathcal{R}_{\mathbf{w}}(i)) & \mbox{if}\ [\mathbf{x},i] \in \Gamma(\mathbf{w}), \\
\emptyset & \mbox{otherwise}, \end{cases}
\]
since $\langle \mathbf{w}, x\, (\mathbf{e}_i - \pi_{\mathbf{u},\mathbf{w}}\, \mathbf{e}_i)\rangle = x \langle \mathbf{w}, \mathbf{e}_i\rangle$ and $\pi_{\mathbf{u},\mathbf{w}} (\mathbf{e}_i - \pi_{\mathbf{u},\mathbf{w}}\, \mathbf{e}_i) = \mathbf{0}$.
This implies that for $\mathbf{x}, \mathbf{y} \in \mathbb{Z}^d$ we have
\[ 
\big({-}\mathbf{x} + \widehat{\mathcal{R}}_{\mathbf{w}}(i)\big) \cap \big(\mathbf{y} + \mathbf{w}^\bot\big) = \begin{cases}\mathbf{y} - \big(\pi_{\mathbf{u},\mathbf{w}} (\mathbf{x} + \mathbf{y}) + \mathcal{R}_{\mathbf{w}}(i)\big) & \mbox{if}\ [\mathbf{x}+\mathbf{y},i] \in \Gamma(\mathbf{w}), \\ \emptyset & \mbox{otherwise}, \end{cases}
\]
i.e., the intersection of $\widehat{\mathcal{C}}_{\mathbf{w}}$ with $\mathbf{y} + \mathbf{w}^\bot$ is a translation of~$\mathcal{C}_{\mathbf{w}}$.
Moreover, we have
\begin{equation} \label{e:yzu}
\big({-}\mathbf{x} + \widehat{\mathcal{R}}_{\mathbf{w}}(i)\big) \cap \big(\mathbf{y} + z\, \mathbf{u} + \mathbf{w}^\bot\big) = \big({-}\mathbf{x} + \widehat{\mathcal{R}}_{\mathbf{w}}(i)\big) \cap \big(\mathbf{y} + \mathbf{w}^\bot\big) + z\, \mathbf{u}
\end{equation}
for all $0 \le z < \langle \mathbf{w}, \mathbf{e}_i - \mathbf{x}-\mathbf{y}\rangle$. 
This proves the statement of the proposition when $\{\langle\mathbf{w},\mathbf{y}\rangle:\, \mathbf{y}\in\mathbb{Z}^d\}$ is dense in~$\mathbb{R}$, i.e., when $\mathbf{w}$ is not a multiple of a rational vector. 
If $\mathbf{w}$ is a multiple of a rational vector, then $\{\langle\mathbf{w},\mathbf{y}\rangle:\, \mathbf{y}\in\mathbb{Z}^d\} = c\, \langle \mathbf{w},\mathbf{u}\rangle\, \mathbb{Z}$ for some $c > 0$. 
Now, \eqref{e:yzu} holds for all $\mathbf{x}, \mathbf{y} \in \mathbb{Z}^d$, $0 \le z < c$, hence the statement of the proposition holds in this case as well. 
\end{proof}

\subsection{Coincidences}

In this subsection, we show that strong coincidence implies non-overlapping of the pieces~$\mathcal{R}(i)$. Moreover, we prove that geometric coincidence is equivalent to tiling. We also give variants of the geometric coincidence condition that can be checked algorithmically in certain cases.  

\begin{proposition} \label{p:strongcoincidence}
Let $S$ be a finite or  infinite set of unimodular substitutions over the finite alphabet $\mathcal{A}$ and assume that $\boldsymbol{\sigma}\in S^{\mathbb{N}}$ has Property PRICE and satisfies the strong coincidence condition.
Then the subtiles~$\mathcal{R}(i)$, $i \in \mathcal{A}$, are pairwise disjoint in measure.
\end{proposition}

\begin{proof}
Let the sequence~$(n_k)$ and the vector~$\mathbf{v}$ be as in Definition~\ref{def:star}.
By the definition of $E_1^*$ strong coincidence can be reformulated by saying that there is $\ell \in \mathbb{N}$ such that, for each pair of distinct $j_1,j_2 \in \mathcal{A}$, there are $i \in \mathcal{A}$ and $\mathbf{y} \in \mathbb{Z}^d$ such that $[\mathbf{y},j_1], [\mathbf{y}, j_2] \in E_1^*(\sigma_{[0,\ell)})[\mathbf{0},i]$. 
Thus Proposition~\ref{p:disjoint} yields that
\begin{equation} \label{e:j1j2disjoint}
\lambda_{\mathbf{v}^{(\ell)}}\big(\mathcal{R}^{(\ell)}_\mathbf{v}(j_1) \cap \mathcal{R}^{(\ell)}_\mathbf{v}(j_2)\big) = \lambda_{\mathbf{v}^{(\ell)}}\big(\big(\pi_{\mathbf{u},\mathbf{v}}^{(\ell)}\, \mathbf{y} + \mathcal{R}^{(\ell)}_\mathbf{v}(j_1)\big) \cap \big(\pi_{\mathbf{u},\mathbf{v}}^{(\ell)}\, \mathbf{y} + \mathcal{R}^{(\ell)}_\mathbf{v}(j_2)\big)\big)= 0.
\end{equation}
We can replace~$\ell$ by any $n \ge \ell$ since, for distinct $j_1,j_2 \in \mathcal{A}$, we have $[\mathbf{0},j_1] \in E_1^*(\sigma_{[\ell,n)})[\mathbf{0},j_1']$ and $[\mathbf{0},j_2] \in E_1^*(\sigma_{[\ell,n)})[\mathbf{0},j_2']$, where $j_1'$ and $j_2'$ are the first letters of $\sigma_{[\ell,n)}(j_1)$ and~$\sigma_{[\ell,n)}(j_2)$, respectively, thus $\lambda_{\mathbf{v}^{(n)}}(\mathcal{R}^{(n)}_\mathbf{v}(j_1) \cap \mathcal{R}^{(n)}_\mathbf{v}(j_2)) = 0$ by Proposition~\ref{p:disjoint} and~\eqref{e:j1j2disjoint}.
By Lemma~\ref{l:close2}, this implies that $\lambda_\mathbf{v}(\mathcal{R}_\mathbf{v}(j_1) \cap \mathcal{R}_\mathbf{v}(j_2)) = 0$.
\end{proof}

\begin{remark}[Negative strong coincidence]\label{rem:-}
It is sometimes convenient  (see Section \ref{sec:examples}) to use the following variant of the strong coincidence condition for suffixes: a~sequence of substitutions $\boldsymbol{\sigma} = (\sigma_n)_{n\in\mathbb{N}}\in S^{\mathbb{N}}$ satisfies the \emph{negative strong coincidence condition} if there is $\ell \in \mathbb{N}$ such that, for each pair $(j_1,j_2) \in \mathcal{A} \times \mathcal{A}$, there are $i \in \mathcal{A}$ and $s_1, s_2 \in \mathcal{A}^*$ with $\mathbf{l}(s_1) = \mathbf{l}(s_2)$ such that $i\hspace{.1em}s_1$ is a suffix of~$\sigma_{[0,\ell)}(j_1)$ and $i\hspace{.1em}s_2$ is a suffix of~$\sigma_{[0,\ell)}(j_2)$, where $v$ is a \emph{suffix} of $w \in \mathcal{A}^*$ if $w \in \mathcal{A}^* v$.

Assume that $\boldsymbol{\sigma}$ has Property PRICE. Then also negative strong coincidence allows to conclude that the sets $\mathcal{R}(i)$, $i \in \mathcal{A}$, are pairwise disjoint in measure.
Indeed, negative strong coincidence implies that
$[\mathbf{l}(j_1) - \mathbf{y}, j_1], [\mathbf{l}(j_2) - \mathbf{y}, j_2]\in E_1^*(\sigma_{[0,\ell)})[\mathbf{0},i]$, with $\mathbf{y} = (M_{[0,\ell)})^{-1}\, \mathbf{l}(i\hspace{.1em}s_1)$, thus
\[
\lambda_{\mathbf{v}^{(\ell)}} \big( \big(\pi_{\mathbf{u},\mathbf{v}}^{(\ell)}\, \mathbf{l}(j_1) + \mathcal{R}^{(\ell)}_\mathbf{v}(j_1)\big) \cap \big(\pi_{\mathbf{u},\mathbf{v}}^{(\ell)}\, \mathbf{l}(j_2) + \mathcal{R}^{(\ell)}_\mathbf{v}(j_2)\big) \big) = 0. 
\]
By the definition of~$\mathcal{R}^{(\ell)}_\mathbf{v}$ and its subtiles, we have 
\[
\bigcup_{j\in\mathcal{A}} \mathcal{R}^{(\ell)}_\mathbf{v}(j) = \mathcal{R}^{(\ell)}_\mathbf{v} = \bigcup_{j\in\mathcal{A}} \big(\pi_{\mathbf{u},\mathbf{v}}^{(\ell)}\, \mathbf{l}(j) + \mathcal{R}^{(\ell)}_\mathbf{v}(j)\big). 
\]
From disjointness in the union on the right, we get that $\lambda_\mathbf{v}^{(\ell)}\big(\mathcal{R}^{(\ell)}_\mathbf{v}\big) = \sum_{j\in\mathcal{A}} \lambda_\mathbf{v}^{(\ell)}\big(\mathcal{R}^{(\ell)}_\mathbf{v}(j)\big)$, hence, the union on the left is also disjoint in measure. The remainder of the proof is now exactly the same as in Proposition~\ref{p:strongcoincidence}.
\end{remark}

\begin{proposition} \label{p:gcc}
Let $S$ be a finite or  infinite set of unimodular substitutions over a finite alphabet and assume that $\boldsymbol{\sigma}\in S^{\mathbb{N}}$ has Property PRICE.
Then the following assertions are equivalent.
\renewcommand{\theenumi}{\roman{enumi}}
\begin{enumerate}
\item \label{i:gcc1} 
The collection $\mathcal{C}_\mathbf{w}$ forms a tiling of~$\mathbf{w}^\bot$ for some $\mathbf{w} \in \mathbb{R}_{\ge0}^d \setminus \{\mathbf{0}\}$.
\item \label{i:gcc2} 
The collection $\mathcal{C}_\mathbf{w}$ forms a tiling of~$\mathbf{w}^\bot$ for all $\mathbf{w} \in \mathbb{R}_{\ge0}^d \setminus \{\mathbf{0}\}$.
\item \label{i:gcc3}  
The sequence $\boldsymbol{\sigma}$ satisfies the geometric coincidence condition, that is, for each $R > 0$ there is $\ell \in \mathbb{N}$, such that, for all $n \ge \ell$,
\begin{equation} \label{e:gcc3}
\big\{[\mathbf{y},j] \in \Gamma(\tr{(M_{[0,n)})}\, \mathbf{1}):\, \|\mathbf{y} - \mathbf{z}_n\| \le R\big\} \subset E_1^*(\sigma_{[0,n)})[\mathbf{0},i_n]
\end{equation}
for some $i_n \in \mathcal{A}$, $\mathbf{z}_n \in (M_{[0,n)})^{-1} \mathbf{1}^\bot$.
\item \label{i:gcc4} 
There are $n \in \mathbb{N}$, $i \in \mathcal{A}$, $\mathbf{z} \in \mathbb{R}^d$, such that 
\[
\big\{[\mathbf{y},j] \in \Gamma(\tr{(M_{[0,n)})}\, \mathbf{1}):\, \|\pi_{(M_{[0,n)})^{-1}\mathbf{u},\mathbf{1}} (\mathbf{y} - \mathbf{z})\| \le C\big\} \subset E_1^*(\sigma_{[0,n)})[\mathbf{0},i],
\]
with $C \in \mathbb{N}$ chosen in a way that $\mathcal{L}_{\boldsymbol{\sigma}}^{(n)}$ is $C$-balanced.
\end{enumerate}
\end{proposition}

\begin{proof}
We show the implications (\ref{i:gcc1}) $\Leftrightarrow$ (\ref{i:gcc2}) $\Rightarrow$ (\ref{i:gcc3}) $\Rightarrow$ (\ref{i:gcc4}) $\Rightarrow$ (\ref{i:gcc1}).

\medskip \noindent
(\ref{i:gcc1}) $\Leftrightarrow$ (\ref{i:gcc2}). 
This is a special case of Proposition~\ref{p:independentmultiple}.

\medskip \noindent
 (\ref{i:gcc2}) $\Rightarrow$ (\ref{i:gcc3}). 
By the tiling property for $\mathbf{w} = \mathbf{1}$, $\mathcal{R}(i)$~contains an exclusive open ball~$\mathcal{B}(i)$ for each $i \in\mathcal{A}$.
For $[\mathbf{y},j] \in \Gamma(\tr{(M_{[0,n)})}\, \mathbf{1})$, we have thus $[\mathbf{y},j] \in E_1^*(\sigma_{[0,n)})[\mathbf{0},i]$ if $M_{[0,n)} (\pi_{\mathbf{u},\mathbf{1}}^{(n)}\, \mathbf{y} + \mathcal{R}^{(n)}_\mathbf{1}(j)) \cap \mathcal{B}(i) \neq \emptyset$.
Let $i \in \mathcal{A}$ and $\tilde{\mathbf{z}} \in \mathcal{B}(i)$.
By Proposition~\ref{p:strongconvergence} and Lemma~\ref{l:smallsubtiles}, we obtain that \eqref{e:gcc3} holds for $i_n = i$ and $\mathbf{z}_n = (M_{[0,n)})^{-1} \tilde{\mathbf{z}}$, provided that $n$ is sufficiently large.

\medskip \noindent
(\ref{i:gcc3}) $\Rightarrow$ (\ref{i:gcc4}). 
Let the sequences $(\ell_k)$ and $(n_k)$, the positive matrix~$B$, and~$C$ be as in Definition~\ref{def:star}.
Then there are constants $c_1,c_2 > 0$ such that $\|\mathbf{x}\| \le c_1 \|\pi_{\tilde{\mathbf{u}},\mathbf{1}}\, \mathbf{x}\| +c_2$ for all $\tilde{\mathbf{u}} \in \mathbb{R}_+^d$, $\mathbf{x} \in \mathbb{R}^d$ with $0 \le \langle \mathbf{x}, \mathbf{w} \rangle < \|\mathbf{w}\|$ for some $\mathbf{w} \in \tr B\, \mathbb{R}_+^d$. 

Let $k$ be such that \eqref{e:gcc3} holds for $R = c_1 C+c_2$, $n = n_k + \ell_k$ and some $i_n \in \mathcal{A}$, $\mathbf{z}_n \in (M_{[0,n)})^{-1} \mathbf{1}^\bot$.
Let $\tilde{\mathbf{u}} = (M_{[0,n_k+\ell_k)})^{-1} \mathbf{u}$, $\mathbf{w} = \tr{(M_{[0,n_k+\ell_k)})}\, \mathbf{1}$, and consider $[\mathbf{y},j] \in \Gamma(\mathbf{w})$ with $\|\pi_{\tilde{\mathbf{u}},\mathbf{1}} (\mathbf{y} - \mathbf{z}_n)\| \le C$.
Since $\mathbf{w} \in \tr B\, \mathbb{R}_+^d$, $0 \le \langle \mathbf{y}, \mathbf{w} \rangle < \|\mathbf{w}\|$, and $\langle \mathbf{z}_n, \mathbf{w}\rangle = 0$, we have $\|\mathbf{y} - \mathbf{z}_n\| \le c_1C+c_2$, thus  \eqref{e:gcc3} implies that $[\mathbf{y}, j] \in E_1^*(\sigma_{[0,n)})[\mathbf{0},i_n]$.
As $\mathcal{L}_{\boldsymbol{\sigma}}^{(n_k+\ell_k)}$ is $C$-balanced, we get (\ref{i:gcc4}) with $i = i_n$, $\mathbf{z} = \mathbf{z}_n$. 

\medskip \noindent
 (\ref{i:gcc4}) $\Rightarrow$ (\ref{i:gcc1}). 
Let $n, i, \mathbf{z}, C$ be as in~(\ref{i:gcc4}).
By Lemmas~\ref{l:bounded} and~\ref{l:e1star} and Proposition~\ref{p:setequation}, there is a neighborhood~$U$ of~$\pi_{\mathbf{u},\mathbf{1}}^{(n)}\, \mathbf{z}$ such that $M_{[0,n)}\, U$ lies in~$\mathcal{R}(i)$ and intersects no other tile of~$\mathcal{C}_\mathbf{1}$. 
By Proposition~\ref{p:independentmultiple}, this implies that~$\mathcal{C}_\mathbf{1}$ is a tiling.
\end{proof}

\begin{proposition}\label{p:gccvariant}
Let $S$ be a finite or  infinite set of unimodular substitutions over the finite alphabet $\mathcal{A}$ and assume that $\boldsymbol{\sigma}\in S^{\mathbb{N}}$ has Property PRICE.

The collection $\mathcal{C}_\mathbf{1}$ forms a tiling of~$\mathbf{1}^\bot$ if and only if $\boldsymbol{\sigma}$ satisfies the strong coincidence condition and for each $R > 0$ there exists $\ell \in \mathbb{N}$ such that $\bigcup_{i\in\mathcal{A}} E_1^*(\sigma_{[0,n)})[\mathbf{0},i]$ contains a ball of radius~$R$ of $\Gamma(\tr{(M_{[0,n)})}\, \mathbf{1})$ for all $n \ge \ell$.

If $\boldsymbol{\sigma}$ satisfies the geometric finiteness property, then $\mathbf{0}$ is an inner point of~$\mathcal{R}$ and $\mathbf{0} \not\in \pi_{\mathbf{u},\mathbf{1}}\, \mathbf{x} + \mathcal{R}(i)$ for all $[\mathbf{x},i] \in \Gamma(\mathbf{1}^\bot)$ with $\mathbf{x} \ne 0$.
\end{proposition}

\begin{proof}
Assume first that $\mathcal{C}_\mathbf{1}$ forms a tiling. Then $(\sigma_n)_{n\in\mathbb{N}}$ satisfies the geometric coincidence condition by Proposition~\ref{p:gcc}. Thus, for each $R > 0$ and sufficiently large~$n$, $E_1^*(\sigma_{[0,n)})[\mathbf{0},i_n]$ contains a ball of radius~$R$ of $\Gamma(\tr{(M_{[0,n)})}\, \mathbf{1})$ for some $i_n \in \mathcal{A}$. 
By Lemma~\ref{l:relativelydense}, there is $R > 0$ such that, for $k$ large enough, each ball of radius $R$ in $\Gamma(\tr{(M_{[0,n_k)})}\, \mathbf{1})$ contains a translate of the patch~$\mathcal{U} = \{[\mathbf{0},i]:\, i\in \mathcal{A}\}$.
Therefore, we have some $k \in \mathbb{N}$, $i \in \mathcal{A}$, and $\mathbf{x} \in \mathbb{Z}^d$ such that $\mathbf{x} + \mathcal{U} \subset E_1^*(\sigma_{[0,n_k)})[\mathbf{0},i]$. This shows that the strong coincidence condition holds.

The proof of the converse direction runs along the same lines as the corresponding part of the proof of Proposition~\ref{p:gcc}, that is, (\ref{i:gcc3}) $\Rightarrow$ (\ref{i:gcc4}) $\Rightarrow$ (\ref{i:gcc1}). We have to replace $E_1^*(\sigma_{[0,n)})[\mathbf{0},i_n]$ and $E_1^*(\sigma_{[0,n)})[\mathbf{0},i]$ by $\bigcup_{i\in\mathcal{A}} E_1^*(\sigma_{[0,n)})[\mathbf{0},i]$ and use Proposition~\ref{p:strongcoincidence}.

If $\boldsymbol{\sigma}$ satisfies the geometric finiteness property, then we obtain as in Proposition~\ref{p:gcc} (\ref{i:gcc3}) $\Rightarrow$ (\ref{i:gcc4}) that $\big\{[\mathbf{y},j] \in \Gamma(\tr{(M_{[0,n)})}\, \mathbf{1}):\, \|\pi_{(M_{[0,n)})^{-1}\mathbf{u},\mathbf{1}}\, \mathbf{y}\| \le C\big\} \subset \bigcup_{i\in\mathcal{A}} E_1^*(\sigma_{[0,n)})[\mathbf{0},i]$ for some $n \in \mathbb{N}$,
with $C$ such that $\mathcal{L}_{\boldsymbol{\sigma}}^{(n)}$ is $C$-balanced, thus $\mathbf{0} \not\in \pi_{\mathbf{u},\mathbf{1}}\, \mathbf{x} + \mathcal{R}(i)$ for all $[\mathbf{x},i] \in \Gamma(\mathbf{1})$ with $\mathbf{x} \ne 0$.
As $\mathcal{C}_\mathbf{1}$ is a covering of~$\mathbf{1}^\bot$ by Proposition~\ref{p:covering}, we get that $\mathbf{0}$ is an inner point of~$\mathcal{R}$.
\end{proof}

\begin{remark}\label{rem:-2}
Proposition~\ref{p:gccvariant} remains true with an analogous proof if strong coincidence is replaced by negative strong coincidence in its statement. 
Also, Proposition ~\ref{p:gccvariant} admits an effective version analogous to Proposition~\ref{p:gcc}~(\ref{i:gcc4}).
\end{remark}

\section{Dynamical properties of $S$-adic shifts}\label{sec:UE}

We now use the results of the previous sections to investigate the dynamics of $S$-adic shifts. At the end of this section we will have collected all the necessary preparations to finish the proofs of Theorems~\ref{t:1} and~\ref{t:3}.

\subsection{Minimality and unique ergodicity}

First we observe that \cite[Theorem~5.2]{Berthe-Delecroix} implies the following result.

\begin{lemma}\label{lem:minimal}
Let $S$ be a finite or  infinite set of unimodular substitutions over a finite alphabet and let $\boldsymbol{\sigma}\in S^{\mathbb{N}}$ be a primitive directive sequence. Then the $S$-adic shift $(X_{\boldsymbol{\sigma}}, \Sigma)$ is minimal. Thus each infinite word of $(X_{\boldsymbol{\sigma}}, \Sigma)$ is uniformly recurrent.
\end{lemma}

To gain unique ergodicity we need slightly stronger assumptions.

\begin{lemma}\label{lem:uniquelyergodic}
Let $S$ be a finite or  infinite set of unimodular substitutions over a finite alphabet and let $\boldsymbol{\sigma}\in S^{\mathbb{N}}$ be a primitive, recurrent directive sequence. Then the $S$-adic shift $(X_{\boldsymbol{\sigma}}, \Sigma)$ is uniquely ergodic.
\end{lemma}

\begin{proof}
Primitivity and recurrence of~$\boldsymbol{\sigma}$ imply that there are indices $k_1 < \ell_1 \le k_2 < \ell_2 \le \cdots$ and a positive matrix~$B$ such that $B = M_{[k_1,\ell_1)} = M_{[k_2,\ell_2)} = \cdots$. From \eqref{e:topPF} we gain therefore that $\bigcap_{n\ge k} M_{[k,n)}\, \mathbb{R}^d_+$ is one-dimensional for  each $k\in\mathbb{N}$ and, hence, \cite[Theorem~5.7]{Berthe-Delecroix} yields the result (the fact that $\boldsymbol{\sigma}$ is ``everywhere growing'' in the sense stated in that theorem is an immediate consequence of primitivity and recurrence).
\end{proof}

\subsection{Representation map}
In order to set up a representation map from~$X_{\boldsymbol{\sigma}}$ to~$\mathcal{R}$, we define refinements of the subtiles of~$\mathcal{R}$ by
\[
\mathcal{R}(w) = \overline{\{\pi_{\mathbf{u},\mathbf{1}}\, \mathbf{l}(p):\, p \in \mathcal{A}^*,\ \mbox{$p\hspace{.1em}w$ is a prefix of a limit word of $\boldsymbol{\sigma}$}\}} \quad
(w \in \mathcal{A}^*).
\]

\begin{lemma} \label{l:convrefinement}
Let $S$ be a finite or  infinite set of unimodular substitutions over a finite alphabet and let $\boldsymbol{\sigma}\in S^{\mathbb{N}}$ be a primitive, algebraically irreducible, and recurrent directive sequence with balanced  language~$\mathcal{L}_{\boldsymbol{\sigma}}$.
Then $\bigcap_{n\in\mathbb{N}} \mathcal{R}(\zeta_0 \zeta_1 \cdots \zeta_{n-1})$ is a single point in~$\mathcal{R}$ for each infinite word $\zeta_0 \zeta_1 \cdots \in X_{\boldsymbol{\sigma}}$. Therefore, the representation map
\[
\varphi:\, X_{\boldsymbol{\sigma}} \to \mathcal{R},\ \zeta_0 \zeta_1 \cdots \mapsto \bigcap_{n\in\mathbb{N}} \mathcal{R}(\zeta_0 \zeta_1 \cdots \zeta_{n-1}),
\]
is well-defined, continuous and surjective.
\end{lemma}

\begin{proof}
Let $\zeta = \zeta_0 \zeta_1 \cdots \in X_{\boldsymbol{\sigma}}$ and let  $\omega$ be a  limit word of~$\boldsymbol{\sigma}$.
Then $\mathcal{R} = \mathcal{R}(\zeta_{[0,0)}) \supset \mathcal{R}(\zeta_{[0,1)}) \supset \cdots$, and $\mathcal{R}(\zeta_{[0,n)}) \ne \emptyset$ for all $\ell\in\mathbb{N}$, where we use the abbreviation $\zeta_{[k,\ell)} = \zeta_k \zeta_{k+1} \cdots \zeta_{\ell-1}$. 
As $(X_{\boldsymbol{\sigma}}, \Sigma)$ is minimal by Lemma~\ref{lem:minimal}, we have a sequence $(n_k)_{k\in\mathbb{N}}$ such that $\zeta_{[n_k,n_k+k)} = \omega_{[0,k)}$ for all $k \in \mathbb{N}$. 
Since $\mathcal{R}(\zeta_{[0,n_k+k)}) \subset \mathcal{R}(\zeta_{[n_k,n_k+k)}) - \pi_{\mathbf{u},\mathbf{1}}\, \mathbf{l}(\zeta_{[0,n_k)})$, it only remains to show that the diameter of $\mathcal{R}(\zeta_{[n_k,n_k+k)}) = \mathcal{R}(\omega_{[0,k)})$ converges to zero. We even show that $\bigcap_{k\in\mathbb{N}} \mathcal{R}(\omega_{[0,k)}) = \{\mathbf{0}\}$. 

Let $\mathcal{S}_k = \{\pi_{\mathbf{u},\mathbf{1}}\, \mathbf{l}(\omega_{[0,n)}):\, 0 \le n \le k\}$. 
Then we clearly have $\mathcal{R}(\omega_{[0,k)}) + \mathcal{S}_k \subset \mathcal{R}$ for all $k \in \mathbb{N}$.
We also have $\lim_{k\to\infty} \mathcal{S}_k = \mathcal{R}$ (in Hausdorff metric) because, for each prefix $\tilde{p}$ of a limit word~$\tilde{\omega}$, $\pi_{\mathbf{u},\mathbf{1}}\, \mathbf{l}(\tilde{p})$ can be approximated arbitrarily well by $\pi_{\mathbf{u},\mathbf{1}}\, \mathbf{l}(p)$ with a prefix $p$ of~$\omega$, by primitivity and Proposition~\ref{p:strongconvergence}.
This implies that $\lim_{k\to\infty} \mathcal{R}(\omega_{[0,k)}) = \{\mathbf{0}\}$, which proves that $\varphi$ is well defined.

Since the sequence $(\mathcal{R}(\zeta_{[0,n)})_{n\in\mathbb{N}}$ is nested and converges to a single point, $\varphi$ is continuous.
The surjectivity follows from a Cantor diagonal argument. 
\end{proof}

\subsection{Domain exchange}
Suppose that the strong coincidence condition\footnote{All the results of this subsection remain true if strong coincidence is replaced by negative strong coincidence.} holds. 
Then, by Proposition~\ref{p:strongcoincidence}, the \emph{domain exchange} 
\begin{equation} \label{e:T}
E:\ \mathcal{R} \to \mathcal{R}, \quad 
\mathbf{x} \mapsto \mathbf{x} + \pi_{\mathbf{u},\mathbf{1}}\, \mathbf{e}_i \quad \mbox{if $\mathbf{x} \in \mathcal{R}(i) \setminus \bigcup_{j\ne i} \mathcal{R}(j)$},
\end{equation}
is well defined almost everywhere on~$\mathcal{R}$.
This map induces a dynamical system $(\mathcal{R}, E, \lambda_\mathbf{1})$.

\begin{proposition}\label{p:domainexchange}
Let $S$ be a finite or  infinite set of unimodular substitutions over a finite alphabet $\mathcal{A}$. If $\boldsymbol{\sigma}\in S^{\mathbb{N}}$ has Property PRICE and satisfies the strong coincidence condition, then the following results hold.
\renewcommand{\theenumi}{\roman{enumi}}
\begin{enumerate}
\itemsep1ex
\item \label{i:de1}
The domain exchange map $E$ is $\lambda_\mathbf{1}$-almost everywhere bijective.
\item \label{i:de2}
Each collection $\mathcal{K}_n = \{\mathcal{R}(w): w \in \mathcal{L}_{\boldsymbol{\sigma}} \cap \mathcal{A}^n\}$, $n \in \mathbb{N}$, is a measure-theoretic partition of~$\mathcal{R}$. 
\item \label{i:de3}
The representation map~$\varphi$ is $\mu$-almost everywhere bijective, where $\mu$ is the unique $\Sigma$-invariant probability measure on $(X_{\boldsymbol{\sigma}},\Sigma)$.
\item \label{i:de4}
The system $(X_{\boldsymbol{\sigma}}, \Sigma, \mu)$ is measurably conjugate to the domain exchange $(\mathcal{R}, E, \lambda_\mathbf{1})$.
More precisely, the following diagram commutes:
\[
\begin{CD}
X_{\boldsymbol{\sigma}} @> \Sigma >> X_{\boldsymbol{\sigma}} \\
@VV\varphi V @VV\varphi V\\
\mathcal{R} @> E >> \mathcal{R}
\end{CD}
\]
\end{enumerate}
\end{proposition}

\begin{proof}
All the following statements are to be understood up to measure zero.
Since $\boldsymbol{\sigma}$ satisfies the strong coincidence condition, Proposition~\ref{p:strongcoincidence} implies that the map~$E$ is a well-defined isometry on~$\mathcal{R}(i)$, with
\[
E(\mathcal{R}(i)) = \overline{\{\pi_{\mathbf{u},\mathbf{1}}\, \mathbf{l}(p\hspace{.1em}i):\, p \in \mathcal{A}^*,\ \mbox{$p\hspace{.1em}i$ is a prefix of a limit word of $\boldsymbol{\sigma}$}\}} \quad (i \in\mathcal{A}).
\]
Therefore, we have $\bigcup_{i\in\mathcal{A}} E(\mathcal{R}(i)) = \mathcal{R}$. 
Thus $E$ is a surjective piecewise isometry, hence, it is also injective, which proves Assertion~(\ref{i:de1}).
As
\begin{equation} \label{e:R0n}
\mathcal{R}(w_0 w_1 \cdots w_{n-1}) = \bigcap_{\ell=0}^{n-1} E^{-\ell} \mathcal{R}(w_\ell),
\end{equation}
Assertion~(\ref{i:de2}) is again a consequence of Proposition~\ref{p:strongcoincidence} together with the injectivity of~$E$.
Since  
\begin{equation}\label{eq:comm}
E  \circ \varphi = \varphi \circ  \Sigma  
\end{equation}
follows easily by direct calculation, the measure $\lambda_\mathbf{1} \circ \varphi$ is a shift invariant probability measure on~$X_{\boldsymbol{\sigma}}$. Thus, by unique ergodicity of $(X_{\boldsymbol{\sigma}}, \Sigma, \mu)$, we have $\mu = \lambda_\mathbf{1} \circ \varphi$.
Now, Assertion~(\ref{i:de2}) implies that $\varphi(\mathbf{x}) \ne \varphi(\mathbf{y})$ for all distinct $\mathbf{x}, \mathbf{y} $ satisfying $\varphi(\mathbf{x}),\varphi(\mathbf{y} )  \in \mathcal{R} \setminus \bigcup_{n\in\mathbb{N}, K\in\mathcal{K}_n} \partial K$.
As, by \eqref{e:R0n} and Proposition~\ref{p:boundary}, $\lambda_\mathbf{1}(\partial K)=\mu(\varphi^{-1}(\partial K)) = 0$ for all $K\in\mathcal{K}_n$, $n\in\mathbb{N}$, the map~$\varphi$ is a.e.\ injective, which, together with Lemma~\ref{l:convrefinement},  proves Assertion~(\ref{i:de3}). 
Finally, using~\eqref{eq:comm}, Assertion~(\ref{i:de4}) follows immediately from Assertion~(\ref{i:de3}).
 \end{proof}

\subsection{Group translations}
Fix some $j\in\mathcal{A}$. If $\mathcal{C}_\mathbf{1}$ forms a tiling of~$\mathbf{1}^\bot$, then $\mathcal{R}$ is a fundamental domain of the lattice $\Lambda = \mathbf{1}^\bot \cap \mathbb{Z}^d$ (which is spanned by $\mathbf{e}_j - \mathbf{e}_i$, $i \in \mathcal{A} \setminus \{j\}$). 
Since $\pi_{\mathbf{u},\mathbf{1}}\, \mathbf{e}_i \equiv \pi_{\mathbf{u},\mathbf{1}}\, \mathbf{e}_j \pmod \Lambda$ holds for each $i\in\mathcal{A}$, the canonical projection of~$E$ onto the torus $\mathbf{1}^\bot / \Lambda \simeq \mathbb{T}^{d-1}$ is equal to the translation $\mathbf{x} \mapsto \mathbf{x} + \pi_{\mathbf{u},\mathbf{1}}\, \mathbf{e}_{j}$. In general, even if the strong coincidence condition is not satisfied, the following proposition holds.

\begin{proposition}\label{p:rotate}
Let $S$ be a finite or  infinite set of unimodular substitutions over the finite alphabet $\mathcal{A}$ and let $\boldsymbol{\sigma}\in S^{\mathbb{N}}$ be a primitive, algebraically irreducible, and recurrent directive sequence with balanced language~$\mathcal{L}_{\boldsymbol{\sigma}}$.
Fix $j\in\mathcal{A}$. If $\mathcal{C}_\mathbf{1}$ forms a multiple tiling of~$\mathbf{1}^\bot$, then the translation $(\mathbf{1}^\bot/\Lambda, + \pi_{\mathbf{u},\mathbf{1}}\, \mathbf{e}_j, \overline{\lambda_\mathbf{1}})$, where $\overline{\lambda_\mathbf{1}}$ denotes the Haar measure on the torus $\mathbf{1}^\bot/\Lambda$, is a topological factor of the dynamical system $(X_{\boldsymbol{\sigma}}, \Sigma, \mu)$. If furthermore $\mathcal{C}_\mathbf{1}$ forms a tiling of~$\mathbf{1}^\bot$, then $(X_{\boldsymbol{\sigma}}, \Sigma, \mu)$ is measurably  conjugate to the translation $(\mathbf{1}^\bot/\Lambda,+ \pi_{\mathbf{u},\mathbf{1}}\, \mathbf{e}_j, \overline{\lambda_\mathbf{1}})$. More precisely, the following diagram commutes:
\[
\begin{CD}
X_{\boldsymbol{\sigma}} @> \Sigma >> X_{\boldsymbol{\sigma}} \\
@VV\overline{\varphi} V @VV\overline{\varphi} V\\
\mathbf{1}^\bot/\Lambda @> + \pi_{\mathbf{u},\mathbf{1}}\, \mathbf{e}_{j} >> \mathbf{1}^\bot/\Lambda
\end{CD}
\]
Here, $\overline{\varphi}$ is the canonical projection of the representation map~$\varphi$ onto $\mathbf{1}^\bot/\Lambda$.
\end{proposition}

\begin{proof}
If $\zeta = \zeta_0\zeta_1\cdots \in X_{\boldsymbol{\sigma}}$, then $\varphi \circ \Sigma(\zeta) = \varphi(\zeta)  + \pi_{\mathbf{u},\mathbf{1}}\, \mathbf{e}_{\zeta_0}$. Applying the canonical projection onto $\mathbf{1}^\bot / \Lambda$, this identity becomes $\overline{\varphi}\circ\Sigma(\zeta) = \overline{\varphi}(\zeta)  + \pi_{\mathbf{u},\mathbf{1}}\, \mathbf{e}_{j}$. The result now follows by noting that $\overline{\varphi}$ is $m$ to~$1$ onto, where $m$ is the covering degree of~$\mathcal{C}_1$, and, hence, a bijection if $\mathcal{C}_1$ forms a tiling.
\end{proof}

\subsection{Proof of Theorem~\ref{t:1}} \label{sec:proof3.1}
Let $S$ be a finite or  infinite set of unimodular substitutions over the finite alphabet $\mathcal{A}$ and let $\boldsymbol{\sigma}\in S^{\mathbb{N}}$. We are now in a position to finish the proof of Theorem~\ref{t:1} by collecting the results proved so far. Throughout the proof, observe that in view of Lemma~\ref{lem:th1star} the conditions of Theorem~\ref{t:1} imply that~$\boldsymbol{\sigma}$ has Property PRICE.

Concerning~(\ref{i:11}), we see that the system $(X_{\boldsymbol{\sigma}},\Sigma)$ is minimal by Lemma~\ref{lem:minimal} and uniquely ergodic by Lemma~\ref{lem:uniquelyergodic}. The unique $\Sigma$-invariant measure on~$X_{\boldsymbol{\sigma}}$ is denoted by~$\mu$. As for~(\ref{i:12}), first observe that $\mathcal{R}(i)$ is closed by definition ($i \in \mathcal{A}$). Thus compactness of~$\mathcal{R}(i)$ follows from Lemma~\ref{l:bounded}. The fact that $\lambda_\mathbf{1}(\partial \mathcal{R}(i)) = 0$ is contained in Proposition~\ref{p:boundary}. The multiple tiling property of the collection~$\mathcal{C}_\mathbf{1}$ in~(\ref{i:13}) follows from Proposition~\ref{p:independentmultiple} by taking $\mathbf{w} = \mathbf{1}$.  The   finite-to-one covering property comes from Proposition~\ref{p:rotate}, and it implies that $(X_{\boldsymbol{\sigma}},\Sigma,\mu)$ is not weakly mixing; see also \cite[Theorem~2.4]{Furstenberg}.
To prove~(\ref{i:14}), first observe that strong coincidence implies that the sets~$\mathcal{R}(i)$, $i\in \mathcal{A}$, are measurably disjoint by Proposition~\ref{p:strongcoincidence}. Thus Proposition~\ref{p:domainexchange}~(\ref{i:de4}) implies that $(X_{\boldsymbol{\sigma}},\Sigma, \mu)$ is measurably conjugate to an exchange of domains on~$\mathcal{R}$. To prove~(\ref{i:15}), we combine Propositions~\ref{p:gcc} and~\ref{p:independentmultiple}. This yields that the geometric coincidence condition is equivalent to the fact that $\mathcal{C}_\mathbf{1}$ forms a tiling. 

We now turn to the results that are valid under the assumption that $\mathcal{C}_\mathbf{1}$ forms a tiling. To prove~(\ref{i:16}), we use Proposition~\ref{p:rotate}, which implies that $(X_{\boldsymbol{\sigma}}, \Sigma, \mu)$ is measurably conjugate to a translation~$T$ on the torus~$\mathbb{T}^{d-1}$. This implies that $(X_{\boldsymbol{\sigma}}, \Sigma, \mu)$ has purely discrete measure-theoretic spectrum by classical results. Assertion~(\ref{i:17}) follows from the definition of a natural coding (see Section~\ref{sec:coding}), as the translation~$T$ was defined in terms of an exchange of domains. Finally, due to \cite[Proposition~7]{Adamczewski:03}, the $C$-balancedness of~$\mathcal{L}_{\boldsymbol{\sigma}}$ implies that~$\mathcal{R}(i)$ is a bounded remainder set for each $i \in \mathcal{A}$, which proves~(\ref{i:18}). 

\subsection{Proof of Theorem~\ref{t:3}}\label{sec:t:3proof}
Let $S$ be a finite or  infinite set of unimodular substitutions over the finite alphabet~$\mathcal{A}$, and let $(G,\tau)$ be an $S$-adic graph. Let $(E_G, \Sigma, \nu)$ be the associated edge shift equipped with an ergodic probability measure $\nu$. We assume that this shift has log-integrable cocycle~$A$ and satisfies the Pisot condition stated in Section~\ref{sec:lyap-expon-pisot}, and that there exists a cylinder of positive measure in~$E_G$ corresponding to a substitution with positive incidence matrix. For $C > 0$, let
\[
E_{G,C} = \{\boldsymbol{\gamma}\in E_G:\, \mathcal{L}_{\boldsymbol{\gamma}}\ \mbox{is $C$-balanced}\}.
\] 
We will use the following statement from \cite{Berthe-Delecroix}, see also  \cite{DelHL}.
\begin{lemma}[{\cite[Theorem~6.4]{Berthe-Delecroix}}] \label{l:DC}
Let $S$ be a finite or  infinite set of unimodular substitutions over the finite alphabet~$\mathcal{A}$ and let $(G,\tau)$ be an $S$-adic graph with associated edge shift $(E_G, \Sigma, \nu)$. We assume that this shift is ergodic, has log-integrable cocycle~$A$, and satisfies the Pisot condition, and that there exists a cylinder of positive measure in~$E_G$ corresponding to a substitution with positive incidence matrix. Then 
\[
\lim_{C\to\infty} \nu(E_{G,C}) = 1.
\] 
\end{lemma}

\begin{lemma} \label{l:Pisot}
Let $S$ be a finite or  infinite set of unimodular substitutions over the finite alphabet~$\mathcal{A}$, and let $(G,\tau)$ be an $S$-adic graph with associated edge shift $(E_G, \Sigma, \nu)$. We assume that this shift is ergodic, has log-integrable cocycle~$A$, and satisfies the Pisot condition, and that $\nu$-almost all sequences $\boldsymbol{\gamma} \in E_G$ are primitive. Then, for $\nu$-almost every sequence $\boldsymbol{\gamma} \in E_G$, for each $k \in \mathbb{N}$, $M_{[k,\ell)}$ is a Pisot irreducible matrix for all sufficiently large $\ell \in \mathbb{N}$. 
\end{lemma}

\begin{proof}
Let $k \in \mathbb{N}$ and choose $\eta$ with $\theta_2 < \eta < 0$. 
Then, for $\nu$-almost all sequences $\boldsymbol{\gamma} \in E_G$, all but the largest singular values of~$M_{[k,\ell)}$ tend to zero for $\ell \to \infty$ with order~$\mathcal{O}(e^{\ell\eta})$. Thus the image of the unit sphere by~$M_{[k,\ell)}$ is an ellipsoid~$\mathcal{E}$ with largest semi-axis close to $\mathbb{R}\, (M_{[0,k)})^{-1} \mathbf{u}$, and length of all other semi-axes tending to zero with order~$\mathcal{O}(e^{\ell\eta})$.
Let $\lambda$ be an eigenvalue of~$M_{[k,\ell)}$ with $|\lambda| \ge 1$, and let  $\mathbf{w}$ be an associated eigenvector  (which depends on~$\ell$), with $\|\mathbf{w}\| = 1$. We have to show that in this case $\lambda$ is equal to the Perron-Frobenius eigenvalue of~$M_{[k,\ell)}$ for $\ell$ large enough (to make $M_{[k,\ell)}$ a positive matrix).

If $\lambda$ is real with $|\lambda| \ge 1$, then the image $M_{[k,\ell)}\, \mathbf{w}$ can lie in~$\mathcal{E}$ only if its direction is close to that of $(M_{[0,k)})^{-1} \mathbf{u}$.
Therefore, if $\ell$ is sufficiently large, the coordinates of~$\mathbf{w}$ all have the same sign, i.e., $\lambda$~is the Perron-Frobenius eigenvalue of~$M_{[k,\ell)}$. 
This shows that $\lambda$ is the only real eigenvalue with $|\lambda| \ge 1$.

If $\lambda$ is non-real with $|\lambda| \ge 1$, then $\mathbf{w} = \mathbf{w}_1 + i \mathbf{w}_2$ for two non-zero real vectors $\mathbf{w}_1,\mathbf{w}_2$. Since $\mathbf{w}$ is determined up to multiplication by a complex number, we may assume that $\|\mathbf{w}_1\| = \|\mathbf{w}_2\| = 1$ with $\mathbf{w}_1\bot \mathbf{w}_2$. Easy calculations now yield that $\|M_{[k,\ell)}\, \mathbf{w}_1\| = \|M_{[k,\ell)}\, \mathbf{w}_2\| = |\lambda|^{\frac12} \ge 1$ with $M_{[k,\ell)}\, \mathbf{w}_1 \bot M_{[k,\ell)}\, \mathbf{w}_2$. This contradicts the fact that $M_{[k,\ell)}\mathbf{w}_1, M_{[k,\ell)}\mathbf{w}_2 \in \mathcal{E}$ for large values of~$\ell$. Thus such an eigenvalue cannot exist.

We then deduce the irreducibility of  the characteristic polynomial of~$M_{[k,\ell)}$ by noticing that these integer matrices have no zero eigenvalue by unimodularity.
\end{proof}

\begin{proof}[Proof of Theorem~\ref{t:3}]\hspace{-.5em}\footnote{This proof is different from the corresponding proof in the version of the paper published in {\it Ann.\ Inst.\ Fourier (Grenoble)}. It allows to get rid of the condition that the set of vertices of the $S$-adic graph $G$ is finite in the statement of Theorem~\ref{t:3}.}\label{proof:3.3}
Our first task is to show that Property PRICE is satisfied for almost all ${\boldsymbol{\gamma} }\in E_G$. Recall that we identify an element $\boldsymbol{\gamma} =(\gamma_n)$ of the edge shift $E_G$ of $(G,\tau)$ with the associated directive sequence $(\tau(\gamma_n))$.

By the assumptions of Theorem~\ref{t:3}, there is a cylinder $\mathcal{Z}(\delta_0,\dots,\delta_{h-1})$ with positive measure such that the incidence matrix of $\delta_{[0,h)}$ is positive. 
By Lemma~\ref{l:DC}, we can choose $C$ such that $\nu(\Sigma^{-h}(E_{G,C})) = \nu(E_{G,C}) > 1 - \nu(\mathcal{Z}(\delta_0,\dots,\delta_{h-1}))$, thus $\nu\big(\mathcal{Z}(\delta_0,\dots,\delta_{h-1}) \cap \Sigma^{-h}(E_{G,C})\big) > 0$.

By Poincar\'e's Recurrence Theorem, we have for almost all $\boldsymbol{\gamma}=(\gamma_n)\in E_G$ some $\ell_0(\boldsymbol{\gamma}) \ge h$ such that $\Sigma^{\ell_0(\boldsymbol{\gamma})-h}(\boldsymbol{\gamma}) \in \mathcal{Z}(\delta_0,\dots,\delta_{h-1}) \cap \Sigma^{-h}(E_{G,C})$, i.e., $(\gamma_0,\dots,\gamma_{\ell_0(\boldsymbol{\gamma})-1})$ ends with $(\delta_0,\dots,\delta_{h-1})$ and $\Sigma^{\ell_0(\boldsymbol{\gamma})}(\boldsymbol{\gamma}) \in E_{G,C}$. 
We extend $\ell_0(\boldsymbol{\gamma})$ for almost all $\boldsymbol{\gamma} \in E_G$ to a sequence $(\ell_k(\boldsymbol{\gamma}))_{k\in\mathbb{N}}$ such that
\begin{itemize}
\item
$(\gamma_0,\dots,\gamma_{\ell_{k+1}(\boldsymbol{\gamma})-1})$ ends with $(\gamma_0,\dots,\gamma_{\ell_k(\boldsymbol{\gamma})-1})$ (and, a fortiori, with $(\delta_0,\ldots,\delta_{h-1})$),
\item
$\Sigma^{\ell_{k+1}(\boldsymbol{\gamma})}(\boldsymbol{\gamma}) \in E_{G,C}$,
\item
$\ell_{k+1}(\boldsymbol{\gamma}) \ge 2\ell_k(\boldsymbol{\gamma})$,
\end{itemize}
for all $k \in \mathbb{N}$.
To this end, assume that $\ell_0(\boldsymbol{\gamma}),\ldots,\ell_k(\boldsymbol{\gamma})$ are already defined for almost all $\boldsymbol{\gamma}\in E_G$.
Consider the set of all $\boldsymbol{\gamma}$ having a given value $\ell_k=\ell_k(\boldsymbol{\gamma})$ and a given prefix $(\gamma_0,\dots,\gamma_{\ell_k-1})$.
Assume that this set has positive measure, which implies that $\nu\big(\mathcal{Z}(\gamma_0,\dots,\gamma_{\ell_k-1}) \cap \Sigma^{-\ell_k}(E_{G,C})\big) > 0$.
Then, for almost all $\boldsymbol{\gamma}$ in this set, we obtain (by Poincar\'e's Recurrence Theorem) some $\ell_{k+1}(\boldsymbol{\gamma})$ with the required properties. 
Applying this for all choices of $\ell_k$ and $(\gamma_0,\dots,\gamma_{\ell_k-1})$, we get some $\ell_{k+1}(\boldsymbol{\gamma})$ for almost all $\boldsymbol{\gamma} \in E_G$.
Therefore, such a sequence $(\ell_k(\boldsymbol{\gamma}))_{k\in\mathbb{N}}$ exists for almost all $\boldsymbol{\gamma} \in E_G$.

Setting $n_k (\boldsymbol{\gamma}) = \ell_{k+1}(\boldsymbol{\gamma}) - \ell_k (\boldsymbol{\gamma})$, we obtain that conditions (P), (R) and (C) of Property PRICE hold for almost all $\boldsymbol{\gamma} \in E_G$.
By Lemma~\ref{l:findv} we can replace $(n_k)$ and $(\ell_k)$ by subsequences such that condition (E) holds. These subsequences also satisfy (P), (R) and (C). 
From the Pisot condition and Lemma~\ref{l:Pisot}, we obtain that almost all $\boldsymbol{\gamma}\in E_{G}$ are algebraically irreducible, i.e., (I) holds. 

Summing up, Property PRICE holds for almost all~$\boldsymbol{\gamma}\in E_G$. 
However, in Section~\ref{sec:proof3.1} it is shown that Assertions (i) to (v) of 
Theorem~\ref{t:1} hold for each $\boldsymbol{\gamma}$ satisfying Property PRICE. 
This proves Assertion (i) of Theorem~\ref{t:3}.
According to Section~\ref{sec:proof3.1}, Assertions (vi) to (viii) of Theorem~\ref{t:1} hold for each $\boldsymbol{\gamma}$ satisfying Property PRICE and the additional assumption that the collection $\mathcal{C}_1$ associated with $\boldsymbol{\gamma}$ forms a tiling of $\mathbf{1}^\bot$. This yields Assertion (ii) of Theorem~\ref{t:3}.
\end{proof}

\section{$S$-adic shifts associated with continued fraction algorithms} \label{sec:examples}

\subsection{Arnoux-Rauzy words}
In this subsection, we prove our results on Arnoux-Rauzy words. To this matter we consider $S$-adic words with $S = \{\alpha_1,\alpha_2,\alpha_3\}$. Recall that the $\alpha_i$ are the Arnoux-Rauzy substitutions defined in~\eqref{eq:AR}. We begin by proving that the conditions of Proposition~\ref{p:gccvariant} (with negative strong coincidence, see Remarks~\ref{rem:-} and~\ref{rem:-2})  hold.

\begin{lemma}\label{lem:strongAR}
Let $\boldsymbol{\sigma} \in S^\mathbb{N}$ be a directive sequence of Arnoux-Rauzy substitutions over three letters. Then $\boldsymbol{\sigma}$ satisfies the negative strong coincidence condition. 
\end{lemma}

\begin{proof}
Just observe that for each $i\in\mathcal{A}$ the image $\alpha_i(j)$ ends with the letter $i$ for each $j\in\mathcal{A}$. 
\end{proof}

We mention that (positive) strong coincidence for sequences of Arnoux-Rauzy substitutions is (essentially) proved in~\cite[Proposition~4]{Barge-Stimac-Williams:13}. 

\begin{proposition}\label{lem:superAR}
Let $(\sigma_n)_{n\in\mathbb{N}} \in S^\mathbb{N}$\ with $S=\{\alpha_1,\alpha_2,\alpha_3\}$ be a directive sequence of Arnoux-Rauzy substitutions such that, for each $i \in \{1,2,3\}$, we have $\sigma_n = \alpha_i$ for infinitely many values of~$n$. Then the geometric finiteness property holds.
\end{proposition}

\begin{proof}
Let $(n_k)_{k\in\mathbb{N}}$ be an increasing sequence of integers such that $\{\sigma_\ell:\, n_k \le \ell < n_{k+1}\} = S$ for each $k \in \mathbb{N}$.
It is shown in the proof of \cite[Theorem~4.7]{Berthe-Jolivet-Siegel:12} that the ``combinatorial radius'' of $\bigcup_{i\in\mathcal{A}} E_1^*(\sigma_{[0,n_k)})[\mathbf{0},i]$ is at least~$k$, i.e., $\bigcup_{i\in\mathcal{A}} E_1^*(\sigma_{[0,n_k)})[\mathbf{0},i]$ contains larger and larger balls in $\Gamma(\tr{(M_{[0,n)})}\, \mathbf{1})$ around~$\mathbf{0}$. 
\end{proof}

\begin{proof}[Proof of Theorem~\ref{t:5}]
By~\cite[Theorem~1]{AD13}\footnote{Let $N_i$ be the incidence matrix of~$\alpha_i$. In~\cite{AD13}, the authors deal with products of the transposes~$\tr{N}_{i}$. However, as indicated in \eqref{eq:transposeequal}, the Lyapunov exponents do not change under transposition.}, the shift $(S^{\mathbb N},\Sigma,\nu)$ satisfies the Pisot condition. Furthermore, any product of substitutions in~$S$ that contains each of the three Arnoux-Rauzy substitutions has a positive incidence matrix. Therefore, in order to apply Theorem~\ref{t:3}, it remains to prove that the collection~$\mathcal{C}_\mathbf{1}$ forms a tiling. However, in view of  Lemma~\ref{lem:strongAR} and Proposition~\ref{lem:superAR}, this follows from Proposition~\ref{p:gccvariant}; see Remark~\ref{rem:-2}. Now all assertions of Theorem~\ref{t:5} directly follow from Theorem~\ref{t:3}.
\end{proof}

\begin{proposition}[{\cite[Theorem~7 and its proof]{Berthe-Cassaigne-Steiner}}] \label{p:1}
Let $\boldsymbol{\sigma}=(\sigma_n) \in \{\alpha_1,\alpha_2,\alpha_3\}^\mathbb{N}$. If each $\alpha_i$ occurs infinitely often in~$\boldsymbol{\sigma}$ and if we do not have $\sigma_n = \sigma_{n+1} = \cdots = \sigma_{n+h}$ for any $n \in \mathbb{N}$, then $\mathcal{L}_{\boldsymbol{\sigma}}^{(n)}$ is $(2h{+}1)$-balanced for each $n \in \mathbb{N}$.
\end{proposition}

\begin{proof}[Proof of Theorem~\ref{t:4}]
Let $\boldsymbol{\sigma}$ be as in Theorem~\ref{t:4}. As $\alpha_i$ occurs infinitely often in~$\boldsymbol{\sigma}$ for each $i \in \mathcal{A}$, \cite[Lemma~13]{Arnoux-Ito:01} implies that for each~$k$ and each sufficiently large $\ell>k$ the matrix~$M_{[k,\ell)}$ has a characteristic polynomial that is the minimal polynomial of a cubic Pisot unit and, hence, irreducible. Thus $\boldsymbol{\sigma}$ is  algebraically irreducible. The primitivity of~$\boldsymbol{\sigma}$ follows from the same fact, as any product~$M_{[k,\ell)}$ containing the incidence matrix of each of the three Arnoux-Rauzy substitutions is positive. Since $\boldsymbol{\sigma}$ is recurrent by assumption, Proposition~\ref{p:1} implies that there is $C>0$ such that for each~$n$ there is~$\ell$ such that $(\sigma_0,\ldots,\sigma_{\ell-1}) = (\sigma_n,\ldots,\sigma_{n+\ell-1})$ and $\mathcal{L}_{\boldsymbol{\sigma}}^{(n+\ell)}$ is $C$-balanced. As in the proof of Theorem~\ref{t:5}, in view of Lemma~\ref{lem:strongAR} and Proposition~\ref{lem:superAR}, it follows from Proposition~\ref{p:gccvariant} that $\mathcal{C}_\mathbf{1}$ induces a tiling. Thus all the assertions of Theorem~\ref{t:1} hold for~$\boldsymbol{\sigma}$, and the proof is finished.
\end{proof}

\begin{proposition}\label{prop:LR}
An Arnoux-Rauzy word is linearly recurrent if and only if it has bounded strong partial quotients, that is, each substitution of~$S$ occurs in its directive sequence with bounded gaps.
\end{proposition}

\begin{proof}
It is easy to check that strong partial quotients have to be bounded for an Arnoux-Rauzy word~$\omega$ to be linearly recurrent; see also \cite{RisleyZamboni}.\footnote{This characterization is already given in \cite[Corollary 3.9]{RisleyZamboni} but it relies on \cite{Durand:00a} and it needs the extra argument of \cite[Lemma 3.1]{Durand:00b}.}
The converse is a direct consequence of \cite[Lemma 3.1]{Durand:00b} by noticing that  the largest difference between two consecutive occurrences of a word of length~$2$ in~$\omega^{(n)}$ is bounded (with respect to~$n$).
\end{proof}

\begin{proof}[Proof of Corollary~\ref{cor:AR}]
This is a direct consequence of Proposition~\ref{prop:LR} together with Theorem~\ref{t:4}.
\end{proof}

\subsection{Brun words}
In this subsection, we prove our results on $S$-adic words defined in terms of the Brun substitutions $\beta_1,\beta_2,\beta_3$ defined in~\eqref{eq:brun}. Consider $S$-adic words, where $S = \{\beta_1,\beta_2,\beta_3\}$. Again we begin by proving that the conditions of Proposition~\ref{p:gccvariant} hold for negative strong coincidences (see Remarks~\ref{rem:-} and \ref{rem:-2}).

\begin{lemma}\label{lem:strongB}
Let $S = \{\beta_1,\beta_2,\beta_3\}$. If $\boldsymbol{\sigma} \in S^\mathbb{N}$ contains $\beta_3$, then it has negative strong coincidences. 
\end{lemma}

\begin{proof}
This follows from the fact that $\beta_3\beta_i(j)$ ends with the letter~$3$ for all $i, j\in\mathcal{A}$.
\end{proof}

Next we use a result from \cite{BBJS14}, where a slightly different set of Brun substitutions is considered, namely
\[
\sigma_1^\mathrm{Br} : \begin{cases} 1 \mapsto 1  \\ 2 \mapsto 2 \\ 3 \mapsto 32 \end{cases} \quad
\sigma_2^\mathrm{Br} : \begin{cases} 1 \mapsto 1  \\ 2 \mapsto 3 \\ 3 \mapsto 23 \end{cases} \quad
\sigma_3^\mathrm{Br} : \begin{cases} 1 \mapsto 2  \\ 2 \mapsto 3 \\ 3 \mapsto 13 \end{cases}
\] 
Note that the incidence matrix of~$\sigma_i^\mathrm{Br}$ is the transpose of that of~$\beta_i$. 
We have the following relation between products of substitutions from the two sets.

\begin{lemma} \label{l:relBrun}
Let $i_0, i_1, \ldots, i_n \in \{1,2,3\}$, $n \in \mathbb{N}$.  
Then
\[
\beta_{i_0} \beta_{i_1} \cdots \beta_{i_n} = \begin{cases}\sigma_2^\mathrm{Br} \sigma_{i_0}^\mathrm{Br} \sigma_{i_1}^\mathrm{Br} \cdots \sigma_{i_{n-1}}^\mathrm{Br} \pi_{(23)} & \mbox{if}\ i_n = 1, \\ \sigma_2^\mathrm{Br} \sigma_{i_0}^\mathrm{Br} \sigma_{i_1}^\mathrm{Br} \cdots \sigma_{i_{n-1}}^\mathrm{Br} & \mbox{if}\ i_n = 2, \\ \sigma_2^\mathrm{Br} \sigma_{i_0}^\mathrm{Br} \sigma_{i_1}^\mathrm{Br} \cdots \sigma_{i_{n-1}}^\mathrm{Br} \pi_{(12)} & \mbox{if}\ i_n = 3,\end{cases}
\]
where $\pi_{(ij)}$ denotes the cyclic permutation that exchanges the letters $i$ and~$j$.
\end{lemma}

\begin{proof}
We have $\beta_1 = \sigma_2^\mathrm{Br}\, \pi_{(23)}$, $\beta_2 = \sigma_2^\mathrm{Br}$, $\beta_3 = \sigma_2^\mathrm{Br}\, \pi_{(12)}$, and $\pi_{(23)}\, \sigma_2^\mathrm{Br} = \sigma_1^\mathrm{Br}$, $\pi_{(12)}\, \sigma_2^\mathrm{Br} = \sigma_3^\mathrm{Br}$.
\end{proof}

\begin{proposition}\label{lem:superB}
Let $(\sigma_n)_{n\in\mathbb{N}} \in S^\mathbb{N}$\ with $S = \{\beta_1,\beta_2,\beta_3\}$ be a directive sequence of Brun substitutions with infinitely many occurrences of~$\beta_3$.
Then, for each $R > 0$, $\bigcup_{i\in\mathcal{A}} E_1^*(\sigma_{[0,n)})[\mathbf{0},i]$ contains a ball of radius~$R$ of $\Gamma(\tr{(M_{[0,n)})}\, \mathbf{1})$ for all sufficiently large $n \in \mathbb{N}$.
\end{proposition}

\begin{proof}
This follows by Lemma~\ref{l:relBrun} from \cite[Theorem~5.4~(1)]{BBJS14} together with Lemma  \ref{lem:strongB}.
\end{proof}

\begin{proof}[Proof of Theorem~\ref{t:6}]
By~\cite[Theorem~1]{AD13}\footnote{Again, in \cite{AD13} the authors deal with products of the transposes of the incidence matrices of  the substitutions.} (see also  \cite{FUKE96,Meester,Schratzberger:98,Broise}), the shift $(S^{\mathbb N},\Sigma,\nu)$ satisfies the Pisot condition.
Moreover, it is easy to see that the product $\beta_3\beta_2\beta_3\beta_2$ has positive incidence matrix. Thus, in order to apply Theorem~\ref{t:3}, we need to prove that the collection~$\mathcal{C}_\mathbf{1}$ forms a tiling. Using Lemma~\ref{lem:strongB} and Proposition~\ref{lem:superB}, this follows for $\nu$-almost every $\boldsymbol{\sigma} \in S^\mathbb{N}$ from Proposition~\ref{p:gccvariant} (see Remark~\ref{rem:-2}).
Now, all assertions of Theorem~\ref{t:6} follow directly from Theorem~\ref{t:3}.
\end{proof}

\begin{proof}[Proof of Theorem~\ref{t:7}]
In view of Proposition~\ref{p:rotate}, Theorem~\ref{t:6} states that almost all $\boldsymbol{\sigma} \in S^\mathbb{N}$ (w.r.t.\ any ergodic shift invariant probability measure~$\nu$ that assigns positive measure to each cylinder) give rise to an $S$-adic shift $(X_{\boldsymbol{\sigma}}, \Sigma)$ that is measurably conjugate to the translation 
\[
\pi_{\mathbf{u},\mathbf{1}}(\mathbf{e}_3) = u_1 (\mathbf{e}_3-\mathbf{e}_1) + u_2 (\mathbf{e}_3-\mathbf{e}_2)
\]
on the torus $\mathbf{1}^\bot/(\mathbb{Z}(\mathbf{e}_3-\mathbf{e}_1) + \mathbb{Z}(\mathbf{e}_3-\mathbf{e}_2))$. Here, $(u_1,u_2,u_3)$ is the frequency vector of a word in~$X_{\boldsymbol{\sigma}}$. Of course, this translation is conjugate to the translation $(u_1,u_2)$ on the standard torus~$\mathbb{T}^2$.
Note that the vector $(x_1,x_2) \in \Delta_2$ corresponds to $(u_1,u_2,u_3) = \big(\frac{x_1}{1+x_1+x_2},\frac{x_2}{1+x_1+x_2},\frac{1}{1+x_1+x_2}\big)$ in the linear version of Brun's algorithm.

Recall the definition of the conjugacy map~$\Phi$ in~\eqref{CDPhi}.
According to \cite[Th\'eor\`eme]{ArnouxNogueira93} (see also \cite[Section~3.1]{Schweiger:91}), the invariant probability measure~$m$ of the map $T_\mathrm{Brun}$ defined in~\eqref{eq:brunmap} has density $h(x_1,x_2) = \frac{12}{\pi^2x_1(1+x_2)}$ and is therefore equivalent to the Lebesgue measure. 
We now define the measure $\nu = m \Phi^{-1}$ on~$S^\mathbb{N}$. It is an ergodic shift invariant probability measure on~$S^\mathbb{N}$.  By~\eqref{CDPhi}, the mapping~$T_\mathrm{Brun}$ is measurably conjugate to the shift $(S^\mathbb{N}, \Sigma,\nu)$ via~$\Phi$.   
Moreover, $\nu(C)$ is positive for each cylinder $C\subset S^\mathbb{N}$, since each cylinder in~$\Delta_2$ has also positive Lebesgue measure and, hence, positive measure~$m$ (it has non-vanishing  Jacobian, see e.g.~\cite{SCHWEIGER}).
 
Let now $Y \subset \Delta_2$ be a set with the property that for each $(x_1,x_2) \in Y$ the $S$-adic shift $X_{\Phi(x_1,x_2)}$ is \emph{not} measurably conjugate to the translation $(u_1,u_2)$ on~$\mathbb{T}^2$. 
Theorem~\ref{t:6} (together with Proposition~\ref{p:rotate}) implies that $\nu \Phi(Y) = m(Y) = 0$. As $m$ is equivalent to the Lebesgue measure, this proves the result.
\end{proof}

\begin{proof}[Proof of Corollary~\ref{cor:8}]
We can prove similarly as in the proof of Theorem~\ref{t:7}, by choosing $j=1$ and $j=2$, respectively in Proposition~\ref{p:rotate} that, for almost all $(x_1,x_2) \in \Delta_2$, the $S$-adic shift $(X_{\boldsymbol{\sigma}},\Sigma)$ with $\boldsymbol{\sigma} = \Phi(x_1,x_2)$ is measurably conjugate to the translation by~$\mathbf{t}$ on the torus~$\mathbb{T}^2$, for each
\begin{equation}\label{eq:tvals}
\mathbf{t}\in \Big\{ \Big(\frac{x_1}{1+x_1+x_2},\frac{x_2}{1+x_1+x_2}\Big),\Big(\frac{x_1}{1+x_1+x_2},\frac{1}{1+x_1+x_2}\Big),\Big(\frac{x_2}{1+x_1+x_2},\frac{1}{1+x_1+x_2}\Big) \Big\} .
\end{equation}

It is easy to see that the set of all $\mathbf{t} \in \mathbb{R}^2$ satisfying~\eqref{eq:tvals} for some pair $(x_1,x_2) \in \Delta_2$ is equal to $\{\mathbf{t} = (t_1,t_2):\, 0\le t_2\le 1,\, t_2 \le t_1 \le 1-t_2\}$. Since the translations $(t_1,t_2)$, $(t_2,t_1)$, $(1-t_1,1-t_2)$, and $(1-t_2,1-t_1)$ on $\mathbb{T}^2$ are pairwise (measurably) conjugate, this implies the result.
\end{proof}

\bigskip

\noindent
{\bf Acknowledgements.} We are grateful to Pascal Hubert  and Sasha Skripchenko for valuable discussions on the subject of this paper. We also thank the anonymous referee for his suggestions.

\bibliographystyle{amsalpha}
\bibliography{sadic}
\end{document}